\documentclass{article}
\usepackage[utf8]{inputenc}
\usepackage{amsmath}
\usepackage{amssymb}
\usepackage{graphicx}
\usepackage{csvsimple}
\usepackage{amsthm}
\usepackage{tikz}
\usepackage{comment}
\usepackage{todonotes}
\usepackage{yfonts}
\usepackage{enumitem}
\usepackage{tikz}
\usepackage{url}
\usepackage{hyperref}

\usepackage{commath}

\newcommand{\pizero}{c_0}
\newcommand{\pione}{c_1}

\title{PDMP Monte Carlo methods for piecewise-smooth densities}
\author{Augustin Chevallier, Sam Power, Andi Wang, Paul Fearnhead}
\date{November 2021}

\usepackage[numbers]{natbib}
\usepackage{graphicx}

\newtheorem{proposition}{Proposition}
\newtheorem{lemma}{Lemma}
\newtheorem{theorem}{Theorem}
\newtheorem{remark}{Remark}
\newtheorem{assumption}{Assumption}

\newcommand{\Dom}{\mathcal D}
\newcommand{\R}{\mathbb R}
\renewcommand{\P}{\mathbb P}

\newcommand{\BPS}{\mathrm{BPS}}
\newcommand{\CS}{\mathrm{CS}}
\newcommand{\ZZ}{\mathrm{ZZ}}

\begin{document}

\maketitle 
{\bf Abstract:}
There has been substantial interest in developing Markov chain Monte Carlo algorithms based on piecewise-deterministic Markov processes. However existing algorithms can only be used if the target distribution of interest is differentiable everywhere. The key to adapting these algorithms so that they can sample from to densities with discontinuities is defining appropriate dynamics for the process when it hits a discontinuity. We present a simple condition for the transition of the process at a discontinuity which can be used to extend any existing sampler for smooth densities, and give specific choices for this transition which work with popular algorithms such as the Bouncy Particle Sampler, the Coordinate Sampler and the Zig-Zag Process. Our theoretical results extend and make rigorous arguments that have been presented previously, for instance constructing samplers for continuous densities restricted to a bounded domain, and we present a version of the Zig-Zag Process that can work in such a scenario. Our novel approach to deriving the invariant distribution of a piecewise-deterministic Markov process with boundaries may be of independent interest.

{\bf Keywords:} Bayesian inference; Bouncy particle sampler; Hamiltonian Monte Carlo; Markov chain Monte Carlo; Non-reversible samplers; Zig-Zag Process.

\section{Introduction}

In recent years there has been substantial interest in using continuous-time \textit{piecewise-deterministic Markov processes} (PDMPs), as the basis for Markov chain Monte Carlo (MCMC) algorithms. These ideas started in the statistical physics literature \cite[]{peters2012rejection}, and have led to a number of new MCMC algorithms such as the Bouncy Particle Sampler (BPS) \cite[]{Bouchard-Cote2018}, the Zig-Zag (ZZ) Process \cite[]{Bierkens2019} and the Coordinate Sampler (CS) \cite[]{wu2020coordinate}, amongst many others. See \cite{fearnhead2018piecewise} and \cite{vanetti2017piecewise} for an introduction to the area. One potential benefit which is associated with these samplers are that they are non-reversible, and it is known that non-reversible samplers can mix faster than their reversible counterparts \cite[]{diaconis2000analysis,bierkens2016non}.

Informally, a PDMP process evolves according to a deterministic flow -- defined via an ordinary differential equation -- for a random amount of time, before exhibiting an instantaneous transition, and then following a (possible different) deterministic flow for another random amount of time, and so on.

Initial PDMP samplers were defined to sample target distributions which were continuously differentiable ($C^1$) on $\mathbb{R}^d$, but there is interest in extending them to more general target distributions. To date, this has been achieved for sampling from distributions defined on the union of spaces of different dimensions \cite[]{chevallier2020reversible,Bierkens2021} and to sample from distributions on restricted domains \cite[]{BIERKENS2018148} and phylogenetic trees \cite{koskela2020zigzag}. Here we consider a further extension to sampling from target distributions on $\R^d$ which are piecewise-$C^1$. That is, they can be defined by partitioning $\mathbb{R}^d$ into a countable number of regions, with the target density $C^1$ on each region. We call such densities {\em piecewise-smooth}. Such target distributions arise in a range of statistical problems, such as latent threshold models \cite[]{nakajima2013bayesian}, binary classification \cite[]{Nishimura_2020} and changepoint models \cite[]{raftery1986bayesian}. The importance of this extension of PDMP samplers is also indicated by the usefulness of extensions of Hamiltonian Monte Carlo (HMC) to similar problems \cite{Pakman2014,NIPS2015_8303a79b,dinh2017probabilistic,zhou2020mixed}.

The challenge with extending PDMP samplers to piecewise-smooth densities is the need to specify the appropriate dynamics when the sampler hits a discontinuity in the density. Essentially, we need to specify the dynamics so that the PDMP has the distribution from which we wish to sample as its invariant distribution. Current samplers are justified based on considering the infinitesimal generator of the PDMP. Informally, the generator is an operator that acts on functions and describes how the expectation of that function of the state of the PDMP changes over time. The idea is that if we average the generator applied to a function of the current state of the PDMP and this is zero for a large set of functions, then the distribution that we average over must be the invariant distribution. Whilst intuitively this makes sense, many papers use this intuitive reasoning without giving a formal proof that the distribution they average over is in fact the invariant distribution \cite[see e.g.][]{vanetti2017piecewise,fearnhead2018piecewise}, though see \cite{Durmus2021} for an exception. Furthermore, once we introduce discontinuities, then this complicates the definition of the generator. The impact of these discontinuities is realised in terms of the set of functions for which the generator is defined, and this necessitates the use of additional arguments which take account of the impact of the discontinuity when applying arguments based on integration by parts.

More specifically, we can see the challenge with dealing with discontinuities and some of the contributions of this paper by comparing with the related work of \cite{BIERKENS2018148}, who consider designing PDMP samplers when the target distribution is only compactly supported on $\R^d$ --  a special case of our present work. They give conditions on the dynamics of a PDMP at the discontinuity defined by the boundary of the support of the target distibrution that ensure the expected value of the generator applied to suitable functions is zero. However they point out that they do not formally prove that the resulting PDMP has the correct invariant distribution. Furthermore, as we show below, they have an additional and unnecessary condition on the dynamics at the discontinuity. The most natural dynamics for the Zig-Zag Process at a discontinuity satisfies our condition, but not this extra condition required by the argument in \cite{BIERKENS2018148}.
We also note at the outset the parallel and independent contribution in \cite{koskela2020zigzag}; see Theorem~1 therein. We will discuss the connections further in our concluding discussion.

The paper is structured as follows. In the next section we give a brief introduction to PDMPs and some common PDMP samplers. Then in Section \ref{sec:invariant_general} we give general conditions for the invariant distribution of a PDMP. The result in this section formalises the informal argument used by previous authors. We then develop these results for the specific cases where we the PDMP has active boundaries -- for example, due to a compact support, or when we wish to sample from a mixture of densities defined on spaces of differing dimensions. The results in this section can be used to formalise the arguments for the algorithm of \cite[]{BIERKENS2018148} for sampling on a bounded domain, and have been used to justify the reversible jump PDMP algorithm of \cite[]{chevallier2020reversible}. In Section \ref{sec:ctsbyparts} we use our results to provide a simple sufficient condition on the dynamics for a PDMP to admit a specific piecewise-smooth density as its invariant density. Various proofs and technical assumptions for these results are deferred to the appendices. We then show how to construct dynamics at the discontinuities for a range of common PDMP samplers, and empirically compare these samplers on some toy examples -- with a view to gaining intuition as to their relative merits, particularly when we wish to sample from high-dimensional piecewise-smooth densities. The paper ends with a discussion.

\section{PDMP basic properties}

\subsection{General PDMP construction}
\label{sec:Davis-construction}
For this work, we require a general construction of \textit{piecewise-deterministic Markov processes} (PDMPs) in spaces featuring boundaries. We will follow the construction of Davis in \cite[p57]{Davis1993}, and largely make use of the notation therein.

Let $K$ be a countable set, and for $k \in K$, let $E_k^0$ be an open subset of $\mathbb{R}^{d_k}$.
Let $E^0$ be their disjoint union:
\begin{equation*}
    E^0 := \left\{ (k,z) : k\in K, z \in E_k^0\right \}.
\end{equation*}
For any $k \in K$, we have a Lipschitz vector field on $E^0_k$ that induces a flow $\Phi_k(t,z)$.

In this setting, trajectories may reach the boundaries of the state. Hence we define the entrance and exit boundaries using the flow; $\partial^- E_k^0$ and $\partial^+ E_k^0$ respectively (see Fig.\ref{fig:exit-entrance-boundary}):
\[
    \partial^\pm E_k^0 = \{z \in \partial E_k^0 |  z = \Phi_k(\pm t,\xi) \text{ for some } \xi \in E_k^0 \text{ and } t>0 \}.
 \]
Note it may be possible for a point to be both an entrance and exit boundary.

\begin{figure}
    \centering
    \tikzset{every picture/.style={line width=0.75pt}} 
    
    \begin{tikzpicture}[x=0.75pt,y=0.75pt,yscale=-0.7,xscale=0.7]
    
    \draw   (100,142) .. controls (274.2,-91.7) and (448.4,258.85) .. (622.6,25.15) ;
    \draw    (163.6,80.15) .. controls (137.86,150.44) and (215.03,171.72) .. (192.32,217.75) ;
    \draw [shift={(191.6,219.15)}, rotate = 298.01] [color={rgb, 255:red, 0; green, 0; blue, 0 }  ][line width=0.75]    (10.93,-3.29) .. controls (6.95,-1.4) and (3.31,-0.3) .. (0,0) .. controls (3.31,0.3) and (6.95,1.4) .. (10.93,3.29)   ;
    \draw [shift={(163.6,80.15)}, rotate = 110.11] [color={rgb, 255:red, 0; green, 0; blue, 0 }  ][fill={rgb, 255:red, 0; green, 0; blue, 0 }  ][line width=0.75]      (0, 0) circle [x radius= 3.35, y radius= 3.35]   ;
    \draw    (435.6,216.15) .. controls (475.2,186.45) and (427.57,143.03) .. (459.6,109.17) ;
    \draw [shift={(460.6,108.15)}, rotate = 495] [color={rgb, 255:red, 0; green, 0; blue, 0 }  ][line width=0.75]    (10.93,-3.29) .. controls (6.95,-1.4) and (3.31,-0.3) .. (0,0) .. controls (3.31,0.3) and (6.95,1.4) .. (10.93,3.29)   ;
    \draw [shift={(435.6,216.15)}, rotate = 323.13] [color={rgb, 255:red, 0; green, 0; blue, 0 }  ][fill={rgb, 255:red, 0; green, 0; blue, 0 }  ][line width=0.75]      (0, 0) circle [x radius= 3.35, y radius= 3.35]   ;
    
    \draw (293,140.4) node [anchor=north west][inner sep=0.75pt]    {$E_{0}$};
    \draw (255,5.4) node [anchor=north west][inner sep=0.75pt]    {$E_{0}^{c}$};
    \draw (170,80.4) node [anchor=north west][inner sep=0.75pt]    {$x_{0} \in \partial E_{0}^{+}$};
    \draw (193.6,222.55) node [anchor=north west][inner sep=0.75pt]    {$x_{t} =\Phi ( t,x_{0})$};
    \draw (458,118.4) node [anchor=north west][inner sep=0.75pt]    {$x'_{t} =\Phi ( t',x_{0} ') \in \partial E_{0}^{-}$};
    \draw (425,222.4) node [anchor=north west][inner sep=0.75pt]    {$x_{0} '$};
    \draw (326,52.4) node [anchor=north west][inner sep=0.75pt]    {$\partial E_{0}$};

    \end{tikzpicture}

    \caption{Exit ($\partial E_0^-$) and Entrance boundary ($\partial E_0^+$): $x_0$ is in the entrance boundary $\partial E_0^-$, while $x_t'$ is in the exit boundary $\partial E_0^+$. The arrows represent the flow $\Phi(\cdot,\cdot)$.}
    \label{fig:exit-entrance-boundary}
\end{figure}
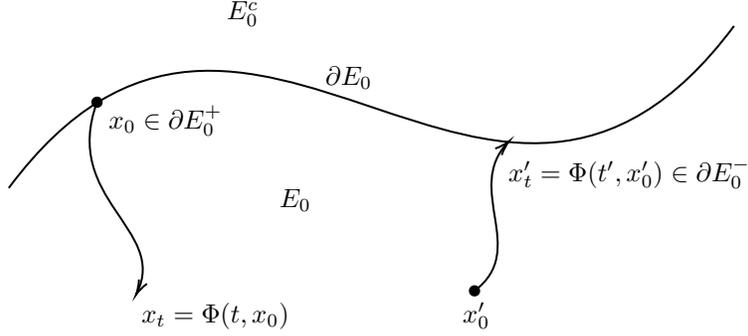

We then have $\partial_1 E_k^0 := \partial^- E_k^0\backslash \partial^+ E_k^0$ as in \cite[p57]{Davis1993}, and also 
\begin{equation*}
    E_k := E_k^0 \cup \partial_1 E_k^0.
\end{equation*}
Finally, the full state space is the disjoint union,
\begin{equation*}
    E := \bigcup_k (\{k\} \times E_k).
\end{equation*}
The active boundary (that is, the exit boundary) is then defined as
\begin{equation*}
    \Gamma := \bigcup_k (\{k\} \times \partial^+ E^0_k).
\end{equation*}
These are points on the boundary that the deterministic flow can hit.

We will denote the state of a PDMP on $E$ at time $t$ by $Z_t \in E$. A detailed construction of the PDMP is provided by Davis \cite[p59]{Davis1993}, but we provide here a summary of the quantities that defines a PDMP $(Z_t)$:
\begin{enumerate}[label=(\roman*)]
    \item An event rate $\lambda(z)$, with $z \in E$. An event occurs in $[t,t+h]$ with probability $\lambda(Z_t) h + o(h)$.
    \item A jump kernel defined for $z \in E \cup \Gamma$: $Q(\cdot|z)$ with $Q(\cdot|z)$ a probability measure on $E$. At each event time $T_i$, the state will change according to the jump kernel: $Z_{T_i} \sim Q(\cdot|Z_{T_i-})$.
    \item The deterministic flow $\Phi$ which determines the behavior of $Z_t$ between jumps. 
    \item For any trajectory $Z_t$ such that 
    \[
        \lim_{t\uparrow t_0} Z_t = Z_{t_0-} \in \Gamma,
    \]
    the state will change according to the jump kernel: $Z_{t_0} \sim Q(\cdot|Z_{t_0-})$.
\end{enumerate}

\begin{remark}
The trajectory never enters $\Gamma$, which is not in the domain.
\end{remark}

\subsection{PDMP samplers}
\label{subsec:PDMP_samplers}
In the case of most PDMP samplers, the state space is constructed by augmenting the space of interest with an auxiliary velocity space $\mathcal{V}_k$: $E_k^0 = U_k \times \mathcal{V}_k$. The deterministic flow is then typically given by free transport, i.e. $\Phi_k(t,x,v) = (x+tv,v)$, though other examples exist \cite[]{vanetti2017piecewise,terenin2018piecewise,bierkens2020boomerang}. The use of such simple dynamics allows for the exact simulation of the process dynamics, without resorting to numerical discretisation.

For the purposes of this work, it will be useful to introduce three of the more popular classes of PDMP sampler, which we will then be able to refer back to as running examples. Each of these processes work on a velocity-augmented state space and follow free-transport dynamics, and so they differ primarily in (i) the set of velocities which they use, and (ii) the nature of the jumps in the process. We describe the dynamics for each process if we wish to sample from a density $\pi(x)$ on $\mathbb R^d$.

\begin{enumerate}
    \item The \textit{Bouncy Particle Sampler} \cite[]{Bouchard-Cote2018} uses a spherically-symmetric velocity space, given by either $\mathbb{R}^d$ equipped with the standard Gaussian measure, or the unit sphere equipped with the uniform surface measure. `Bounce' events occur at rate $\lambda(x, v) = \langle v, -\nabla \log \pi  (x) \rangle_+$, and at such events, the velocity deterministically jumps to $v' = \left( I - 2 \frac{\left( \nabla \log \pi (x) \right) \left( \nabla \log \pi (x) \right)^\top}{\left( \nabla \log \pi (x) \right)^\top \left( \nabla \log \pi (x) \right)} \right) v$, i.e. a specular reflection against the level set of $\log \pi$ at $x$.
    \item The \textit{Zig-Zag Process} \cite{Bierkens2019} uses $\{ \pm 1 \}^d$ as its velocity space, equipped with the uniform measure. There are now $d$ different types of bounce events, corresponding to each coordinate of the velocity vector. Bounces of type $i$ occur at rate $\lambda_i (x, v) = \left( -v_i \partial_i \log \pi(x) \right)_+$, and at such events, $v_i$ is deterministically replaced by $-v_i$.
    \item The \textit{Coordinate Sampler} \cite{wu2020coordinate} uses $\{ \pm e_i \}_{i = 1}^d$ as its velocity space, equipped with the uniform measure, where $e_i$ is the $i^\text{th}$ coordinate vector. Bounce events again happen at rate $\lambda(x, v) = \langle v, -\nabla \log \pi (x) \rangle_+$. At such events, the velocity is resampled from the full velocity space, with probability proportional to $\lambda(x, -v')$.
\end{enumerate}
When $d = 1$, all of these processes are identical. Additionally, all three processes can be supplemented with `refreshment' events, which occur at a rate independent of $v$, and modify the velocity in a way which leaves its law invariant. This can include either full resampling, autoregressive resampling in the case of spherical velocities, coordinate-wise resampling in the case of velocity laws with independent coordinates, and other variations.

From the above definitions, it is easy to see that the event rates only make sense when $\pi$ is sufficiently smooth, and that in the presence of discontinuities, complications in defining the process will arise. Furthermore, it is not \textit{a priori} clear which processes will work best in the presence of such discontinuities.

\subsection{Review: extended generator and semigroup}
We collect some basic definitions and facts which will be crucial for our later results.

Let $\mathcal B(E)$ denote the set of bounded measurable functions $E\to \R$. For any $f \in \mathcal B(E)$, we recall the definition of the semigroup $P_t$ associated to the process $Z_t$:
\[
    P_t f(z) = \mathbb E_z[f(Z_t)], \qquad z\in E,
\]
where $\mathbb E_z$ is the expectation with respect to $\P_z$, with $\P_z$ the probability such that $\P_z[Z_0 = z] = 1$.
\begin{proposition}
    The semigroup $P_t$ is a contraction for the sup norm:
    \[
        \|P_t f\|_\infty \leq \|f\|_\infty,
    \]
    for all $f\in \mathcal{B}(E)$ and $t \geq 0$.
\end{proposition}
\begin{proof}
    See \cite[p28]{Davis1993}.
\end{proof}
The semigroup is said to be \textit{strongly continuous for} $f\in \mathcal B(E)$ if $\lim_{t\downarrow 0} \|P_t f-f\|_\infty =0$, and let $\mathcal B_0$ be the set of functions for which $P_t$ is strongly continuous:
\begin{equation*}
    \mathcal B_0:= \left\{f \in \mathcal B(E): \lim_{t\downarrow 0} \|P_t f-f\|_\infty =0 \right\}.
\end{equation*}
\begin{lemma}
    We have that $\mathcal B_0 \subset \mathcal B(E)$ is a Banach space with sup norm $\|\cdot\|_\infty$, and $P_t$ maps $\mathcal B_0\to \mathcal B_0$ for any $t \ge 0$.
\end{lemma}
\begin{proof}
    See \cite[p29]{Davis1993}.
\end{proof}

Let us write $(A, \mathcal D(A))$ for the infinitesimal generator (also referred to as the \textit{strong generator}) of the semigroup $(P_t)$. By definition, for all $f\in \mathcal D(A)$,
\begin{equation*}
    Af = \lim_{t\rightarrow 0} \frac{1}{t}( P_t f - f),
\end{equation*}
with this limit being taken in $\|\cdot\|_\infty$, with
\begin{equation}
    \mathcal D(A) = \left\{ f\in \mathcal{B}_0: \left\|\frac{1}{t}( P_t f - f)-g \right \|_\infty \to 0, \text{ for some }g\in \mathcal B(E) \right\}.
    \label{eq:D(A)_strong}
\end{equation}
Since $g$ in \eqref{eq:D(A)_strong} is a limit of functions in a Banach space, if such a $g$ exists, it must be unique, and $Af$ is well-defined. 

\begin{lemma}
    Let $f\in \Dom(A)$. Then $Af \in \mathcal{B}_0$.
    In other words, $A: \Dom(A) \to \mathcal{B}_0$.
\end{lemma}
\begin{proof}
    This is immediate since $P_t$ maps $\mathcal{B}_0\to \mathcal{B}_0$.
\end{proof}

We now define the \textit{extended generator} $(\textswab{A}, \mathcal D(\textswab{A}))$: $\mathcal D(\textswab{A})$ is the set of (potentially unbounded) measurable functions $f: E \to \R$ such that there exists a measurable function $h: E\to \R$ with $t \mapsto h(Z_t)$ $\P_z$-integrable almost surely for each initial point $z$, and such that the process 
\begin{equation}
    C_t^f := f(Z_t) -f(Z_0) -\int_0^t h(Z_s)\dif s, \quad t\ge 0,
    \label{eq:Ctf}
\end{equation}
is a local martingale; see \cite[(14.16)]{Davis1993}. For $f \in\mathcal D(\textswab{A}) $, $\textswab{A}f=h$, for $h$ as in \eqref{eq:Ctf}.

\begin{proposition}
    The extended generator is an extension of the infinitesimal generator:
    \begin{enumerate}
        \item $\Dom(A)\subset \Dom(\textswab{A})$
        \item $A f = \textswab{A} f $ for any $f\in \Dom (A)$.
    \end{enumerate}
\end{proposition}
\begin{proof}
    See \cite[p32]{Davis1993}.
\end{proof}

To simplify notation, we will define the action of our probability kernel, $Q$, on a function, $f$, as 
\[
(Qf)(z)=\int_E f(y)Q(\mbox{d}y|z).
\]

We will assume throughout this work that the \textit{standard conditions} of Davis \cite[(24.8)]{Davis1993} hold. Under this assumption or PDMPs, $(\textswab{A}, \mathcal D(\textswab{A}))$ are fully characterized in \cite[(26.14)]{Davis1993}. In particular, the set $\mathcal D(\textswab{A})$ is entirely known, and for all $f\in \mathcal{D}(\textswab{A})$:
\begin{equation}
    \label{eq:extended-gen-expression}
    \textswab{A} f(z) = \Xi f(z) + \lambda(z)\{ (Qf)(z)-f(z) \},
\end{equation}
where $\Xi$ is the differential operator associated to the deterministic flow of the PDMP.

The PDMP samplers we are interested in (see Section~\ref{subsec:PDMP_samplers}) have flow corresponding to free transport, with corresponding $\Xi$ operator for continuously differentiable $f$,
\begin{equation*}
    \Xi f(x,v) = v\cdot \nabla_x f(x,v).
\end{equation*}

\begin{remark}
To reiterate, while the domain of the strong generator $\mathcal D(A)$ is not known, the domain $\mathcal D(\textswab{A})$ is known and $\Dom(A)\subset \Dom(\textswab{A})$.
\end{remark}

\section{A general framework for the invariant measure of a PDMP}
\label{sec:invariant_general}

A challenge in the piecewise-smooth setting is that the usual approach to constructing and working with PDMPs does not work without changing the topology. In particular, existing results concerning the invariant measure of such processes requires the process to be Feller. For PDMPs with boundaries, this is in fact not the case in general \cite{Davis1993}.

\subsection{Strong continuity of the semigroup}
First, we give a general result that is not tied to our specific context and is valid for any piecewise-deterministic Markov process that follows Davis's construction, \cite[Section 24, Conditions (24.8)]{Davis1993}.

Let $\mathcal{F}$ be the space of $C^1$ functions contained in $\mathcal{D}(\textswab{A})$ with compact support.
\begin{proposition}
    \label{prop:strong-continuity}
    Assume that the deterministic flow and the jump rates are bounded on any compact set, and that $Qf$ 
    has compact support whenever $f$ has compact support. Then, $\mathcal{F} \subset \mathcal{B}_0$. In other words,
        the semigroup $P_t$ of the process is strongly continuous on $\mathcal{F}$: 
        \[
            P_t f \rightarrow f \text{ for all } f \in \mathcal{F},
        \]
        in $\|\cdot\|_\infty$, as $t\downarrow 0$.
\end{proposition}
\begin{proof}
Let $f\in \mathcal F$. Since $f \in \mathcal{D}(\textswab{A})$, 
\[
C_t^f := f(Z_t) -f(Z_0) -\int_0^t \textswab{A} f(Z_s) \dif s, \quad t\ge 0,
\]
is a local martingale. Furthermore, by examining the expression of $\textswab{A}f$, one sees that it can be rewritten as
\[
    \textswab{A} f(z) = \Xi f(z) + \lambda(z) Qf(z) - \lambda(z) f(z) 
\]
from (\ref{eq:extended-gen-expression}).
Since $f$ and $Qf$ have compact support and are bounded, using the assumptions on $\Xi$ and $\lambda$,
we deduce that $\textswab{A}f$ is bounded. 

Since $f$ and $\textswab{A}f$ are bounded, $C_t^f$ is bounded for any fixed $t$ which implies that it is a martingale. More precisely: consider the stopped process $C^f_{t\wedge T}$ for any $T>0$. This is a uniformly bounded local martingale, and is hence a true martingale.

We have $C_0^f = 0$ hence for any starting point $z$ and $t>0$,
\[
    \mathbb{E}_z[C_t^f] = 0.
\]
Hence
\[
    P_t f (z) - f(z) = \int_0^t P_s \textswab{A} f(z) \dif s,
\]
where we used Fubini's theorem to swap the integral and the expectation.
Since $P_t$ is a contraction for the sup norm, we see that
\begin{align*}
    \|P_t f - f\|_\infty &\leq \int_0^t \|P_s \textswab{A} f \|_\infty \dif s \\
    &\leq \int_0^t \| \textswab{A}f \|_\infty \dif s \\
    &\leq t \|\textswab{A}f\|_\infty.
\end{align*}
We thus conclude that $P_t f - f \rightarrow 0$ as $t \downarrow 0$. 
\end{proof}

\begin{remark}
The set $\mathcal{F}$ does not capture every function of $\mathcal{B}_0$, nor is it invariant under $P_t$. We will will not attempt to prove that $\mathcal{F}$ is a core of the infinitesimal generator.
\end{remark}

Recall that a set of functions $\mathcal F_0 \subset \mathcal B(E)$ separates measures if for any probability measures $\mu_1, \mu_2$ on $E$, $\int f\dif \mu_1=\int f \dif \mu_2$ for all $f\in \mathcal F_0$ implies that $\mu_1=\mu_2$. In order to study the invariant measure through the semigroup and its effect on functions, it is important to consider sets of functions which separate measures. Therefore, we will now show that $\mathcal{F}$ separates measures on $E^0$.

\begin{proposition}
    Assume that the jump kernel $Q$ is such that for any $z \in \Gamma$, the measure $Q(\cdot | z)$ is supported on the boundary $\cup_k \partial E^0_k$.
    Then $\mathcal{F}$ separates measures on $E^0$. 
    \label{prop:F_sep_meas}
\end{proposition}
\begin{proof}
    Consider the set of $C^1$ functions $f:E \to \R$, which are compactly supported on each open set $E_k^0$. The collection of such functions separates measures on $E^0$.
    
    We will show that such functions belong to $\mathcal D(\textswab{A})$, and hence to $\mathcal{F}$, by using the explicit characterisation of  $\mathcal D(\textswab{A})$ from \cite[Theorem 26.14]{Davis1993}.
    
    Firstly, if $f$ is $C^1$ with compact support, Conditions~1 and 3 of \cite[Theorem 26.14]{Davis1993} are automatically satisfied. 
    
    It remains to check the boundary condition:
    \begin{equation}
        f(z) = \int_E f(y) Q(\dif y;z) = Qf(z),\quad z\in \Gamma.
        \label{eq:bdy_cond}
    \end{equation}
    However, since we are considering only $f$ which are compactly supported on each $E_k^0$, it holds that $f(z)=0$ for any $z \in \Gamma$ on the boundary. Recalling that $Q(\cdot | z)$ is supported only on the boundary, it follows that $Qf(z)=0$ also for any $z \in \Gamma$. It hence follows that the boundary condition is satisfied.
\end{proof}

\subsection{Invariant measure}

We now turn to giving conditions for the invariant measure of our PDMP. The following lemma will be important within the proof, as it allows us to ignore
contributions from the boundary when calculating expectations over the path of the PDMP.
\begin{lemma}
    \label{lemma:boundary-time}
    For all $z\in E$, the process starting from $z$ spends a negligible amount of time on the boundary: for any $t>0$,
    \[
        \int_0^t 1_{Z_s\in E\setminus E^0}\dif s = 0,
    \] $\mathbb{P}_z$-a.s.
\end{lemma}
\begin{proof}
    In Davis's construction, the number of events, including jumps at the boundary, is countable for every trajectory. Hence, for every trajectory of the process, the set of times for which $Z_t \in E\setminus E^0$ is countable and therefore negligible.
\end{proof}

\begin{theorem}
\label{th:distr-invariant-active}
Let $\mu$ be a measure on $E$.
Assuming the following conditions hold
\begin{enumerate}
    \item  the vector field of the deterministic flow and the jump rates are bounded on any compact set,
    \item $Qf$ has compact support whenever $f$ has compact support, and $Q$ satisfies the condition of Proposition~\ref{prop:F_sep_meas},
    \item $\mu(E\setminus E^0) = 0$,
    \item for all $f \in \mathcal{D}(A)$, $\int_E Af \dif\mu = 0$,
\end{enumerate}
then $\mu$ is invariant.
\end{theorem}
\begin{proof}
The assumptions of Proposition~\ref{prop:strong-continuity} hold, hence the semigroup $P_t$ is strongly continuous on $\mathcal{F}$.
Using this fact and Proposition~1.5 from Ethier and Kurtz \cite{Ethier1986} (or \cite[(14.10)]{Davis1993}), we note that for any $f \in \mathcal{F}$ and $t > 0$, we have that $\int_0^t P_s f \dif s \in \mathcal{D}(A)$, the domain of the strong generator. We also have that
\begin{equation}
        \int_E P_t f \dif \mu  = \int_E \left [f + A \int_0^t P_s f \dif s\right  ] \dif \mu=  \int_E f \dif \mu,
        \label{eq:th-invariant-measure1}
\end{equation}
where we have used our assumption that $\int Ag \dif \mu=0$ for any $g \in \Dom (A)$ and taken $g = \int_0^t P_s f\dif s$.

Let $f_B(z) := f(z) 1_{z\in E\setminus E^0}$ and $f_0(z) := f(z) 1_{z\in E^0}$ be a decomposition of $f$ with $f = f_0 + f_B$. Let $1_B(z) = 1_{z\in E\setminus E^0}$ be the indicator function of the boundary $E\setminus E^0$.

From Lemma~\ref{lemma:boundary-time}, $\int_0^t P_s 1_B(z) \dif s = 0$ for all $z\in E$. Hence we have that $\int_E [ \int_0^t P_s 1_B(z) \dif s ] \dif \mu(z) = 0$, and by using Fubini's theorem, that
\[
\int_0^t \int_E P_s 1_B \dif \mu \dif s = 0.
\]
By the nonnegativity of $P_t 1_B$, there exists a null set $\mathcal{N} \subset \mathbb{R}^+$ such that for all $t\in \mathbb{R}^+\setminus \mathcal{N}$, 
\begin{equation*}
    \int_E P_t 1_B \dif\mu = 0.  
\end{equation*}

For all $z$, $f_B(z) \leq \|f\|_\infty 1_B(z)$, hence for all $t \notin \mathcal N$, $\int_E P_t f_B \dif \mu = 0$. Hence $\int_E P_t f \dif \mu = \int_E P_t f_0 \dif\mu$ for all $t \notin \mathcal N$.
Since $\mu$ is supported on $E^0$, $\int_E f \dif \mu = \int_E f_0 \dif \mu$ and we deduce using \eqref{eq:th-invariant-measure1} that for all $t\notin \mathcal N$:
\begin{equation}
    \label{eq:th-invariant-measure2}
    \int_E P_t f_0 \dif \mu = \int_E f_0 \dif \mu.
\end{equation}

Let $\mu_t$ be the law $Z_t$ with $Z_0 \sim \mu$. Let $\mu_t^0$ and $\mu_t^B$ be the measures defined by $\mu_t^0(A) = \mu_t(A\cap E^0)$ and $\mu_t^B(A) = \mu_t(A\cap (E\setminus E^0))$. Using \eqref{eq:th-invariant-measure2}, for all $t\notin \mathcal N$:
\[
    \int_{E^0} f_0 \dif \mu_t^0 = \int_{E} P_t f_0 \dif \mu = \int_E f_0 \dif \mu = \int_{E_0} f_0 \dif \mu^0.
\]
Since $\mathcal F$ separates measures on $E_0$ by Proposition~\ref{prop:F_sep_meas}, $\mu_t^0 = \mu^0$ for all $t\notin \mathcal N$. Furthermore, $\mu(E) = \mu_t(E)$ and $\mu(E) = \mu^0(E)$, hence $\mu_t^0(E) = \mu_t(E)$ and $\mu_t^B(E) = 0$. Thus $\mu_t = \mu_t^0 = \mu^0 = \mu$ for all $t\notin \mathcal N$.

Let $t_1,t_2\notin \mathcal N$. Then $\mu_{t_1} = \mu_{t_2} = \mu$, and for all functions $g$ which are measurable and bounded, it holds that
\[
    \int_E P_{t_1+t_2} g \dif \mu = \int_E P_{t_1} (P_{t_2} g) \dif \mu = \int_E P_{t_2} g \dif \mu_{t_1} = \int_E P_{t_2} g \dif \mu = \int_E g \dif \mu.
\]
Hence $\mu_{t_1+t_2} = \mu$.
To conclude, since $\mathcal N$ is a null set, for all $ t > 0$, there exists $t_1,t_2 \notin \mathcal N$ such that $t = t_1 + t_2$, and therefore $\mu_t = \mu$.
\end{proof}

\section{PDMP samplers with active boundaries}
Let $U_k$ be an open set of $\mathbb{R}^{d_k}$ for all $k \in K$, and let $\mathcal{V}^k \subset \mathbb{R}^{d_k}$ be the velocity space associated to the PDMP sampler on $U_k$. We consider the state space defined following Davis's construction described in Section~\ref{sec:Davis-construction}: first, set
\[
    E_k^0 = U_k \times \mathbb{R}^{d_k}.
\]
\begin{remark}
We do not take $E_k^0 = U_k\times \mathcal{V}^k$ because $E_k^0$ must be open and $\mathcal{V}^k$, the set of velocities of the PDMP sampler, might not be.
\end{remark}
Let $\pi$ be a measure on the disjoint union $\cup_k U_k$ with a density $\pi_k$ on each $U_k$, where $\pi_k$ can be extended continuously to the closure $\bar{U}_k$. Let $p_k$ be the marginal velocity probability distribution on $\mathbb{R}^{d_k}$ with support on $\mathcal{V}^k$ and let $\mu(k,x,v) = \pi_k(x) p_k(v)$ be a density on $\cup_{k\in K} U_k \times \mathcal{V}^k$ defining a measure on $E$.

The core result of this section relies on integration by parts, and as such requires extra assumptions on the sets $U_k$. For clarity of the exposition, we give here an intuitive version of the required assumptions, and a detailed version can be found in the appendix in Assumptions~\ref{ass:space} and \ref{ass:space-v}.
\begin{assumption}
    Assumptions \ref{ass:space} and \ref{ass:space-v} can be informally described as:
    \begin{enumerate}[label=(\roman*)]
        \item $U_k$ has no interior discontinuities on a $(d_k -1)$-dimensional subset. 
        \item The boundary $\partial U_k$ can be decomposed into a finite union of smooth parts, on each of which the normal is well-defined. \label{ass:informal-finite-union}
        \item The set of corner points, on which the normals of the boundary are not defined, is small.
        \item For each $x\in U_k, v\in \mathcal{V}^k$, there is a finite number of intersections between each line $x+\mathbb{R}v$ and $\partial U_k$.
        \item For each $v \in \mathcal{V}^k$, the projection of the points on the boundary, which are tangent to the velocity, onto $H_v = \mathrm{span}(v)^{\perp}$ is small. \label{ass:informal-projection-tangent}
    \end{enumerate}
    \label{ass:space-informal}
\end{assumption}

Let $N_k$ be the subset of points $x$ on the boundary $\partial U_k$ where the normal $n(x)$ is properly defined (see \eqref{eq:def-Nk} of the Appendix for a precise statement). 

\begin{assumption}
    \label{ass:pi-interior}
    \begin{enumerate}[label=(\roman*)]
        \item $\int |\lambda(z)| \dif \mu < \infty$; \label{ass:lambda-integrable}
        \item for all $k\in K$, $\pi_k$ is $C^1$ in $\bar{U}_k$ \label{ass:pi-C1}
        \item For any $k \in K$, and any $v \in \mathcal{V}^k$, $\nabla \pi_k \cdot v$ is in $L_1(\mathrm{Leb})$. \label{ass:pi-dot-integrable}
        \item for every $z\in E^0$, $Q(\cdot |z)$ is a probability measure supported on the interior $E^0$.
    \end{enumerate}
\end{assumption}

\begin{theorem}
If Assumptions~\ref{ass:pi-interior},~\ref{ass:space}~and~\ref{ass:space-v} hold, then for all $f\in \Dom(A)$:
\begin{align*}
\int_{E} Af \dif \mu = &-\int_{E^0} f(k,x,v) \nabla_x \mu \cdot v  \dif x \dif v +\int_{E^0} \lambda(z) [Qf(z) - f(z)] \dif \mu\\
    &+\sum_{k\in K}
    \int_{x\in \partial U_k \cap N_k,v\in \mathcal{V}^k} f(k,x,v) \pi_k(x) \langle n(x),v \rangle \dif \sigma(x) \dif p(v)
\end{align*}
where $f(k,x,v)$ is defined as $f(k,x,v) := \lim_{t\downarrow 0} f(k,x-tv,v)$, for $x\in \partial U_k \cap N_k,v\in \mathbb{R}^{d_k}$ such that $\langle n(x),v\rangle > 0$,and $\sigma$ is the Lebesgue measure induced on the surface $\partial U_k$.
\label{th:intAf}
\end{theorem}
\begin{proof}[Proof (sketch).]
The outline of the proof is that we first use the definition of the generator acting on a function $Af$ and then rearrange the resulting integral using integration by parts. If our PDMP has no boundary this would give just the first two terms on the right-hand side \cite[]{fearnhead2018piecewise,vanetti2017piecewise}. The effect of the boundaries is to introduce the additional terms when performing integration by parts.
For full details of the proof, see Section \ref{subsec:proof-th-inth}.
\end{proof}

\section{PDMP samplers for piecewise continuous densities} \label{sec:ctsbyparts}
Let $\pi$ be a density on $\mathbb{R}^d$ and $\{U_k:k\in K\}$ be a finite collection of disjoint open subsets of $\mathbb{R}^d$, such that $\overline{\cup_k U_k}=\R^d$, which satisfy our technical Assumptions~\ref{ass:space} and \ref{ass:space-v}. We assume that $\pi$ is $C^1$ on each $\bar{U}_k$. We are now in the same setting as the previous section. 

Let $\partial U = \bigcup_{k\in K} (\partial U_k \cap N_k)$ be the union of the boundaries, i.e. the set of discontinuities of $\pi$. We consider now only points on the boundary where exactly two sets $U_{k_1}, U_{k_2}$ intersect, and where the respective normals are well-defined; the set of points where more sets intersect or the normal is ill-defined form a null set by assumption and thus have no impact on the resulting invariant distribution.

In the following we will restrict ourselves to transition kernels on the boundary that keep the location, $x$, unchanged and only update the velocity, $v$. 

For each such $x$ in $\partial U$, there exists $k_1(x)$ and $k_2(x)$ such that $x \in \bar{U}_{k_1}$ and $x \in \bar{U}_{k_2}$. We will define the ordering of the labels such that $\pi_{k_1}(x)<\pi_{k_2}(x)$. 
Let $n(x)$ be the outer normal for $U_{k_1}$. Thus this is the normal that points to the region $k_2(x)$, which is the region that has higher density, under $\pi$, at $x$. 

Let $\mathcal{V}_x^+ = \{v \in \mathcal{V} | \langle v,n(x) \rangle > 0\}$ and  $\mathcal{V}_x^- = \{v \in \mathcal{V} | \langle v,n(x) \rangle < 0\}$. Thus $\mathcal{V}_x^+$ is the set of velocities that would move the position $x$ into $k_2(x)$, thereby increasing the density under $\pi$, and $\mathcal{V}_x^+$ is the set of velocities that move the position into $k_1(x)$.

For $x\in \partial U$, let $l_x$ be the following (unnormalized) density on $\mathcal{V}$ 
\[
    l_x(v) =\left\{ \begin{array}{cl} |\langle n(x),v \rangle| p(v) \pi_{k_2(x)}(x) & \forall v\in \mathcal{V}_x^+,\\
    |\langle n(x),v \rangle| p(v) \pi_{k_1(x)}(x) & \forall v\in \mathcal{V}_x^-. \end{array} \right.
\]
This is just proportional to the density $p(v)$ weighted by the size of the velocity in the direction of the normal $n(x)$ and weighted by the density at $x$ either in the region, $k_1(x)$ or $k_2(x)$, that the velocity is moving toward.

Let $Q'_x$ be the Markov kernel for the velocity obtained by flipping the velocity and then applying Markov kernel $Q$. Since we assume $Q$ only changes the velocity, we have $Q'_x(\dif v'|v)=Q(\dif x,\dif v'|x,-v)$.

\begin{theorem}
Assume that for all $v\in \mathcal{V}_x$ that $p(v)=p(-v)$, and that the transition kernel $Q$ only changes the velocity, and define the family of kernels, $Q'_x$ for $x\in \delta U$ as above. 
Further, assume that
     \[-\int_{E^0} f(k,x,v) \nabla_x \mu \cdot v \dif x \dif v +\int_{E^0} \lambda(z) [Qf(z) - f(z)] \dif\mu = 0,\]
     and that 
     \begin{equation} \label{eq:boundary_condition}
     \mbox{ $\forall x \in \partial U$, $l_x$ is an invariant density of $Q'_x$.}
     \end{equation}
    Then $\int_E Af \dif \mu = 0$
    for all $f \in \mathcal{D}(A)$, and $\mu$ is the invariant distribution of the process.
    \label{thm:bdy_inv}
\end{theorem}
\begin{proof}
Starting from Theorem~\ref{th:intAf}, for all $f\in \mathcal{D}(A)$:
\begin{align*}
\int_{E} Af \dif\mu = &-\int_{E^0} f(k,x,v) \nabla_x \mu \cdot v \dif x \dif v +\int_{E^0} \lambda(z) [Qf(z) - f(z)] \dif \mu\\
    &+\sum_{k\in K}
    \int_{x\in \partial U_k \cap N_k,v\in \mathcal{V}} f(k,x,v) \pi_k(x) \langle n(x),v \rangle \dif \sigma(x) \dif p(v).
\end{align*}
By assumption, $-\int_{E^0} f(k,x,v) \nabla_x \mu \cdot v \dif x \dif v +\int_{E^0} \lambda(z) [Qf(z) - f(z)] \dif\mu = 0$, and so we simplify the integral of $Af$ to:
\[
    \int_{E} Af \dif \mu =\sum_{k\in K} \int_{x\in \partial U_k \cap N_k,v\in \mathcal{V}} f(k,x,v) \pi_k(x) \langle n(x),v \rangle \dif\sigma(x) \dif p(v).
\]

To simplify notation, in the rest of the proof we write $k_1$ for $k_1(x)$ and $k_2$ for $k_2(x)$
We rewrite the previous equation:
\begin{align*}
    \int_{E} Af \dif \mu &=\int_{x\in \partial U} \int_{v \in\mathcal{V}_x^+} \left[ f(k_2,x,v) \pi_{k_2}(x) - f(k_1,x,v) \pi_{k_1}(x) \right] |\langle n(x),v \rangle| \dif p(v) \dif \sigma(x)\\
    &-\int_{x\in \partial U} \int_{v \in\mathcal{V}_x^-} \left[ f(k_2,x,v) \pi_{k_2}(x) - f(k_1,x,v) \pi_{k_1}(x) \right] |\langle n(x),v \rangle| \dif p(v) \dif\sigma(x).
\end{align*}
Using that if $v\in \mathcal{V}_x^-$ then $-v \in \mathcal{V}_x^+$ we can rewrite the right-hand side as
\begin{align*}
&    \int_{x\in \partial U} \left(
\int_{v \in\mathcal{V}_x^+} \left[ f(k_2,x,v) \pi_{k_2}(x) - f(k_1,x,v) \pi_{k_1}(x) \right] |\langle n(x),v \rangle| \dif p(v)
    \right. \\
& \left. - \int_{v \in\mathcal{V}_x^+} \left[ f(k_2,x,-v) \pi_{k_2}(x) - f(k_1,x,-v) \pi_{k_1}(x) \right] |\langle n(x),v \rangle| \dif p(-v)\right) \dif\sigma(x).
\end{align*}
A sufficient condition for $\int_E Af \dif \mu = 0$ is that for all $x$ the integral over $v$ in the brackets is 0.

Any $f \in \mathcal{D}(A)$ satisfies the boundary condition \eqref{eq:bdy_cond} on $\Gamma$. For $v \in \mathcal{V}^+_x$, we have $(k_1,x,v) \in \Gamma$ and $(k_2,x,-v) \in \Gamma$, hence,
\begin{align*}
    f(k_1,x,v) &= \int f(z') Q(\dif z'|(k_1,x,v)),\\
    f(k_2,x,-v) &= \int f(z') Q(\dif z'|(k_2,x,-v)).
\end{align*}
Thus our sufficient condition for $\int_E Af \dif \mu = 0$ becomes
\begin{align*}
    &\int_{\mathcal{V}_x^+} [f(k_2,x,v) \pi_{k_2}(x) -  Q f(k_1,x, v) \pi_{k_1}(x)] |\langle n(x),v \rangle|\dif p(v) \\
    &= \int_{\mathcal{V}_x^+} [Qf(k_2,x,-v) \pi_{k_2}(x) - f(k_1,x,-v) \pi_{k_1}(x)] |\langle n(x),v \rangle| \dif p(-v).
\end{align*}
Using again the fact that if $v\in \mathcal{V}_x^+$ then $-v \in \mathcal{V}_x^-$, this condition can be rewritten as: 
\begin{equation*}
    \begin{split}
         \int_{\mathcal{V}^+_x} f(k_2,x,v)&\pi_{k_2}(x) |\langle n(x),v \rangle|\dif p(v)
+ \int_{\mathcal{V}^-_x} f(k_1,x,v)\pi_{k_1}(x) |\langle n(x),v \rangle|\dif p(v) \\
&= \int_{\mathcal{V}^+_x} Qf(k_2,x,-v)\pi_{k_2}(x) |\langle n(x),v \rangle|\dif p(-v)\\
&\quad + \int_{\mathcal{V}^-_x} Qf(k_1,x,-v)\pi_{k_1}(x) |\langle n(x),v \rangle|\dif p(-v).
    \end{split}
\end{equation*}
We can then write this in terms of $l_x$ and $Q'_x$ by
introducing a function $f'(x,v)$ that is defined as
\[
f'(x,v)=f(k_1,x,v) \mbox{ if $v\in\mathcal{V}_x^-$, and }
f'(x,v)=f(k_2,x,v) \mbox{ if $v\in\mathcal{V}_x^+$}.
\]
Then, using $p(v)=p(-v)$ and the definitions of $l_x$ and $Q'_x$, our sufficient condition becomes
\[ \int_{\mathcal{V}_x} f'(x,v) \dif l_x(v) =
\int_{\mathcal{V}_x} \left(\int_{\mathcal{V}_x} f'(x,v')Q'_x(\dif v'|v) \right) \dif l_x(v). 
\]
This is true if $l_x$ is an invariant density of $Q'_x$.
\end{proof}

\subsection{The case of restricted domains}
\label{subsec:compare_restricted}
One special case of our general construction is the situation where the target density $\pi$ on $\R^d$ is only compactly supported. This particular scenario was already considered in the work of \cite{BIERKENS2018148}. We briefly compare our respective results in this setting.

Firstly, the key boundary condition of \cite{BIERKENS2018148}, their Equation (5), accords with our condition (\ref{eq:boundary_condition}) on $l_x$ in Theorem~\ref{thm:bdy_inv}. However, a full rigorous proof of their invariance result \cite[Proposition 1]{BIERKENS2018148} is not presented; in the Supplementary material of \cite{BIERKENS2018148}, they defer the full proof to future work. 

Finally, in the work of \cite{BIERKENS2018148}, another condition is required, their Equation (4). We do not have any equivalent in our work, the reason for which can be found in the proof provided in the supplementary material of \cite{BIERKENS2018148}: the boundary condition presented in equation (2) of the Supplementary material of \cite{BIERKENS2018148} is in fact only required to hold on the \textit{exit boundary}, which we have denoted $\Gamma$ (\textit{c.f.} \cite[Theorem 26.14, Condition 2]{Davis1993}), rather than on the entire boundary.

\section{Boundary kernels for usual PDMP samplers}

We give here possible Markov kernels for the Bouncy Particle Sampler, the Zig-Zag sampler, and the Coordinate Sampler. Since the condition of Theorem \ref{thm:bdy_inv} only depends on the velocity distribution, any two processes that share the same velocity distribution can use the same boundary Markov kernels. We present two approaches to constructing appropriate kernels on the boundary.

\subsection{Sampling $l$ using Metropolis--Hastings} \label{sec:MH}

Recall from Theorem \ref{thm:bdy_inv} that a valid kernel for the velocity when we hit a boundary can be constructed as follows: first, construct a transition kernel $Q'_x$ which leaves $l_x$ invariant; and if the current velocity is $v$, simulate a new velocity from $Q'_x(\cdot|-v)$. That is we can simulate a new velocity by (i) flipping the velocity, and (ii) applying $Q'_x$.

The simplest choice of $Q'_x$ is just the identity map, i.e. a kernel that keeps the velocity. However this would correspond to a transition kernel on the boundary which simply flips the velocity, thus forcing the PDMP to retrace its steps. Where possible, we can improve on this by choosing to be $Q'_x$ a kernel which samples from $l_x$, though it should be noted that this choice may be difficult to implement.

When $\mathcal{V}$ is bounded, an alternative is to define $Q'_x$ as a Metropolis--Hastings kernel \cite{dunson2020hastings} targeting $l_x$, with proposals from a uniform sampler of $\mathcal{V}$. The algorithm starting from $v$ then proceeds as follows:
\begin{enumerate}
    \item Sample $v'$ uniformly in $\mathcal{V}$.
    \item Accept $v^*=v'$ with probability  $\alpha = \frac{l_x(v')}{l_x(v)}$, otherwise set $v^*=v$.
\end{enumerate}
Of course, it is also possible to iterate the Metropolis--Hastings kernel several times to get a good sample from $l_x$ at a reasonable cost. 

\subsection{Limiting behaviours}

A natural strategy for constructing the transition kernel for the velocity at a boundary is to consider the limiting behaviour of the sampler for a family of continuous densities which tend to a piecewise-discontinuous density in an appropriate limit. We will do this for a density $\pi$ with one discontinuity on a hyperplane with normal $n$: $\pi(x) = \pizero 1_{\langle x,n\rangle < 0} + \pione 1_{\langle x,n\rangle > 0}$ with $\pizero < \pione$. 
In the following we will assume $\pizero>0$, but the extension of the arguments to the case $\pizero=0$ is straightforward.

We can approximate $\pi$ by a continuous density $\pi_k$ such that $\nabla \log(\pi_k)$ is piecewise constant:
\[
    \pi_k(x) = \pizero 1_{\langle x,n\rangle \in [-\infty,-C/k]} + \pione 1_{\langle x,n\rangle \in [0,\infty]} + \pione \exp\{k\langle x,n\rangle\} 1_{\langle x,n\rangle \in [-C/k,0]},
\]
where $C=\log(\pione/\pizero)$.
As $k\rightarrow \infty$ we can see that $\pi_k$ converges to $\pi$. In the following we will call the region where $\langle x,n\rangle \in [-1/k,0]$ the boundary region of $\pi_k$, as this is approximating the boundary defined by the discontinuity in $\pi$.

The advantage of using the densities $\pi_k$ is that the resulting behaviour of standard PDMP samplers is tractable, and, as we will see, the distribution of the change in velocity from entering to exiting the boundary region of $\pi_k$ will not depend on $k$. The effect of increasing $k$ is just to reduce the time spent passing through the boundary region -- and in the limit as $k\rightarrow\infty$ this becomes instantaneous.

We consider this limiting behaviour for BPS, the Coordinate Sampler, and Zig-Zag. Whilst we derive the transition distribution for the velocity in each case from this limiting behaviour, we will demonstrate that each distribution is valid for the corresponding sampler directly by showing that it satisfies our condition (\ref{eq:boundary_condition}). The proofs of the propositions in this section are deferred to Appendix~\ref{app:Q_proofs}.

\subsubsection{Limiting behavior of the Bouncy Particle Sampler} \label{sec:BPSlim}

Consider the BPS dynamics for sampling from $\pi_k$ for a trajectory that enters the boundary region, and ignore any refresh events. If the state of the BPS is $(x,v)$ then dynamics are such that events occur at a rate $\max\{0,-\langle v, \nabla \log \pi_k(x) \rangle\}$, and at an event the velocity is reflected in $\nabla \log \pi_k(x)$. Whilst in the boundary region, $\nabla \log \pi_k(x) = k n$.

For any $v$ such that $\langle v,n\rangle > 0$, it is clear that $\lambda(x,v) = 0$ for all $x$. Hence, the trajectory through the boundary region will be a straight line.

Let $v$ be such that $\langle v,n\rangle < 0$. Without loss of generality assume that the trajectory enters the boundary region at $t=0$ with $x_0 = 0$. If no jumps occurs, the trajectory will exit the boundary region at some time $t_e$, where  $\langle x_{t_{e}},n\rangle = -C/k$, which implies $t_{e} = -C/(k \langle v,n\rangle)$.  For such a trajectory, the Poisson rate whilst passing through the boundary region is  $\lambda = -k \langle v,n\rangle$. Remembering that $C=\log(\pione/\pizero)$, the probability of a trajectory passing through the region in a straight line is
\[
\exp\{-\lambda t_{e}\}=\exp\left\{-\log\left(\frac{\pione}{\pizero}\right) \right\}=\frac{\pizero}{\pione},
\]
which does not depend on $k$.

Finally, the probability of an event that changes the velocity in the boundary region is thus $1-\pizero/\pione$. If there is event, the velocity reflects in the normal and becomes $v'=v - 2\langle v,n\rangle$. As $\langle v',n\rangle>0$ no further events will occur whilst passing through the boundary region.

Hence the probability transition kernel assigns probabilities 
\[
Q(x,v'|x,v)=\left\{\begin{array}{cl}
1 &
v'=v \mbox{ and } \langle v,n\rangle>0,\\
\pizero/\pione & v'=v \mbox{ and } \langle v,n\rangle<0,\\
1-\pizero/\pione & v'=v - 2 \langle n,v\rangle n  \mbox{ and } \langle v,n\rangle<0.
\end{array}
\right.
\]

If we translate this into the corresponding transition kernel at a general discontinuity, for a trajectory that hits the boundary defined by the continuity at a point with unit normal $n=n(x)$
 then
\[
Q_{\BPS}(x,v'|x,v)=\left\{\begin{array}{cl}
1 &
v'=v \mbox{ and } v\in\mathcal{V}_x^+,\\
\pi_{k_1(x)}(x)/\pi_{k_2(x)}(x) & v'=v \mbox{ and } v\in\mathcal{V}_x^-,\\
1-\pi_{k_1(x)}(x)/\pi_{k_2(x)}(x) & v'=v - 2 \langle n,v\rangle n \mbox{ and } v\in\mathcal{V}_x^-.
\end{array}
\right.
\]
That is, if the trajectory is moving to the region of lower probability density, then it passes through the discontinuity with a probability proportional to the ratio in the probability densities. Otherwise, it reflects off surface of discontinuity.

\begin{proposition} \label{prop:BPS}
The transition kernel of the velocity, $Q'_x$, derived from  $Q_{\BPS}$ satisfies (\ref{eq:boundary_condition}).
\end{proposition}

This result holds for either implementation of the Bouncy Particle Sampler, i.e. where the distribution of the velocity is uniform on the sphere, or is an isotropic multivariate Gaussian. Examination of the proof of the proposition shows that we only require that $p(v)$ is spherically symmetric.

\subsubsection{Limiting behavior of the Coordinate Sampler} \label{sec:CSlim}

 In the coordinate sampler the velocity is always in the direction of one of the coordinates of $x$, and we will denote the set of $2d$ possible velocities by $\mathcal{V}$. The dynamics of the coordinate sampler are similar to those for BPS except that the transition kernel at an event is different. At an event the probability of the new velocity being $v'\in\mathcal{V}$ is proportional to $\max\{0,\langle v',\nabla \log \pi_k(x)\rangle\}$.
 
 The calculations for the transition kernel of the coordinate sampler for a trajectory that enter the boundary region of $\pi_k$ is similar to that for the BPS, except that, if there is an event, the distribution of the new velocity changes to that for the coordinate sampler. 
 
The resulting probability transition kernel, expressed for a general discontinuity is:
\[
Q_{\CS}(x,v'|x,v) = \left\{\begin{array}{cl} 1 & v'=v \mbox{ and } v\in \mathcal{V}_x^+ \\
\pi_{k_1(x)}(x)/\pi_{k_2(x)}(x) & v'=v \mbox{ and } v\in\mathcal{V}_x^-,\\
(1-\pi_{k_1(x)}(x)/\pi_{k_2(x)}(x))\frac{\langle v',n \rangle}{K} & v'\in \mathcal{V}_x^+ \mbox{ and } v\in\mathcal{V}_x^-,
\end{array}
\right.
\]
where $K=\sum_{v\in\mathcal{V}_x^+} \langle v,n \rangle$ is a normalising constant for the distribution of the new velocity if it changes.

That is, a trajectory moving to the higher probability region is unaffected by the discontinuity. For a trajectory moving to a lower probability region it either passes through the discontinuity, or bounces. The bounce direction is chosen at random from $v' \in \mathcal{V}_x^+$ with probability equal to the component of $v'$ in the direction of the normal at the discontinuity, $n$.

\begin{proposition} \label{prop:CS}
The transition kernel of the velocity, $Q'_x$, derived from  $Q_{\CS}$ satisfies (\ref{eq:boundary_condition}).
\end{proposition}

\subsubsection{Limiting behavior of the Zig-Zag Sampler} \label{sec:ZZlim}

For Zig-Zag the velocities are of the form $\{\pm1\}^d$. Given the positions, events occur independently for each component. That is if $v_i\in\{\pm1\}$ is the component of the velocity in the $i$ coordinate axis then this velocity will flip, i.e. change sign, at a rate $\max\{0,-v_i\partial \log \pi_k(x)/\partial x_i\}$. For the boundary region of $\pi_k(x)$ we have
\[
\frac{\partial \log \pi_k(x)}{\partial x_i} = k n_i,
\]
where $n_i$ is the $i$th component of the normal $n$.

If we consider the dynamics of Zig-Zag once it enters the boundary region of $\pi_k(x)$, then each velocity component with $v_in_i<0$ will potentially flip. Whilst travelling through the boundary region the rate at which such a $v_i$ will flip will be $-v_i n_i k$. Each component will flip at most once whilst the trajectory passes through the boundary region. 

The resulting dynamics are somewhat complicated, but can be easily simulated, using the following algorithm. 
\begin{enumerate}[label=\alph*)]
    \item \label{ZZ(a)} For $i=1,\ldots,d$ simulate $\tau_i$ as independent realisations of an exponential random variable with rate $k \max\{-n_iv_i,0\}$. If $n_iv_i\geq 0$ then set $\tau_i=\infty$.
    \item \label{ZZ(b)} Calculate the time $t^*$ at which we leave the boundary as the smallest value $t>0$ for which
    \[
    \sum_{i=1}^n v_in_i (t-2\max\{0,t-\tau_i\})= \pm C/k,
    \]
    or 
     \[
    \sum_{i=1}^n v_in_i (t-2\max\{0,t-\tau_i\})=0,
    \]
    \item \label{ZZ(c)} The new velocity has $v'_i=v_i$ if $\tau_i>t^*$ and $v'_i=-v_i$ otherwise.
\end{enumerate}

The key idea is that whilst the trajectory remains within the boundary region, each velocity component flips independently with its own rate. Step~\ref{ZZ(a)} then simulates the time at which each component of the velocity would switch. Step~\ref{ZZ(b)} then calculates, for the event times simulated in \ref{ZZ(a)}, what time, $t^*$ the trajectory will leave the boundary. There are two possibilities, the first corresponds to passing through the boundary, the second to bouncing back to the region where we started. For the first of these possibilities we have two possibilities, corresponding to $C/k$ and $-C/k$ to allow for the two possible directions with which we can enter the boundary region. Then in step~\ref{ZZ(c)} we calculate the velocity once the trajectory exits the boundary region -- using the fact that a velocity component will have flipped if and only if $\tau_i<t^*$.

It is simple to show that the distribution of the new velocity, $v'$, simulated in step~\ref{ZZ(c)} is independent of $k$, as the value of $k$ just introduces a scale factor into the definition of the event times, $\tau_i$, the and the exit time, $t^*$. Thus for a general general discontinuity, we define the probability transition kernel, $Q_{\ZZ}$ as corresponding to the above algorithm with $k=1$, with $n=n(x)$ and $C=\log(\pi_{k_2(x)}(x)/\pi_{k_1(x)}(x))$, the log of the ratio in probability density at $x$ for the two regions.

\begin{proposition} \label{prop:ZZ}
The transition kernel of the velocity, $Q'_x$ induced by $Q_{\ZZ}$ satisfies (\ref{eq:boundary_condition}).
\end{proposition}

\section{Comparison of samplers}

We now present some simulation results that aim to illustrate the theory, and show how different samplers and different choices of kernel at the discontinuities behave. We do this by considering a simple model for which it is easy to see the boundary behavior, with a target density
\[
    \pi(x) = \alpha_{\mathrm{in}} e^{-\frac{\|x\|^2}{2 \sigma_{\mathrm{in}}}} 1_{x \in [-1,1]^d} + \alpha_{\mathrm{out}} e^{-\frac{\|x\|^2}{2 \sigma_{\mathrm{out}}}} 1_{x \notin [-1,1]^d},
\]
which is Gaussian inside and outside the hypercube $[-1,1]^d$, with a discontinuous boundary on the hypercube.

For algorithms such as Zig-Zag and Coordinate Sampler, the choice of basis is extremely important. In particular, we expect the Zig-Zag Process to perform very well for product measures if the velocity basis is properly chosen. Hence we use a rotated basis where we generate a random rotation matrix $R$ and rotate the canonical basis by $R$. (For results for Zig-Zag with the canonical basis, see Appendix \ref{sec:add_results}.)

Since the goal of the experiments is to highlight the boundary behavior, and not a general comparison between BPS, Zig-Zag, and Coordinate Sampler, we only perform basic tuning of these algorithms, in particular with respect to the refresh rate which is necessary for BPS to be ergodic (without boundaries). For each sampler we consider a range of transition kernels at the discontinuity. These are the Metropolis--Hastings kernel of Section~\ref{sec:MH}; using 100 iterations of the Metropolis-Hastings kernel; and using the kernel derived from the limiting behaviour in Sections \ref{sec:BPSlim}--\ref{sec:ZZlim}. We have implemented all methods in dimensions $d=2$, 10 and 100; though we only present results for $d=100$ here, with the full results shown in Appendix \ref{sec:add_results}.

 An example of the resulting trajectories for $d=100$, for the case of a Gaussian restricted to the cube, i.e. $\alpha_{\mathrm{out}}=0$, can be found in Figures \ref{fig:d100-trajectories}. There are a number of obvious qualitative conclusions that can be drawn. First, using a single Metropolis--Hastings kernel leads to poor mixing -- with the trajectories often doubling back on themselves when they hit the boundary, and the trajectories for all three algorithms explore only a small part of the sample space. Increasing the number of Metropolis--Hastings kernels improves mixing noticeably, but does introduce diffusive-like behaviour. For the Bouncy Particle Sampler and the Zig-Zag Process, the kernel derived from the limiting behaviour allows for smaller changes in the velocity at the boundary. We see this as the trajectories look qualitatively different from the Metropolis kernels, with the diffusive behaviour being suppressed. 
 Overall the Bouncy Particle Sampler with the limiting kernel appears to mix best -- though this may in part be because this sampler mixes well for low-dimensional summaries of the target, but less well for global properties \cite[]{deligiannidis2018randomized}. 
 
\begin{figure}
    \centering
    \begin{tabular}{|c | c | c | c|}
    \hline
    & Bouncy Particle & Coordinate& Zig-Zag \\
    \hline
    \rotatebox{90}{Limit Kernel} & \includegraphics[width=3.5cm]{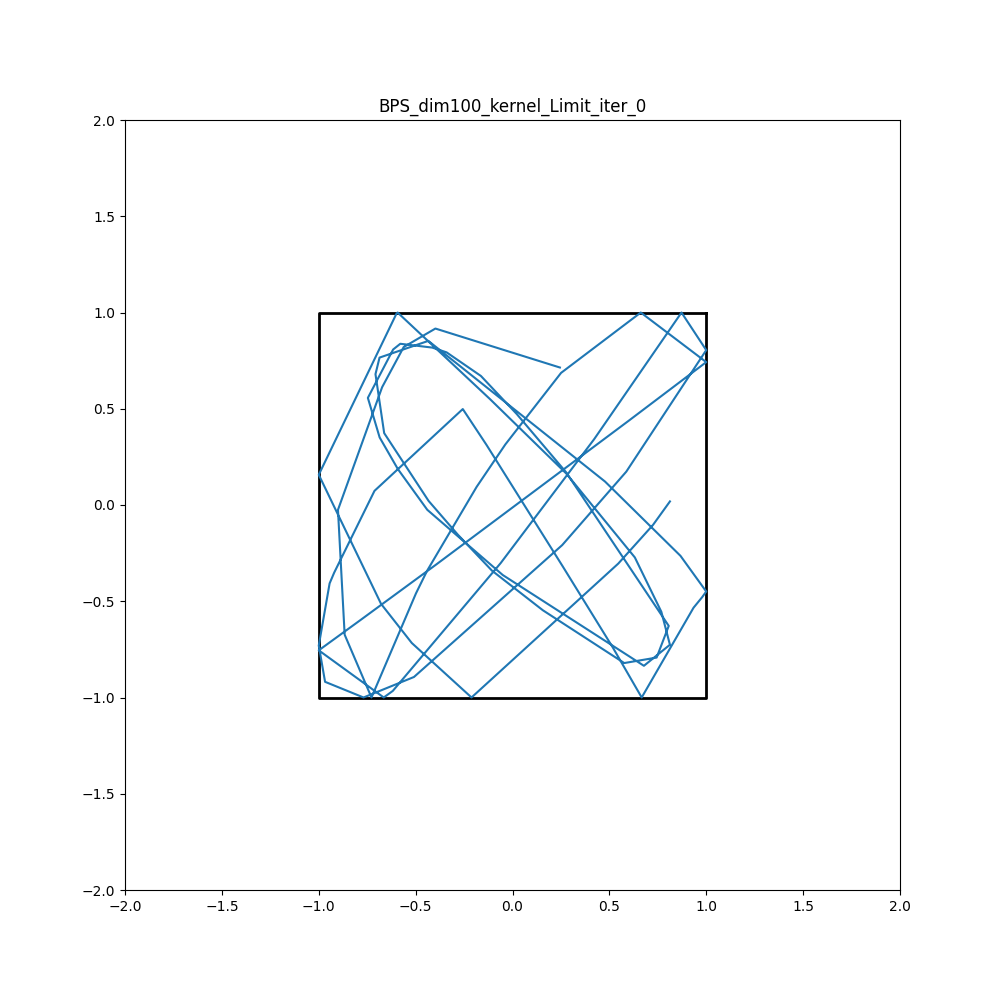} &
    \includegraphics[width=3.5cm]{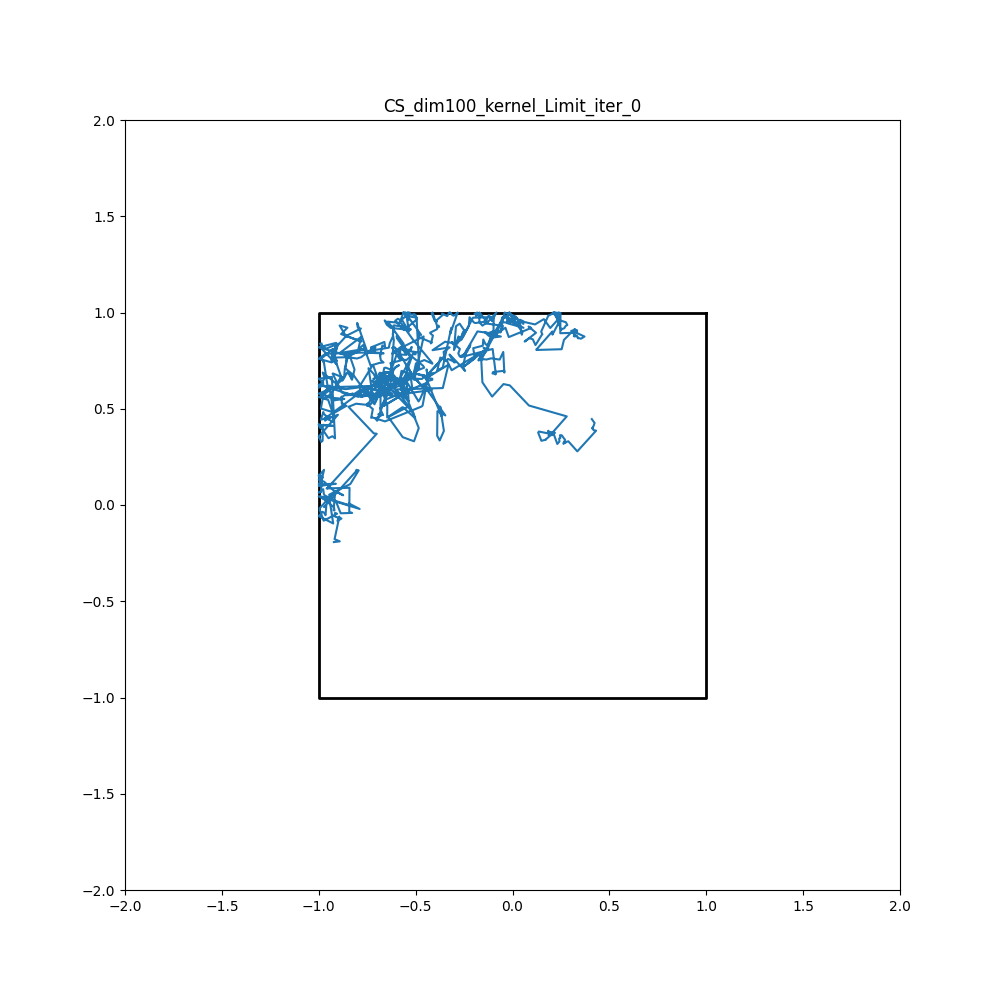} &
 \includegraphics[width=3.5cm]{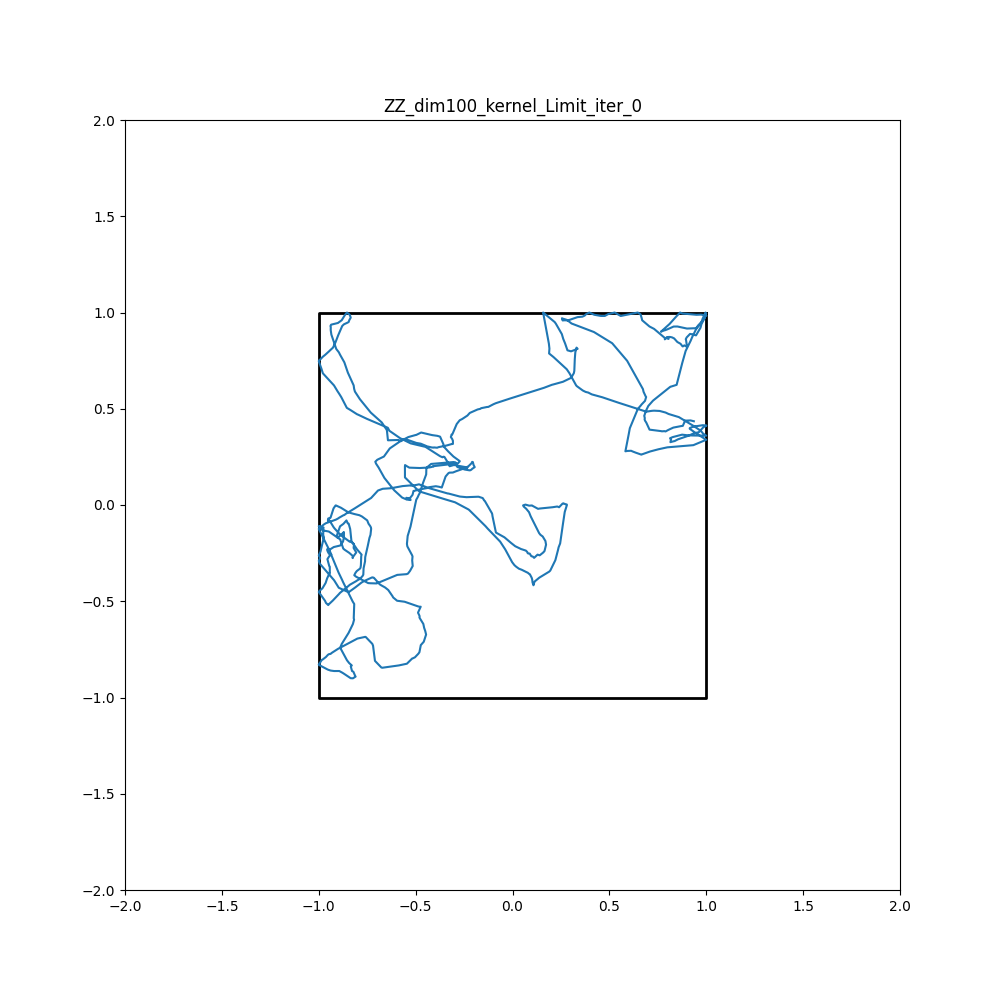} \\
   \hline
    \rotatebox{90}{Metropolis 1} & \includegraphics[width=3.5cm]{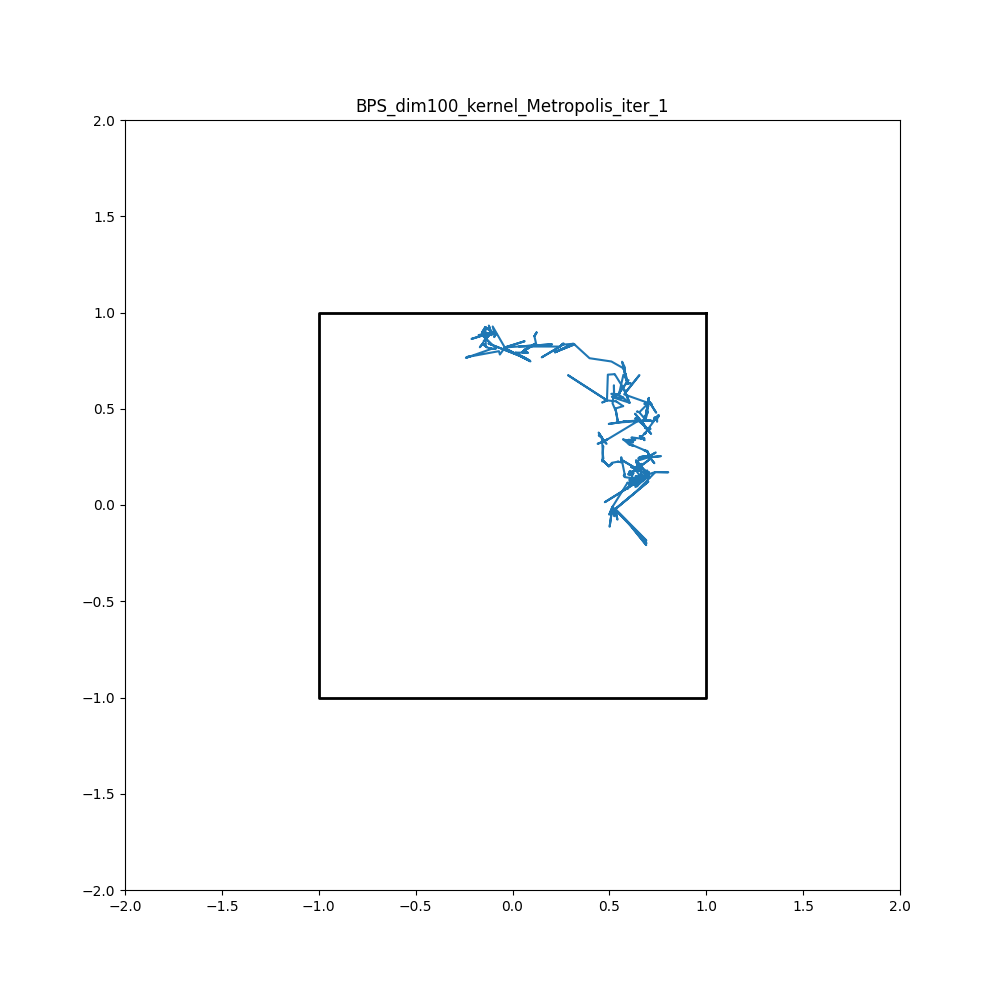} &
    \includegraphics[width=3.5cm]{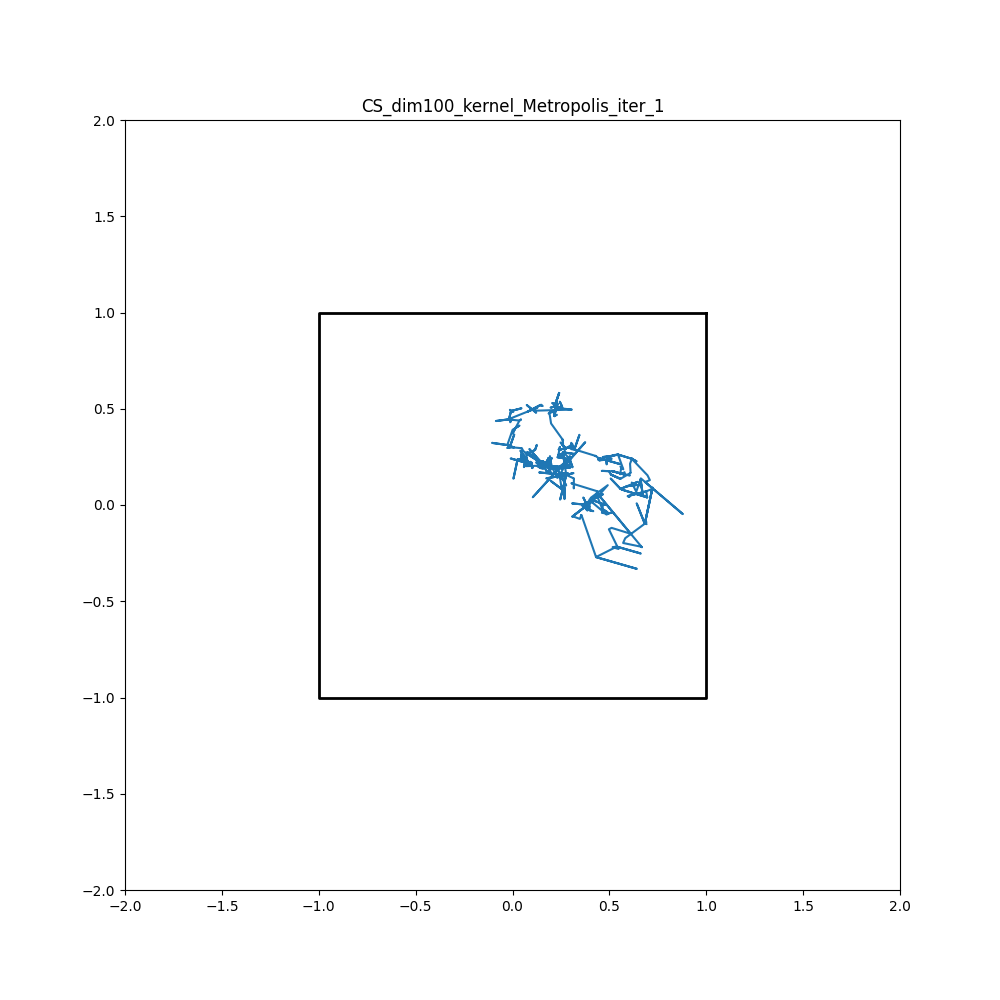} &
\includegraphics[width=3.5cm]{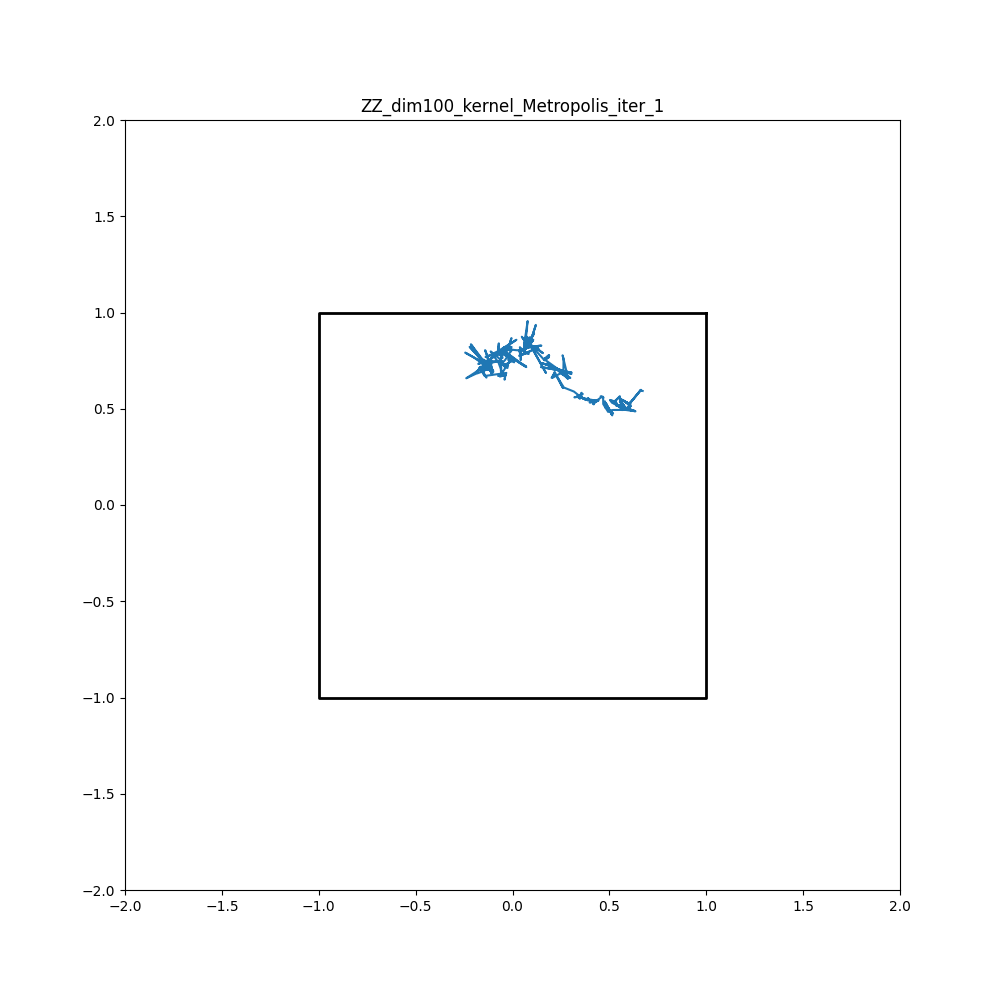} \\
     \hline
    \rotatebox{90}{Metropolis 100} & \includegraphics[width=3.5cm]{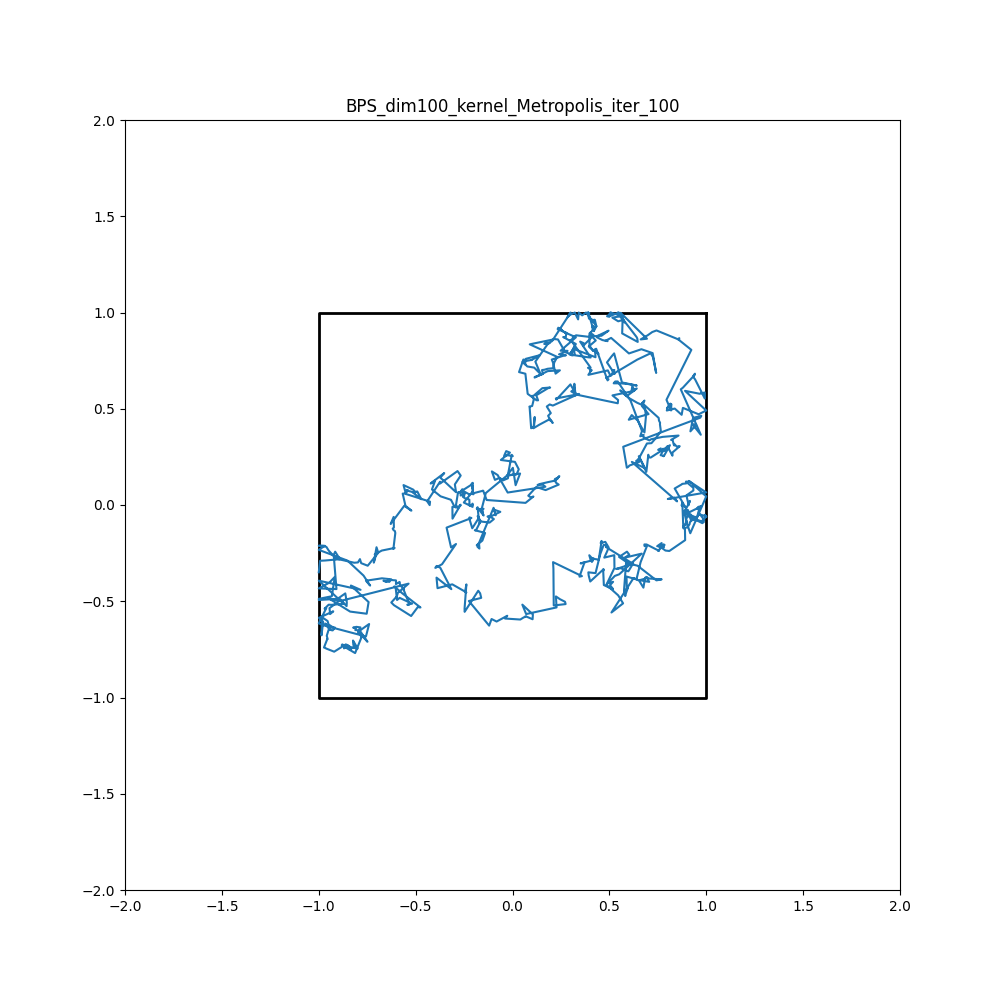} &
    \includegraphics[width=3.5cm]{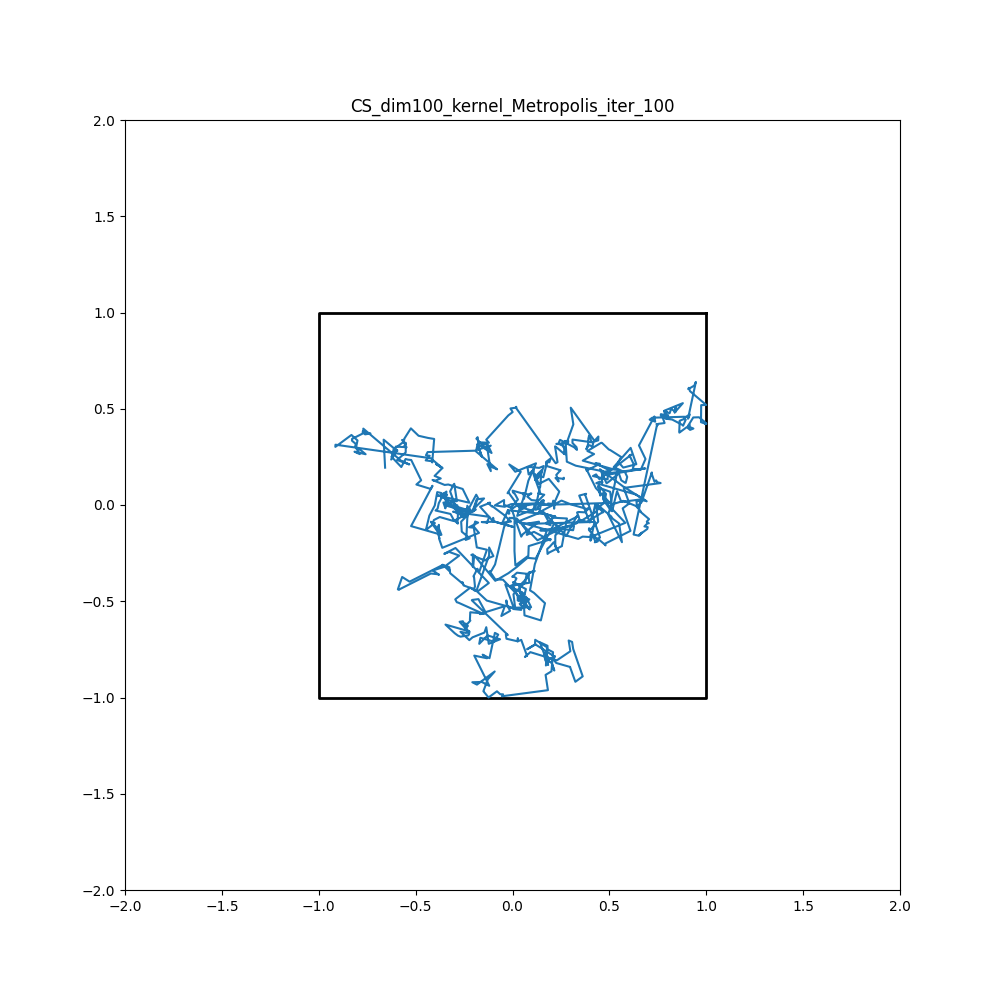}  &
\includegraphics[width=3.5cm]{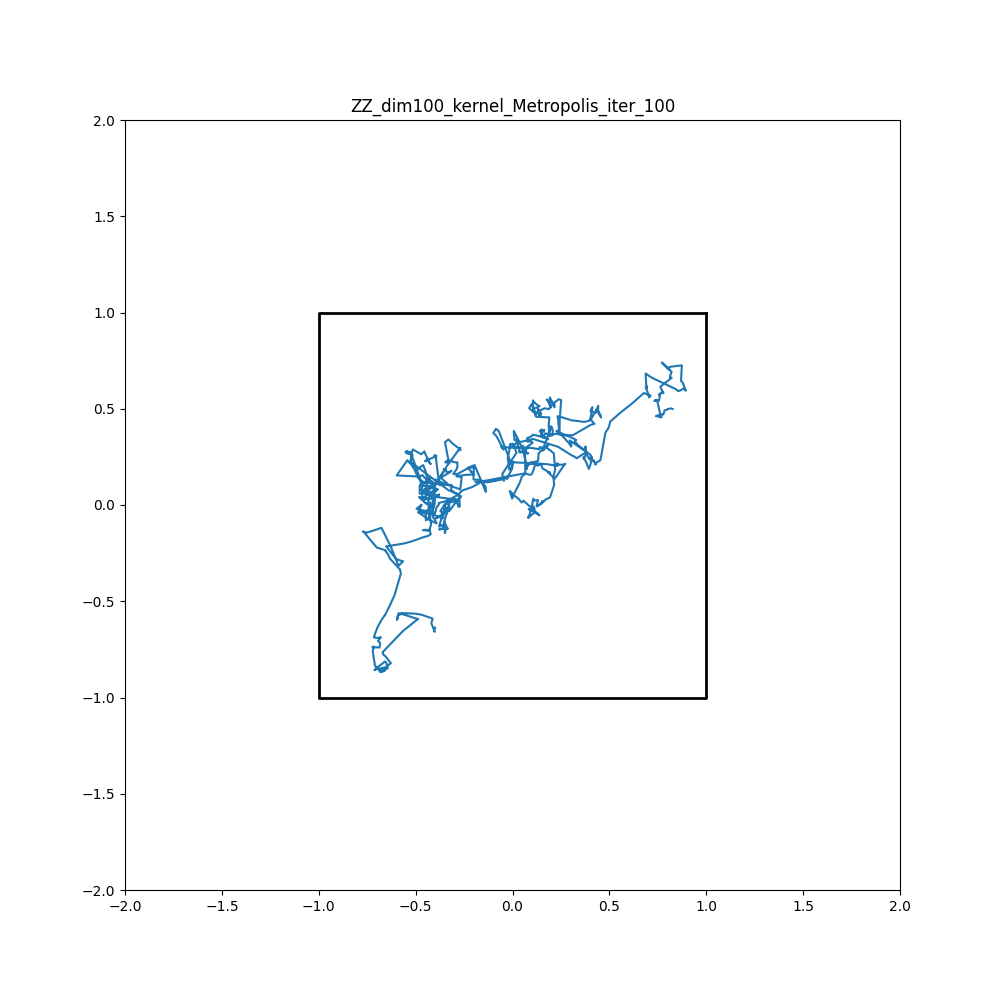} \\
\hline
    \end{tabular}
    \caption{Example trajectories for the Bouncy Particle Sampler (left), Coordinate Sampler (middle) and Zig-Zag Process (right) for simulating from a $100$-dimensional Gaussian distribution restricted to a cube for different transitions on the boundary. We show the dynamics for the first two coordinates only. The different transitions correspond to the limiting behaviour Section \ref{sec:BPSlim}--\ref{sec:ZZlim} (top); using a single Metropolis-Hastings step to sample from $l_x$ (middle); and using 100 Metropolis-Hastings steps to sample from $l_x$ (bottom).}
    \label{fig:d100-trajectories}
\end{figure}

\section{Discussion}

This paper focuses on PDMP-based MCMC samplers to sample densities which are only piecewise smooth. In particular, we presented a general framework for showing invariance of a given target, and then specialise to the case of the common PDMP samplers, namely the Bouncy Particle Sampler, Coordinate Sampler and Zig-Zag sampler when the target is piecewise smooth. Our general framework avoids the general functional-theoretic approach of establishing a given set of functions is a core \cite{Ethier1986, Durmus2021}. Rather, we make use of specific properties of the PDMP processes which we are interested in.

When the target $\pi$ possesses discontinuities, we found that PDMP-based samplers display a surprisingly rich set of behaviours at the boundary, as evidenced by our empirical results, which demonstrate that the choice of jump kernel at the boundary is crucial. We see that the \textit{limiting kernels} compare favourably to Metropolis--Hastings-based jump kernels. For the three samplers we considered, we saw that in our examples the BPS was the best-performing, but other the algorithms may also have more opportunities to be optimized.

We briefly discuss now relationships with the recent and parallel work of \cite{koskela2020zigzag}, in particular Theorem~1 of the most recent preprint version \cite{koskela2020zigzag}. This theorem focuses on the Zig-Zag sampler on a collection of spaces, also with boundaries. The overall approach, assumptions and result statement are similar to ours, since we are all making common use of the framework of \cite{Davis1993}. However, \cite{koskela2020zigzag} is ultimately interested in inference on phylogenetic trees, whereas we are more interested in exploring the specific boundary behaviours of currently popular PDMP samplers. 

There remain several avenues for future exploration. We believe it is possible to weaken Assumptions~\ref{ass:space-informal} and \ref{ass:pi-interior}. For example Assumption~\ref{ass:space-informal}\ref{ass:informal-finite-union} could be relaxed to allow for a countable union of smooth parts; it should be possible to remove Assumption~\ref{ass:space-informal}\ref{ass:informal-projection-tangent} using Sard's theorem; and Assumption~\ref{ass:pi-interior}\ref{ass:pi-C1} could be relaxed to: for all $x\in U_k$, for all $v\in \mathcal{V}^k$, $t\rightarrow \pi_k(x+tv)$ is absolutely continuous. Our chosen set of Assumptions~\ref{ass:space}, \ref{ass:space-v} are sufficient to allow an application of integration by parts, but a simpler and more transparent set of sufficient assumptions would also be desirable. 
Finally, we conjecture that \textit{nonlocal moves} into PDMP samplers, for example based on \cite{Wang2021} or otherwise, might also be useful in boosting convergence in the presence of significant discontinuities.

\section*{Acknowledgements}
This work was supported by EPSRC grants EP/R018561/1 and EP/R034710/1.
\bibliographystyle{plain}
\bibliography{references}

\begin{thebibliography}{10}

\bibitem{NIPS2015_8303a79b}
Hadi~Mohasel Afshar and Justin Domke.
\newblock Reflection, refraction, and hamiltonian monte carlo.
\newblock In C.~Cortes, N.~Lawrence, D.~Lee, M.~Sugiyama, and R.~Garnett,
  editors, {\em Advances in Neural Information Processing Systems}, volume~28,
  pages 3007--3015, 2015.

\bibitem{bierkens2016non}
Joris Bierkens.
\newblock {Non-reversible Metropolis-Hastings}.
\newblock {\em Statistics and Computing}, 26(6):1213--1228, 2016.

\bibitem{BIERKENS2018148}
Joris Bierkens, Alexandre Bouchard-Côté, Arnaud Doucet, Andrew~B. Duncan,
  Paul Fearnhead, Thibaut Lienart, Gareth Roberts, and Sebastian~J. Vollmer.
\newblock {Piecewise deterministic Markov processes for scalable Monte Carlo on
  restricted domains}.
\newblock {\em Statistics and Probability Letters}, 136:148--154, 2018.

\bibitem{Bierkens2019}
Joris Bierkens, Paul Fearnhead, and Gareth Roberts.
\newblock {The Zig-Zag process and super-efficient sampling for Bayesian
  analysis of big data}.
\newblock {\em The Annals of Statistics}, 47(3):1288--1320, 2019.

\bibitem{bierkens2020boomerang}
Joris Bierkens, Sebastiano Grazzi, Kengo Kamatani, and Gareth Roberts.
\newblock The boomerang sampler.
\newblock In {\em International Conference on Machine Learning}, pages
  908--918. PMLR, 2020.

\bibitem{Bierkens2021}
Joris Bierkens, Sebastiano Grazzi, Frank van~der Meulen, and Moritz Schauer.
\newblock {Sticky PDMP samplers for sparse and local inference problems}, 2021.
\newblock \url{http://arxiv.org/abs/2103.08478}.

\bibitem{Bouchard-Cote2018}
Alexandre Bouchard-C{\^{o}}t{\'{e}}, Sebastian~J. Vollmer, and Arnaud Doucet.
\newblock {The Bouncy Particle Sampler: A Nonreversible Rejection-Free Markov
  Chain Monte Carlo Method}.
\newblock {\em Journal of the American Statistical Association},
  113(522):855--867, 2018.

\bibitem{chevallier2020reversible}
Augustin Chevallier, Paul Fearnhead, and Matthew Sutton.
\newblock {Reversible Jump PDMP Samplers for Variable Selection}, 2020.
\newblock \url{http://arxiv.org/abs/2010.11771}.

\bibitem{Davis1993}
M.~H.~A. Davis.
\newblock {\em {Markov Models and Optimization}}.
\newblock Monographs on Statistics and Applied Probability. Springer Science
  and Business Media, 1993.

\bibitem{deligiannidis2018randomized}
George Deligiannidis, Daniel Paulin, Alexandre Bouchard-C{\^o}t{\'e}, and
  Arnaud Doucet.
\newblock {Randomized Hamiltonian Monte Carlo as scaling limit of the bouncy
  particle sampler and dimension-free convergence rates}.
\newblock {\em Annals of Applied Probability}, 2018.
\newblock To appear; available at \url{https://arxiv.org/abs/1808.04299}.

\bibitem{diaconis2000analysis}
Persi Diaconis, Susan Holmes, and Radford~M Neal.
\newblock {Analysis of a nonreversible Markov chain sampler}.
\newblock {\em Annals of Applied Probability}, pages 726--752, 2000.

\bibitem{dinh2017probabilistic}
Vu~Dinh, Arman Bilge, Cheng Zhang, and Frederick A.~IV Matsen.
\newblock {Probabilistic Path Hamiltonian Monte Carlo}.
\newblock In {\em 34th International Conference on Machine Learning, ICML
  2017}, volume~3, pages 1690--1704. International Machine Learning Society
  (IMLS), 2017.

\bibitem{dunson2020hastings}
David~B Dunson and JE~Johndrow.
\newblock The {H}astings algorithm at fifty.
\newblock {\em Biometrika}, 107(1):1--23, 2020.

\bibitem{Durmus2021}
Alain Durmus, Arnaud Guillin, and Pierre Monmarch{\'{e}}.
\newblock {Piecewise deterministic Markov processes and their invariant
  measures}.
\newblock {\em Ann. Inst. H. Poincar{\'{e}} Probab. Statist.},
  57(3):1442--1475, 2021.

\bibitem{Ethier1986}
Stewart~N. Ethier and Thomas~G. Kurtz.
\newblock {\em {Markov Processes: Characterization and Convergence}}.
\newblock Wiley Series in Probability and Statistics. Wiley, 1986.

\bibitem{fearnhead2018piecewise}
Paul Fearnhead, Joris Bierkens, Murray Pollock, and Gareth~O Roberts.
\newblock {Piecewise deterministic Markov processes for continuous-time Monte
  Carlo}.
\newblock {\em Statistical Science}, 33(3):386--412, 2018.

\bibitem{koskela2020zigzag}
Jere Koskela.
\newblock {Zig-zag sampling for discrete structures and non-reversible
  phylogenetic MCMC}.
\newblock \url{https://arxiv.org/abs/2004.08807v3}, 2020.

\bibitem{nakajima2013bayesian}
Jouchi Nakajima and Mike West.
\newblock Bayesian analysis of latent threshold dynamic models.
\newblock {\em Journal of Business \& Economic Statistics}, 31(2):151--164,
  2013.

\bibitem{Nishimura_2020}
Akihiko Nishimura, David~B Dunson, and Jianfeng Lu.
\newblock {Discontinuous Hamiltonian Monte Carlo for discrete parameters and
  discontinuous likelihoods}.
\newblock {\em Biometrika}, 107(2):365–380, Mar 2020.

\bibitem{Pakman2014}
Ari Pakman and Liam Paninski.
\newblock {Exact Hamiltonian Monte Carlo for truncated multivariate Gaussians}.
\newblock {\em Journal of Computational and Graphical Statistics},
  23(2):518--542, 2014.

\bibitem{peters2012rejection}
E.~A. J.~F. Peters and G.~de~With.
\newblock {Rejection-free Monte Carlo sampling for general potentials}.
\newblock {\em Physical Review E}, 85(2):026703, 2012.

\bibitem{raftery1986bayesian}
Adrian~Elmes Raftery and VE~Akman.
\newblock Bayesian analysis of a {P}oisson process with a change-point.
\newblock {\em Biometrika}, 73:85--89, 1986.

\bibitem{terenin2018piecewise}
Alexander Terenin and Daniel Thorngren.
\newblock A piecewise deterministic {M}arkov process via {$(r,\theta)$} swaps
  in hyperspherical coordinates, 2018.
\newblock \url{http://arxiv.org/abs/1807.00420}.

\bibitem{vanetti2017piecewise}
Paul Vanetti, Alexandre Bouchard-C{\^o}t{\'e}, George Deligiannidis, and Arnaud
  Doucet.
\newblock {Piecewise-deterministic Markov chain Monte Carlo}, 2017.
\newblock \url{http://arxiv.org/abs/1707.05296}.

\bibitem{Wang2021}
Andi~Q. Wang, Murray Pollock, Gareth~O. Roberts, and David Steinsaltz.
\newblock {Regeneration-enriched Markov processes with application to Monte
  Carlo}.
\newblock {\em The Annals of Applied Probability}, 31(2):703--735, 2021.

\bibitem{wu2020coordinate}
Changye Wu and Christian~P Robert.
\newblock {Coordinate sampler: a non-reversible Gibbs-like MCMC sampler}.
\newblock {\em Statistics and Computing}, 30(3):721--730, 2020.

\bibitem{zhou2020mixed}
Guangyao Zhou.
\newblock {Mixed Hamiltonian Monte Carlo for Mixed Discrete and Continuous
  Variables}.
\newblock In {\em 34th Conference on Neural Information Processing Systems
  (NeurIPS 2020)}, Vancouver, 2020.

\end{thebibliography}

\newpage
\appendix

\section{Theorem~\ref{th:intAf}}
\subsection{Precise assumptions}

\begin{assumption}
    \label{ass:space}
    For each $k \in K$:
    \begin{enumerate}[label=(\roman*)]
        \item $\dim_H((\overline{U_k})^{\mathrm{o}}\setminus U_k) \leq d_k - 2$, where $\dim_H$ is the Hausdorff dimension.
        \item We assume that there is a finite collection $\{W_1^k,\dots ,W_l^k\}$ of disjoint open sets, and a second collection of disjoint open sets, $\{\Omega_1^k,\dots ,\Omega_l^k\}$, where each $W_i^k, \Omega_i^k \subset \R^{d_k-1}$ with $W_i^k \subset \overline{ W_i^k} \subset \Omega_i^k$ for each $i=1,\dots l$.
        \item The boundaries satisfy $\dim \partial W_i^k \le d_k-2$.
        \item Furthermore we assume that we have $C^1$ (injective) embeddings $\phi_i^k :\Omega_i^k \to \R^{d_k}$, and also have continuous normals $n_i:\Omega_i^k \to S^{d_k-1}$. 
        \item Set $M_i := \phi(\overline{ W_i^k})$, for each $i=1,\dots ,l$. Then we have $\partial U_k = M_1 \cup \dots \cup M_l$.
        \item The intersections satisfy $\dim M_i \cap M_j\le d-2$ for any $i\neq j$.
    \end{enumerate}
\end{assumption}
Let 
\begin{equation}
    N_k = \{ x \in \partial U_k | \exists ! i \text{ such that } \exists u \in W_i \text{ such that } \phi(u) = x\}
    \label{eq:def-Nk}    
\end{equation}
be the set of points of $\partial U_k$ for which the normal $n(x) = n(u)$ is well-defined. Since $\dim_H((\overline{U_k})^{\mathrm{o}}\setminus U_k) \leq d_k - 2$, for all points for which the normal exists, the boundary separates $U_k$ and $\mathbb{R}^{d_k}\setminus U_k$, and does not correspond to an ``internal" boundary of $U_k$ that is removed. By convention, we assume that $n(x)$ is the outer normal.

\begin{assumption}
    \label{ass:space-v}
    We make the following assumptions: for all $v\in \mathcal{V}$, there is a refinement $\{W_1^v,...,W_m^v\}$ and $\{\Omega_1^v,...,\Omega_m^v\}$ of the boundary decomposition such that: 
    \begin{enumerate}[label=(\roman*)]
        \item This new decomposition satisfies the previous assumptions.
        \item For all $0 < i \leq m$, there exists $j \in \{1,\dots,l\}$ such that $\phi(W_i^v) \subset \phi(W_j)$.
        \item for all $x \in W_i^v$, $y \in W_j$ such that $\phi(x) = \phi(y)$, then $n_i(x) = n_i(y)$.
        \item for each $0 < i \leq m$, $\dim_H p_v(M_i^v)\leq n-2$ where $M_i^v=\{\phi_i(x):x\in \overline W_i^v \text{ and } \langle v, n_i(x)\rangle =0\}$ and $p_v$ is the orthogonal projection on $H_v=\mathrm{span}(v)^{\perp}$. 
    	\item for all $ 0 < i \leq m$ and all $x\in\R^n$, the sets $M_i^{v+}\cap(x+\R v)$ and $M_i^{v-}\cap(x+\R v)$ have at most one element where  $M_i^{v+}=\{x=\phi(y)\in M_i: \langle n_i(y), v \rangle >0\}$ and $M_i^{v-}=\{x=\phi(y)\in M_i: \langle n_i(y), v\rangle <0\}$.
    \end{enumerate}

\end{assumption}

\subsection{Proof of Theorem~\ref{th:intAf} }
\label{subsec:proof-th-inth}
    We abuse notations and write $\partial_v f$ the derivative at $t=0$ of $f(x+tv,v)$ with respect to $t$ for a fixed $v$, which corresponds to the term $\Xi f$. 

\subsubsection{Integrability}
We give two integrability lemmas that will be useful for the following proof.

\begin{lemma}
Let $f \in \mathcal{D}(A)$. We know that:
\begin{enumerate}[label=(\roman*)]
    \item $f$ is bounded.\label{AppA_i}
    \item $Af$ is bounded.\label{AppA_ii}
    \item $f \in \mathcal{D}(\textswab{A})$ and $A f = \textswab{A} f$.\label{AppA_iii}
\end{enumerate}
\end{lemma}
\begin{proof}
    \ref{AppA_i} and \ref{AppA_ii} are immediate since $\mathcal D(A)\subset \mathcal{B}_0 \subset \mathcal B(E)$, and $A:\mathcal D(A) \to \mathcal{B}_0\subset \mathcal B(E) $.
    
    \ref{AppA_iii} follows from the fact that for $f\in \Dom(A)$ there is the Dynkin formula which exactly implies $C_t^f$ is a true martingale (\cite[(14.13)]{Davis1993}), hence also a local martingale.
\end{proof}

\begin{lemma}\label{lem:L1(mu)}
$\lambda(k,x,v)(Q f(k,x,v) - f(k,x,v)) \in L_1(\mu)$ and $\partial_v f(k,x,v)\in L_1(\mu)$
\end{lemma}
\begin{proof}
Since $Af$ is bounded, $Af \in L_1(\mu)$. Since $f$ is bounded and with \ref{ass:lambda-integrable} of Assumption~\ref{ass:pi-interior}, $\lambda(k,x,v)(Q f(k,x,v) - f(k,x,v)) \in L_1(\mu)$.
Hence, $\partial_v f(k,x,v)\in L_1(\mu)$.
\end{proof}
This means that we can treat each term independently.

\subsubsection{Integration of the infinitesimal generator over $E$}
\begin{proposition}
    \label{prop:integration-by-part-on-positions}
    Let $f\in \mathcal{D}(A)$ and let $k \in K$. For all $v\in \mathcal{V}^k$
    we have:
    \begin{equation*}
        \begin{split}
              \int_{U_k} \pi_k(x) \partial_v f(k,x,v)\dif x=-\int_{U_k} f(k,x,v) \partial_v \pi_k(x) \dif x\dif v   \\
              + \int_{\partial U_k\cap N_k}\pi_k(x) f(k,x,v)\,|\langle n(x), v\rangle|\dif\sigma(x),
        \end{split}
    \end{equation*}
where $\sigma$ is the Lebesgue measure of the boundary (seen as a Riemannian manifold).
\end{proposition}

\begin{proof}
    We would like to use an integration by parts result on the integral in question, which is precisedly detailed in Appendix~\ref{app:subsec:int}.
    
    So now let $v \in \mathcal{V}^k$.
    Assumptions~\ref{ass:space} and \ref{ass:space-v} imply that $U_k$ satisfies Assumption~\ref{ass:integration-geo} of Appendix~\ref{app:subsec:int}. Furthermore, $f\in \mathcal D(A) \subset \mathcal{D}(\textswab{A})$, thus for all $x\in U_k$, $t\rightarrow f(k,x+tv,v)$ is absolutely continuous. Finally, Lemma~\ref{lem:L1(mu)} implies that $\partial_v f(k,x,v) \in L_1(\mu)$, hence we can use  
    Proposition~\ref{prop:integration-by-part-general} of Appendix~\ref{app:subsec:int} on $U_k$ with the function $z\mapsto f(z) \pi(z)$. 
    In this context, for $x \in \partial U_k$,
    if $\langle n(x),v \rangle > 0$,
    \begin{align*}
        \pi_k(x^-) f(k,x^-,v) &= \lim_{t\uparrow 0}\pi_k(x+tv) f(k,x+tv,v) = \pi_k(x) f(k,x,v), \\
        \pi_k(x^+) f(k,x^+,v) &= 0;
    \end{align*}
    otherwise if $\langle n(x),v \rangle > 0$,
    \begin{align*}
        \pi_k(x^-) f(k,x^-,v) &= 0, \\
        \pi_k(x^+) f(k,x^+,v) &= \lim_{t\downarrow 0}\pi_k(x+tv) f(k,x+tv,v) = \pi_k(x) f(k,x,v).
    \end{align*}
    This yields:
    \begin{align*}
        \int_{U_k}\partial_v \pi_k(x) f(k,x,v)\dif x &=-\int_{U_k} f(k,x,v) \partial_v \pi_k(x) \dif x\dif v\\
        &+ \int_{\partial U_k\setminus C_v}\pi_k(x) f(k,x,v)\, \langle n(x), v \rangle \dif\sigma(x), \\
    \end{align*}
    where we removed the absolute value around $\langle n(x), v \rangle$ to account for the sign difference of $\pi_k(x^-) f(k,x^-,v) - \pi_k(x^+) f(k,x^+,v)$.
	Furthermore, $\partial U_k\setminus C_v \subset \partial U_k \cap N_k$ and these two sets differs by a set of zero measure. 
	Hence:
	\begin{equation*}
	    \begin{split}
	        	    \int_{U_k}\partial_v \pi_k(x) f(k,x,v)\dif x=-\int_{U_k} f(k,x,v) \partial_v \pi_k(x) \dif x\dif v\\
	        	    + \int_{\partial U_k\cap N_k}\pi_k(x) f(k,x,v)\, \langle n(x), v\rangle\dif\sigma(x),
	    \end{split}
	\end{equation*}
which concludes the proof.
    
\end{proof}

\begin{lemma}
\begin{align*}
    \int_{E_k} \partial_v f(k,x,v) &\mu(k,x,v) \dif x \dif v 
    = -\int_{U_k\times \mathcal V^k} f(k,x,v) \nabla_x \mu \cdot v  \dif x \dif v\\
    &+ \int_{(\partial U_k \cap N_k)\times \mathcal V^k} f(k,x,v) \mu(k,x,v) \langle v, n(x) \rangle \dif\sigma(x) \dif v
\end{align*}
\label{lemma:1rst-part}
\end{lemma}
\begin{proof}
Since $\partial_v f$ is in $L_1(\mu)$, $\partial_v f(k,x,v) \mu(k,x,v)$ is integrable. Hence by Fubini's theorem:
\[
    \int_{U_k\times\mathcal{V}^k} \partial_v f(k,x,v) \mu(k,x,v) \dif x\dif v = \int_{\mathcal{V}^k}\int_{U_k} \partial_v f(k,x,v) \mu(k,x,v) \dif x\dif v.
\]
Using Proposition~\ref{prop:integration-by-part-on-positions}:
\begin{align*}
    \int_{U_k\times\mathcal{V}^k} \partial_v f(k,x,v) \mu(k,x,v) &\dif x\dif v 
    = - \int_{\mathcal{V}^k}\int_{U_k} f(k,x,v) (\nabla \pi_k(x) \cdot v) p_k(v) \dif x\dif v \\
    &+ \int_{\mathcal V^k}\int_{\partial U_k \cap N_k} f(k,x,v) \mu(k,x,v) \langle v, n(x) \rangle \dif\sigma(x) \dif v.
\end{align*}
Using \ref{ass:pi-dot-integrable} from Assumption~\ref{ass:pi-interior}, $f(k,x,v) (\nabla \pi_k(x) \cdot v) p_k(v)$ is integrable, and $f(k,x,v) \mu(k,x,v)$ is bounded. Hence we can use Fubini a second time on both terms to get the result.
\end{proof}

\section{Theorem: integration by parts}
\label{app:subsec:int}
\begin{assumption}[informal geometrical assumption]
    \label{ass:ipp-geo-informal}
    Let $U$ be an open set in $\R^n$ such that:
    \begin{itemize}
        \item $\bar{U} = \R^n$
        \item the boundary $\partial U$ can be decomposed as a finite union of smooth closed sub-manifolds with piecewise $C^1$ boundaries in $\R^n$,
        \item for any $x,v \in \R^n$, the intersection $\partial U \cup \{x+\R v\}$ is finite (not taking into account the points where $v$ is tangent to $\partial U$),
        \item $\dim_H(N_v) \leq n-2$ where $N_v$ is the subset of $\partial U$ where the normal is ill-defined,
        \item $\dim_H p_v(M^v)\leq n-2$ where $M^v=\{x\in \partial U \text{ such that } \langle v,n(x) \rangle=0\}$ and $p_v$ is the orthogonal projection on $H_v=v^{\perp}$,
    \end{itemize}
    with $\dim_H$ the Hausdorff dimension.
\end{assumption}
These assumptions are made precise in Assumption~\ref{ass:integration-geo} of the next section.

\begin{proposition}
\label{prop:integration-by-part-general}
Let $U$ be an open set of $\mathbb{R}^n$ satisfying Assumption~\ref{ass:integration-geo}, and $\partial U$ be its boundary. 
Let $f$ and $g$ be measurable functions from $U$ to $\R$ such that:
\begin{enumerate}
    \item $f$ is bounded;
    \item for any sequence $(y_n)\subset U$ with $\|y_n\|\rightarrow\infty$, $\lim_{n\to\infty}g(y_n)=0$;
    \item for each $x,v\in \R^{n}$, the functions $t\mapsto f(x+tv)$ and $t\mapsto g(x+tv)$ are absolutely continuous on  $U\cap(x+\R v)$ and $\partial_t f,\partial_t g\in L^1(U)$.
\end{enumerate}
Fix $v\in \R^n$. Then, using the convention 
\[
	f(x+) = \lim_{t\downarrow 0}f(x+tv) \text{ and } f(x-) = \lim_{t\uparrow 0}f(x+tv)
\]
and 
\[
	g(x+) = \lim_{t\downarrow 0}g(x+tv) \text{ and } g(x-) = \lim_{t\uparrow 0}g(x+tv)
\]
we have:
\begin{equation*}
    \begin{split}
        	\int_U\partial_v f(x) \, g(x) \dif x = \int_{\partial U\setminus N_v}(g(x^-)f(x^-)-f(x^+)g(x^+))\, |\langle n(x), v \rangle| \dif\sigma(x)
	\\
	-\int_U f(x) \, \partial_v g(x) \dif x,
    \end{split}
\end{equation*}
where the second term is integrated with respect to the Lebesgue measure of the boundary (seen as a Riemannian manifold) and $N_v$ is the set of points where the normal is ill-defined. 
\end{proposition}
\begin{proof}
The proof is a corollary of the next section.
\end{proof}

\subsection{Integration over open domain}

\begin{assumption}
    \label{ass:integration-geo}
    Let $U$ be an open set in $\R^n$ such that for each $v\in S^{n-1}$ (the unit sphere in $\R^{n}$),
    \begin{enumerate}
    	\item there exist $W_1\subset\overline W_1\subset\Omega_1,\dots,W_k\subset\overline W_k\subset\Omega_k$, $2k$ open sets in $\R^{n-1}$ with $\dim_H\partial W_i\leq n-2$ (these open sets may depend on $v$ because of \ref{ass:intersection-M-unique} of this assumption).\label{ass:integration-geo1}
    	\item there exist $\phi_i:\Omega_i\rightarrow\R^n$, $i=1,\dots k$, $C^1$ one to one maps such that the differential $D\phi_i(x)$ is one to one for all $x\in\Omega_i$. It implies that there is a  continuous normal $n_i:\Omega_i\rightarrow S^{n-1}$.\label{ass:integration-geo2} 
    	\item $\partial U=M_1\cup\dots\cup M_k$  where the sets $M_i=\phi_i(\overline W_i)$ are closed,\label{ass:integration-geo3}
    	\item $\dim_H M_i\cap M_j\leq n-2$ for all $i\neq j$.\label{ass:integration-geo4}
    	\item  Let $W_i^0=\{x\in  W_i : v\cdot n_i(x)=0\})$. For each $i$, $\dim_H p_v(M_i^0)\leq n-2$ where $M_i^0=\phi_i(W_i^0)$ and $p_v$ is the orthogonal projection on $H=H_v=v^{\perp}$.\label{ass:integration-geo5} 
    	\item Let $W_i^+=\{x\in W_i\in M_i:n_i(x)\cdot v>0\}$ and $W_i^-=\{x\in W_i:n_i(x)\cdot v<0\}$. For all $i$ and all $y\in\R^n$, the sets $M_i^{+}\cap(y+\R v)$ and $M_i^{-}\cap(y+\R v)$ have at most one element where  $M_i^{+}=\phi_i(W_i^+)$ and $M_i^{-}=\phi_i(W_i^-)$.
    	\label{ass:intersection-M-unique}
    \end{enumerate}
\end{assumption}

\begin{assumption}
    \label{ass:integration-fun}
    Let $f:U\rightarrow\R$ be a measurable function such that
    \begin{enumerate}
    	\item[i.] 
    	for each $y,v\in \R^{n}$, the function $t\mapsto f(y+tv)$ is absolutely continuous on every bounded interval $I$ such that $y+I v \subset U$
    	\label{ass:absolute-continuity}
    	\item[ii.] $\lim_{\|y\|\rightarrow\infty}f(y)=0$. That is, for any sequence $(y_n)\subset U$ with $\|y_n\| \to \infty$.
    	\item[iii.] If $U$ is not bounded, then for each $v\in \R^{n}$, $\partial_v f \in L^1(U)$.
    	\label{ass:f-derivative-bounded}
    \end{enumerate}
\end{assumption}

We extend $f$ to $\R^{n}\setminus \partial U$ with $f(y)=0$ for every $y\notin \overline U$. So that we can suppose that $U=\R^n\setminus \partial U$.

\begin{theorem} \label{thm:integrationbyparts}
    Let $U$ be an open set of $\mathbb R^n$ satisfying Assumption \ref{ass:integration-geo}, for some fixed $v \in \R^n$ with $\| v\|=1$.
	Let 
	\[
	\mathcal N_v=\big(\bigcup_{i=1}^k(\phi_i(\partial W_i)\cup M_i^0)\big)\bigcup\big(\bigcup_{1\leq i<j\leq k} M_i\cap M_j)\big)
	\]
	the set on which normals are ill-defined.
	Then for any $f$ satisfying Assumption \ref{ass:integration-fun}:
	\begin{enumerate}
	    \item $\dim_H \mathcal N_v\leq 2$ and $\dim_H p_v(\mathcal N_v)\leq n-2$;
	    \item for each $y\in\partial U\setminus \mathcal N_v$, the limits 
	\[
	\lim_{t\downarrow 0}f(y+tv)=f(y+) \text{ and } \lim_{t\uparrow 0}f(y+tv)=f(y-)
	\] 
	exist;
	   \item The normal $n(y)$ is well-defined at each point $y\in\partial U\setminus \mathcal N_v$ and   
	   \[
	\int_U\partial_v f(y)\dif y=\int_{\partial U\setminus \mathcal N_v}(f(y^-)-f(y^+))\,|n(y)\cdot v|\dif\sigma(y),
	\]
	where $\sigma$ is the Lebesgue measure on $\partial U$.
	\end{enumerate}
\end{theorem}
We can use the theorem with a product $f=g\pi$ where 
\begin{itemize}
	\item 
$g:U\rightarrow\R$ is measurable, bounded,  absolutely continuous on each sets $U\cap(y+\R v)$, $y\in\R^n$, and $\partial_t g\in L^1(U)$, 
\item $\pi:U\rightarrow\R$ is in $C^1(U)\cap L^1(U)$, bounded with bounded derivatives, $\lim_{\|y\|\rightarrow\infty}\pi(y)=0$, and the derivative $\partial_v\pi\in L^1(U)$.
\end{itemize}

\subsection{Proof of Theorem \ref{thm:integrationbyparts} }
For the first point, let
 \begin{align*}
 	\mathcal N_v&=\big(\bigcup_{i=1}^k(\phi_i(\partial W_i)\cup M_i^0)\big)\bigcup\big(\bigcup_{1\leq i<j\leq k} M_i\cap M_j)\big),\\
\mathcal N&=p_v^{-1}(p_v(\mathcal N_v)) \text{ and } V=H\setminus p_v(\mathcal N_v).
\end{align*}
By Assumptions \ref{ass:integration-geo}.\ref{ass:integration-geo1}, \ref{ass:integration-geo}.\ref{ass:integration-geo4} and \ref{ass:integration-geo}.\ref{ass:integration-geo5}, $\dim_H(p_v(\mathcal N_v))\leq n-2$, therefore $p_v(\mathcal N_v)$ has zero $H$-Lebesgue measure.

For the second point, the fact that $f(y+)$ and $f(y-)$ exist is a direct consequence of Asssumption \ref{ass:integration-fun}.\ref{ass:absolute-continuity}.

Finally we consider the third point. For all $z\in H$ denote
\[
	E(z)=\{t\in \R:z+tv\in\partial U\}.
\] 
By Assumption \ref{ass:integration-geo}.\ref{ass:intersection-M-unique}, the set $E(z)$ has $2k$ elements at most for all $z\in V$.
Set 
\begin{align*}
U'=U\setminus\mathcal N.
\end{align*}
Since $\dim_H(p_v(\mathcal N_v))\leq n-2$, it follows that $\dim_H \mathcal N\leq n-1$. Therefore,
\[
\int_U\partial_vf(y)\dif y=\int_{U'}\partial_vf(y)\dif y.
\]
By Fubini's theorem,
\begin{align*}
\int_{U'}\partial_v f(y)\dif y&=\int_V\left(\int_{\R\setminus  E(z)}\partial_tf(z+tv)\dif t\right)\dif z.
\end{align*}
By Assumption~\ref{ass:integration-fun}, for almost all $z\in V$ and for each connected component $(a,b)$ of $\R\setminus E(z)$,
\[
\int_{(a,b)}\partial_v f(z+tv)\dif t=f((z+tb)^-)-f((z+ta)^+).
\]
Taking into account that $\lim_{t\rightarrow\pm\infty}f(z+tv)=0$, we obtain 
\begin{align*}
	\int_{U}\partial_tf(y)\dif y
=\int_V\sum_{t\in E(z)}(f((z+tv)^-)-f((z+tv)^+))\dif z. 
\end{align*}
Now we want to see that the latter integral is equal to 
\[\int_{\partial U\setminus \mathcal N_v}(f(y^-)-f(y^+))\,|n(y)\cdot v|\dif\sigma(y).
\]
Set 
\begin{align*}
 I&=\{1,\dots,k\}\times\{+,-\},\\  
 J(z)&=\{(i,s)\in I:\exists t\in E(z), z+tv\in M_i^s\} \text{ for } z\in H,\\
 V_J&=\{z\in V: J(z)=J\} \text{ for } J\subset I, \text{ and }\\
 V_i^s&=p_v(M_i^s) \text{ for } (i,s)\in I.
\end{align*}
By Assumption \ref{ass:integration-geo}.\ref{ass:intersection-M-unique}, for each $(i,s)\in I$, the map $p_v\circ \phi_i:W_i^s\rightarrow V_i^s$ is a bijection, so that we can define the map $F_i^s:V_i^s\rightarrow M_i^s$  by $F_i^s(z)=\phi_i((p_v\circ \phi_i)^{-1}(z))$. 
Using the definition of the set $V=H\setminus\mathcal N$, we see that  for each $z\in V$ and each $t\in E(z)$, there exists $(i,s)\in I$ unique such that $z+tv=F_i^s(z)\in M_i^s$. Furtheremore, for each $(i,s)\in I$, $V_i^s\cap V=\cup_{J\ni (i,s)}V_J$. It follows that
\begin{align*}
    \int_U\partial_t f(y)\dif y&=\int_V\sum_{t\in E(z)}(f((z+tv)^-)-f((z+tv)^+))\dif z\\
    &=\sum_{J\subset I}\int_{V_J} \sum_{t\in E(z)}(f((z+tv)^-)-f((z+tv)^+))\dif z\\
    &=\sum_{J\subset I}\int_{V_J} \sum_{(i,s)\in J}(f(F_i^s(z)^-)-f(F_i^s(z)^+))\dif z\\
    &=\sum_{(i,s)\in I}\sum_{J\ni (i,s)} \int_{V_J}(f(F_i^s(z)^-)-f(F_i^s(z)^+))\dif z\\
    &=\sum_{(i,s)\in I} \int_{V_i^s\cap V}(f(F_i^s(z)^-)-f(F_i^s(z)^+))\dif z.
\end{align*}
Since the differential of each $\phi_i$ is always one to one and since $D\phi_i(x)(u).u$ is never orthogonal to $v$ for $x\in W_i^s $ and $u\neq 0$, the local inverse function theorem implies that the maps $F_i^s$ are $C^1$. Furthermore, the image of $F_i^s$ is $M_i^s$ and the normal to $M_i^s$ at $y=F_i^s(z)$ is $n(y)=\pm n_i((p_v\circ \phi)^{-1}(z))$. Therefore,
\[
\int_{V_i^s\cap V}(f(F_i^s(z)^-)-f(F_i^s(z)^+))\dif z=\int_{M_i^s\setminus\mathcal N_v}(f(y^-)-f(y^+))|n(y)|\dif\sigma(y).
\]
Finally, since $\partial U\setminus \mathcal N_v=\cup_{(i,s)\in I}(M_i^s\setminus \mathcal N_v)$,
\begin{align*}
    \sum_{(i,s)\in I} \int_{V_i^s\cap V}(f(F_i^s(z)^-)-f(F_i^s(z)^+))\dif z=\int_{\partial U\setminus \mathcal N_v}(f(y^-)-f(y^+))|n(y)|\dif\sigma(y).
\end{align*}

\section{Validity of transition kernels derived from limiting behaviour}
\label{app:Q_proofs}

\subsection{Bouncy Particle Sampler: Proof of Proposition~\ref{prop:BPS} }

For the specified $Q_{\BPS}$, we first derive the form of the associated probability kernel on velocities, $Q'_x$, remembering that $Q'_x$ is obtained by flipping the velocity and then applying the transition defined by $Q_{\BPS}$. This gives
\[
Q'_x(v'|v)=\left\{\begin{array}{cl}
1 &
v'= -v \mbox{ and } v\in\mathcal{V}_x^-,\\
\pi_{k_1(x)}(x)/\pi_{k_2(x)}(x) & v'= -v \mbox{ and } v\in\mathcal{V}_x^+,\\
1-\pi_{k_1(x)}(x)/\pi_{k_2(x)}(x) & v'= -v + 2 \langle n,v\rangle n \mbox{ and } v\in\mathcal{V}_x^+.
\end{array}
\right.
\]

The transition kernel $Q'_x$ allows for two possible transitions, either $v'=- v$  or $v'=v - 2 \langle n,v\rangle n$. These transitions have the following properties:
\begin{enumerate}[label=(\roman*)]
\item In the first case if $v \in
\mathcal{V}_x^+$ then $v' \in 
\mathcal{V}_x^-$, and vice versa. While for the second case if $v\in \mathcal{V}_x^+$ then $v'\in \mathcal{V}_x^+$. 
\item For either transition, $\langle v',v'\rangle=\langle v,v \rangle$, and $|\langle n,v'\rangle|=|\langle n,v\rangle|$. Furthermore by the spherical symmetry of $p(v)$ for the Bouncy Particle Sampler the first of these means that $p(v)=p(v')$.\label{appC:ii}
\end{enumerate}

We need to show that $l_x(v')=\sum Q'_x(v'|v)l_x(v)$, where 
\[
   l_x(v) =\left\{ \begin{array}{cl} |\langle n,v \rangle| p(v) \pi_{k_2(x)}(x) & \forall v\in \mathcal{V}_x^+,\\
    |\langle n,v \rangle| p(v) \pi_{k_1(x)}(x) & \forall v\in \mathcal{V}_x^-. \end{array} \right.
\]
We will show this holds first for $v'\in \mathcal{V}_x^+$ and then for $v'\in \mathcal{V}_x^-$.

If $v'\in \mathcal{V}_x^+$ then there are two possible transitions, from $v=-v' \in \mathcal{V}_x^-$ and from $v\in \mathcal{V}_x^+$ where $v=-v'+2\langle v',n \rangle n$. Let $v^*=-v'+2\langle v',n \rangle n$
\begin{eqnarray*}
\lefteqn{\sum Q'_x(v'|v)l_x(v)
= 1\cdot\left( |\langle n,-v' \rangle| p(-v') \pi_{k_1(x)}(x)\right)}\\ & & + 
\left(1-\frac{\pi_{k_1(x)}(x)}{\pi_{k_2(x)}(x)}\right)\left( 
|\langle n,v^* \rangle| p(v^*) \pi_{k_2(x)}(x)
\right),\\
&=&  |\langle n,v' \rangle| p(v') \left(\pi_{k_1(x)}(x) + \left(1-\frac{\pi_{k_1(x)}(x)}{\pi_{k_2(x)}(x)}\right) \pi_{k_2(x)}(x) \right).
\end{eqnarray*}
Where the last equality comes from applying Property~\ref{appC:ii} of the transition. The last expression simplifies to $l_x(v')$ as required.

For $v'\in \mathcal{V}_x^-$ we have only one transition and thus
\[
\sum Q'_x(v'|v)l_x(v) = \left(\frac{\pi_{k_1(x)}(x)}{\pi_{k_2(x)}(x)} \right)  \left( |\langle n,-v' \rangle| p(-v') \pi_{k_2(x)}(x)\right),
\]
which, again using Property~\ref{appC:ii}, is $l_x(v')$. 

\subsection{Coordinate Sampler: Proof of Proposition~\ref{prop:CS} }

We follow a similar argument to that of the previous section. First we write down the form of $Q'_x$ derived from $Q_{\CS}$:
\[
Q'_x(v'|v)=\left\{\begin{array}{cl}
1 &
v'= -v \mbox{ and } v\in\mathcal{V}_x^-,\\
\pi_{k_1(x)}(x)/\pi_{k_2(x)}(x) & v'= -v \mbox{ and } v\in\mathcal{V}_x^+,\\
(1-\pi_{k_1(x)}(x)/\pi_{k_2(x)}(x))\frac{\langle v',n\rangle}{K} & v' \in\mathcal{V}_x^+ \mbox{ and } v\in\mathcal{V}_x^+,
\end{array}
\right.
\]
where $K=\sum_{v\in\mathcal{V}_x^+} |\langle n,v \rangle|$.

For the Coordinate Sampler, $p(v)=1/(2d)$ for each of the possible values for $v$. Now we need to show $l_x(v')=\sum Q'_x(v'|v)l_x(v)$. We will consider the case $v'\in \mathcal{V}_x^+$ and $v'\in\mathcal{V}_x^-$ separately. For the latter case, the argument is the same as for the Bouncy Particle Sampler. Thus we just present the case for $v'\in \mathcal{V}_x^+$.
\begin{eqnarray*}
\lefteqn{\sum Q'_x(v'|v)l_x(v)
= 1\cdot\left( |\langle n,-v' \rangle| p(-v') \pi_{k_1(x)}(x)\right)}\\ & & + 
\langle v',n \rangle \left(1-\frac{\pi_{k_1(x)}(x)}{\pi_{k_2(x)}(x)}\right)
\sum_{v\in\mathcal{V}_x^+} \frac{|\langle n,v \rangle|}{K} p(v) \pi_{k_2(x)}(x)
\\
&=& p(v') \langle v',n \rangle 
\left( \pi_{k_1(x)}(x) + \left(1-\frac{\pi_{k_1(x)}(x)}{\pi_{k_2(x)}(x)}\right)
\pi_{k_2(x)}(x)\sum_{v\in\mathcal{V}_x^+} \frac{|\langle n,v \rangle|}{K}
\right) \\
&= & p(v') \langle v',n \rangle  \left( \pi_{k_1(x)}(x) + \left(1-\frac{\pi_{k_1(x)}(x)}{\pi_{k_2(x)}(x)}\right)
\pi_{k_2(x)}(x)
\right).
\end{eqnarray*}
The second equality comes from the fact that $p(v)$ is constant for all $v\in\mathcal{V}$. The third equality comes from the definition of $K$. 

\subsection{Zig-Zag Sampler: Proof of Proposition \ref{prop:ZZ} }

In the following we will set $k=1$ for implementing the algorithm that defines $Q_{\ZZ}$.
We need to show that $Q'_x$ keeps $l_x$ invariant. We will prove this by showing the following stronger detailed balance condition holds:
\[
l_x(v)Q'_x(v'|v)=l_x(v')Q'_x(v|v'), \quad\forall v,v'\in \mathcal{V}.
\]
As $Q'_x(v'|v)=Q_{\ZZ}(v'|-v)$, then writing the detailed balance condition for pairs $-v$ and $v'$ we have that it suffices to show
\[
l_x(-v)Q_{\ZZ}(v'|v)=l_x(v')Q_{\ZZ}(-v|-v'), \quad\forall v,v'\in \mathcal{V}.
\]

By a slight abuse of notation let $k(v)=k_1(x)$ if $v\in\mathcal{V}^-_x$ and $k(v)=k_2(x)$ if $v\in\mathcal{V}^+_x$. Then we can write $l_x(v)=|\langle n,v \rangle|p(v)\pi_{k(v)}(x)$. Thus using the fact that $p(v)$ defines a uniform distribution on $\mathcal{V}$, the detailed balance condition simplifies to
\[
|\langle n,v \rangle| \pi_{k(-v)}(x)Q_{\ZZ}(v'|v)=
|\langle n,v' \rangle| \pi_{k(v')}(x)Q_{\ZZ}(-v|-v'), \quad\forall v,v'\in \mathcal{V}.
\]
This can be viewed as matching the probability we have a velocity $v$ and transition to $v'$ with one where we flip the velocities: starting at $-v'$ and transition to $-v$.

We show that the detailed balance condition holds separately for different combinations of whether $v\in\mathcal{V}_x^+$ or $v\in\mathcal{V}_x^-$ and whether $v'\in\mathcal{V}_x^+$ or $v'\in\mathcal{V}_x^-$.

First assume $v\in\mathcal{V}_x^+$ and $v'\in\mathcal{V}_x^-$. This corresponds to a trajectory that is moving from the lower to the higher density region, but that reflects off the boundary and stays in the lower density region. It is straightforward to see that the events that change the velocity only increase $\langle n,v \rangle$, the speed at which the trajectory moves through the boundary region to the higher density region. Thus  a transition from $\mathcal{V}_x^+$ to $\mathcal{V}_x^-$ is impossible, and $Q_{\ZZ}(v'|v)=0$. Similarly, $-v'\in\mathcal{V}_x^+$ and $-v\in\mathcal{V}_x^-$ so $Q_{\ZZ}(-v|-v')=0$. Hence the detailed balance conditions trivially hold in this case.

Next assume $v\in\mathcal{V}_x^-$ and $v'\in\mathcal{V}_x^+$. This corresponds to a trajectory that is moving from the higher to the lower density region, but that reflects off the boundary and stays in the higher density region. In this case $k(-v)=k(v')$ and thus the detailed balance condition becomes
\begin{equation} \label{eq:DB-+}
|\langle n, v\rangle| Q_{\ZZ}(v'|v) =
|\langle n, v' \rangle| Q_{\ZZ}(-v|-v'), \quad \forall v\in\mathcal{V}_x^-,~v'\in\mathcal{V}_x^+.
\end{equation}
To prove the detailed balance condition holds we will first obtain an expression for $Q_{\ZZ}(v'|v)$, and then introduce a coupling between a transition for $v$ to $v'$ and one from $-v'$ to $-v$ to link it to a similar expression for $Q_{\ZZ}(-v|-v')$. 

The randomness in the algorithm that defines $Q_{\ZZ}$ only comes through the randomness of the event times simulated in step~\ref{ZZ(a)} of Section~\ref{sec:ZZlim}. Remember that $\tau_i$ is the time at which component $i$ of the velocity would switch, if the trajectory is still within the boundary region. Each $\tau_i$ is (conditionally) independent of the others, and has an exponential distribution with rate $\max\{0, -n_iv_i\}$, where $n_i$ is the component of the $i$th coordinate of the unit normal $n$. If $n_iv_i\geq0$, then $\tau_i=\infty$.

It is helpful to introduce three sets of components.
\begin{itemize}
    \item Let $\mathcal{S}_1$ be the set of components $i$ such that $v'_i=v_i$ and $n_iv_i<0$.
   \item Let $\mathcal{S}_2$ be the set of components $i$ such that $v'_i=-v_i$. 
    \item Let $\mathcal{S}_3$ be the set of components $i$ such that $v'_i=v_i$ and $n_iv_i\geq0$.
\end{itemize}
So $\mathcal{S}_1$ is the set of components of the velocity $v$ that are moving the particle towards the low-density region, and are unchanged by the transition to $v'$; $\mathcal{S}_2$ is the set of components that flip during the transition from $v$ to $v'$; and $\mathcal{S}_3$ is the set of components of the velocity $v$ that are moving the particle towards the high-density region, and are unchanged by the transition to $v'$.

Only components $i$ of the velocity for which $n_iv_i<0$ can change during the transition from $v$ to $v'$. This means that if there exists $i\in\mathcal{S}_2$ such that $n_iv_i \geq 0$ then the transition from $v$ to $v'$ is impossible. By the same argument, the transition from $-v'$ to $v$ is impossible. Thus in this case $Q_{\ZZ}(v'|v)=Q_{\ZZ}(-v|-v')=0$ and detailed balance trivially holds. So in the following we will assume that $n_iv_i<0$ for $i \in\mathcal{S}_2$. 

By a similar argument we have that the set $\mathcal{S}_1$ is the set of indices of the velocity that could change during the transition from $v$ to $v'$, but did not. Whereas $\mathcal{S}_3$ are the set of indices of the velocity that could never have changed during the transition.

To ease notation let $m=|\mathcal{S}_2|$, the number of indices in set $\mathcal{S}_2$, and note that $m\geq 1$ as $v\neq v'$. Without loss of generality we can relabel the coordinates so that $\mathcal{S}_2=\{1,\ldots,m\}$, and we will use $\tau_{1:m}$ to denote the vector of event times for the coordinates in $\mathcal{S}_2$.

We now introduce a function of time, $t$, that depends on $\tau_{1:n}$. This is
\[
h(t;\tau_{1:n})=
\sum_{i=1}^{m} n_iv_i(t - \max\{0,t-\tau_i\})
+\sum_{i \in \mathcal{S}_1} n_iv_i t
+ \sum_{i \in \mathcal{S}_3} n_iv_i t.
\] 
This can be viewed as the net distance travelled by the trajectory up to time $t$ in the direction of the normal $n$, given that only velocity coordinates in $\mathcal{S}_2$ can change, and these change at times $\tau_{1:m}$. This function is important as it determines when the trajectory leaves the boundary region, and determines the termination of the simulation algorithm in step~\ref{ZZ(b)}. As $v\in\mathcal{V}_x^-$ and $v'\in\mathcal{V}_x^+$, and the changes in velocity in the direction of $n$ is monotone as we flip components, we have that $h(t;\tau_{1:n})$ is strictly decreasing at $t=0$, strictly increasing for large enough $t$ and is unimodal. As $h(0;\tau_{1:m})=0$, this means that there is a unique $t^*(\tau_{1:m})>0$ such that $h(t^*(\tau_{1:m});\tau_{1:m})=0$. This is the exit time from the boundary region calculated in step~\ref{ZZ(b)} of the algorithm.

We can now define the set $\mathcal{T}$ of values of $\tau_{1:m}$ that are consistent with a transition from $v$ to $v'$. The conditions are that all components of the velocity must flip before $t^*$, and that the trajectory must not pass through the boundary region -- see the other stopping criteria in step~\ref{ZZ(b)} of the algorithm. This gives us that
\[
\mathcal{T}=
\left\{
\tau_{1:m} : \tau_i \leq t^*(\tau_{1:m}),~i=1,\ldots,m; ~~\min_{0<t<t^*(\tau_{1:m})} h(t;\tau_{1:n}) > -C
\right\}.
\]
The probability of a transition from $v$ to $v'$ is thus the probability $\tau_{1:m}\in\mathcal{T}$ times the probability that $\tau_i>t^*(\tau_{1:m})$ for $i\in\mathcal{S}_1$. As each $\tau_i$, $i \in \mathcal{S}_1$ or $i\in\mathcal{S}_2$, has an independent exponential distribution with rate $-n_iv_i$,
\[
Q_{\ZZ}(v'|v) =\int_\mathcal{T} \left(\prod_{i\in\mathcal{S}_1} \exp\{n_iv_it^*(\tau_{1:m})\} \right)\left(
\prod_{i=1}^m (-n_iv_i)\exp\{n_iv_i\tau_i\}
\dif \tau_{1:m}
\right).
\]

Now consider the reverse transition, from $-v'$ to $-v$. Under our existing definitions $\mathcal{S}_1$, $\mathcal{S}_2$ and $\mathcal{S}_3$, we have that $\mathcal{S}_2$ is still the set of indices that the flip for the transition from $-v'$ to $-v$, but now $\mathcal{S}_1$ is the set of components of the velocity that could never have flipped, while $\mathcal{S}_3$ is the set of components that could have flipped but did not. 

We can define the same quantities for the reverse transition from $-v'$ to $-v$. We will use tildes to denote quantities that relate to this transition. So $\tilde{\tau}_{1:m}$ will be the vector of flip times for components in $i\in\mathcal{S}_2$. We have
\[
\tilde{h}(t;\tilde{\tau}_{1:n})=
\sum_{i=1}^{m} n_iv_i(t - \max\{0,t-\tilde{\tau}_i\})
-\sum_{i \in \mathcal{S}_1} n_iv_i t
- \sum_{i \in \mathcal{S}_3} n_iv_i t,
\]
using the fact that $-v'_i=v_i$ for $i \in \mathcal{S}_2$ and $-v'_i=-v_i$ otherwise. By the same argument as above, there is a unique $\tilde{t}^*(\tilde{\tau}_{1:m})>0$ such that $\tilde{h}(\tilde{t}^*(\tilde{\tau}_{1:m});\tilde{\tau}_{1:m})=0$. The set of possible values of $\tilde{\tau}_{1:m}$ that are consistent with the transition from $-v'$ to $v$ is
\[
\tilde{\mathcal{T}}=
\left\{
\tilde{\tau}_{1:m} : \tilde{\tau}_i \leq \tilde{t}^*(\tilde{\tau}_{1:m}),~i=1,\ldots,m; ~~\min_{0<t<\tilde{t}^*(\tilde{\tau}_{1:m})} \tilde{h}(t;\tilde{\tau}_{1:n}) > -C
\right\}.
\]

Finally we can write down the transition probability as before, remembering that the rate of flipping for components $i\in\mathcal{S}_2$ is $-n_iv_i$ as before; but for $i\in\mathcal{S}_3$ it is $n_iv_i$. Thus
\begin{eqnarray} 
\lefteqn{Q_{\ZZ}(-v|-v') =} \nonumber\\
& &\left(\prod_{i\in\mathcal{S}_3} \exp\{-n_iv_i\tilde{t}^*(\tilde{\tau}_{1:m}) \right)\left(
\int_{\tilde{\mathcal{T}}}
\prod_{i=1}^m (-n_iv_i)\exp\{n_iv_i\tilde{\tau}_i\}
\dif \tilde{\tau}_{1:m}
\right). \label{eq:Qrev}
\end{eqnarray}

To relate the two transition probabilities, we introduce a coupling between $\tau_{1:m}$ and $\tilde{\tau}_{1:m}$, so
$\tilde{\tau}_{1:m}=g(\tau_{1:m})$, where
\[
\tilde{\tau}_i=g(\tau_{1:m})_i=t^*(\tau_{1:m})-\tau_i.
\]
This coupling is a natural one. If we consider the path through the boundary region given by $\tau_{1:m}$ that transitions from $v$ to $v'$, we can reverse that path to get a path that transitions from $-v'$ to $-v$. For the forward path a flip of component $i$ at time $\tau_i$ occurs at a time $t^*(\tau_{1:m})-\tau_i$ prior to the end of the path. Thus for the reverse path the flip would occur at time $t^*(\tau_{1:m})-\tau_i$.

It is straight forward to show that if $\tilde{\tau}_{1:m}=g(\tau_{1:m})$ then $h(t;\tau_{1:m})=\tilde{h}(t^*(\tau_{1:m})-t;\tilde{\tau}_{1:m})$. This result is intuitive; it is saying the distance of the forward trajectory within the boundary region at time $t$ is equal to the distance of the backward trajectory within the boundary region at time $t^*(\tau_{1:m})-t$. This immediately implies that
$t^*(\tau_{1:m})=\tilde{t}^*(\tilde{\tau}_{1:m})$, the exit time for the forward and backward trajectories are the same. Furthermore, if we consider the second constraint on $\tau_{1:m}$ in the definition of $\mathcal{T}$ then we have
\[
\min_{0<t<t^*(\tau_{1:m})} h(t;\tau_{1:n}) = 
\min_{0<t<\tilde{t}^*(\tilde{\tau}_{1:m})} \tilde{h}(t;\tilde{\tau}_{1:n}),
\]
for $\tilde{\tau}_{1:m}=g(\tau_{1:m})$. Together with the fact that $\tau_i\leq t^*(\tau_{1:m})$ then $\tilde{\tau}_{1:m}\leq \tilde{t}^*(\tilde{\tau}_{1:m})$. We have that the function $g$ maps $\tau_{1:m}\in\mathcal{T}$ to $\tilde{\tau}_{1:m}\in\tilde{\mathcal{T}}$. Furthermore, the function $g$ is invertible, and by similar arguments we have that $g^{-1}$ maps $\tilde{\tau}_{1:m}\in\tilde{\mathcal{T}}$ to $\tau_{1:m}\in\mathcal{T}$. Hence $g$ is a bijection from $\mathcal{T}$ to $\tilde{\mathcal{T}}$.

The function $g$ defines a linear map between $\tau_{1:m}$ and $\tilde{\tau}_{1:m}$. For $\tau_{1:m}\in\mathcal{T}$ we have that, by definition of $t^*(\tau_{1:m})$,
\begin{equation} \label{eq:t*}
\sum_{i=1}^m n_iv_i(2\tau_i-t^*(\tau_{1:m}))+\sum_{i=m+1}^d
n_iv_it^*(\tau_{1:m})=0.
\end{equation}
This gives that
\[
t^*(\tau_{1:m})=\sum_{i=1}^m \left( \frac{-2n_iv_i}{K} \right)\tau_i, \mbox{ where } K= -\sum_{i=1}^m n_iv_i+\sum_{i=m+1}^d n_iv_i.
\]
Furthermore, using that $v'$ is equal to $v$ except that $v_i$ is flipped for $i=1,\ldots,m$, $K=\langle v',n\rangle $.

Let $b_{1:m}$ be the $1\times m$ vector whose $i$th entry is $b_i=2v_in_i/K$. If we let $1_m$ denote the $1\times m$ vector of ones, and $I_m$ the $m\times m$ identity matrix then we have
\[
\tilde{\tau}_{1:m} = g(\tau_{1:m}) = (b_{1:m}1_{1:m}^\top-I_m)\tau_{1:m}=A\tau_{1:m},
\]
where the $m \times m$ matrix $A=(b_{1:m}1_{1:m}^\top-I_m)$. In the following argument we will make the change of variables $\tilde{\tau}_{1:m}=g(\tau_{1:m})=A\tau_{1:m}$, and we will need the determinant of the Jacobian of this transformation. Using the matrix determinant lemma, this is given by
\[
|\mbox{det}(A)|=|(1-1_{1:m}^\top b_{1:m})||\mbox{det}(-I_m)|=\left|
1-\sum_{i=1}^m b_i \right|.
\]
This simplifies to
\[
\left|
1-\sum_{i=1}^m b_i \right|=
\left|
1+\sum_{i=1}^m \frac{2v_in_i}{\langle v',n\rangle}
\right|=
\left|
\frac{\langle v',n\rangle+\sum_{i=1}^m 2v_in_i}{\langle v',n\rangle}
\right|=
\left|
\frac{\langle v,n\rangle}{\langle v',n\rangle}
\right|.
\]

So now, taking the definition of $Q_{\ZZ}(v'|v)$ and applying the change of variables $\tilde{\tau}_{1:m}=g(\tau_{1:m})$ we get
\begin{eqnarray*}
\lefteqn{Q_{\ZZ}(v'|v) = \int_\mathcal{T} \prod_{i\in\mathcal{S}_1} \exp\{n_iv_it^*(\tau_{1:m})\} 
\prod_{i=1}^m \left((-n_iv_i)\exp\{n_iv_i\tau_i\}\right)
\dif \tau_{1:m}}\\ 
&=& \int_{\tilde{\mathcal{T}}}
\exp\left\{\sum_{i\in\mathcal{S}_1} n_iv_i \tilde{t}^*(\tilde{\tau}_{1:m})\right\}
\left|\frac{\langle v,n\rangle}{\langle v',n\rangle}\right|
\prod_{i=1}^m \left((-n_iv_i)\exp\{n_iv_i(\tilde{t}^*(\tilde{\tau}_{1:m})-\tilde{\tau}_i)\}\right)
\dif \tilde{\tau}_{1:m} \\
&=&\left|\frac{\langle v,n\rangle}{\langle v',n\rangle}\right|
\int_{\tilde{\mathcal{T}}}
\left(\prod_{i=1}^m (-n_iv_i)\right)
\exp\left\{\sum_{i\in\mathcal{S}_1} n_iv_i \tilde{t}^*(\tilde{\tau}_{1:m})+ \sum_{i=1}^m n_iv_i(\tilde{t}^*(\tilde{\tau}_{1:m})-\tilde{\tau}_i)\right\}
\dif \tilde{\tau}_{1:m} \\
&=&\left|\frac{\langle v,n\rangle}{\langle v',n\rangle}\right|
\int_{\tilde{\mathcal{T}}}
\left(\prod_{i=1}^m (-n_iv_i)\right)
\exp\left\{\sum_{i\in\mathcal{S}_3} -n_iv_i \tilde{t}^*(\tilde{\tau}_{1:m})+\sum_{i=1}^m n_iv_i \tilde{\tau}_i  \right\}
\dif \tilde{\tau}_{1:m},
\end{eqnarray*}
where the final equality comes from the definition of $\tilde{t}^*(\tilde{\tau}_{1:m})=t^*(\tau_{1:m})$, using (\ref{eq:t*}) after substituting in $\tau_i=\tilde{t}^*(\tilde{\tau}_{1:m})-\tilde{\tau}_i$.

By comparing the final expression with (\ref{eq:Qrev}), we get that
\[
Q_{\ZZ}(v'|v) = \left|\frac{\langle v,n\rangle}{\langle v',n\rangle}\right| Q_{\ZZ}(-v|v'),
\]
which satisfies (\ref{eq:DB-+}) as required.

The final combination involve $v,v'\in\mathcal{V}_x^+$ and $-v',-v\in\mathcal{V}_x^-$, or vice versa. The detailed balance condition in this case becomes
\[
\pi_{k_1(x)}(x)|\langle n, v\rangle| Q_{\ZZ}(v'|v) =
\pi_{k_2(x)}(x)|\langle n, v' \rangle| Q_{\ZZ}(-v|-v'), \quad \forall v,v'\in\mathcal{V}_x^-.
\]
We can show this using a similar argument to above, with the same coupling of paths from $v$ to $v'$ with paths from $v'$ to $v$. The main differences are, first, that the definition of $\mathcal{T}$ is simplified to
\[
\left\{
\tau_{1:m} : \tau_i \leq t^*(\tau_{1:m}),~i=1,\ldots,m
\right\},
\]
as, by monotonicity of the changes in velocity, we do not need to check whether the other exit condition in step~\ref{ZZ(b)} holds. Second, that the definition of $t^*(\tau_{1:m})$ changes, with it being the value of $t$ for which $h(t;\tau_{1:m})=C$. For $\tau_{1:m}\in\mathcal{T}$,this  becomes
\[
\sum_{i=1}^m n_iv_i(2\tau_i-t^*(\tau_{1:m}))+
\sum_{i=m+1}^d n_iv_it^*(\tau_{1:m})=C
\]
due to the different exit condition in step~\ref{ZZ(b)}. We have similar changes to the definitions of $\tilde{\mathcal{T}}$ and $\tilde{t}^*(\tilde{\tau}_{1:m})$.

However we can define $Q_{\ZZ}(v'|v)$ and $Q_{\ZZ}(-v|-v')$ in a similar way. Furthermore we can use the same linear transformation $g$, which is still a bijection between $\mathcal{T}$ and $\tilde{\mathcal{T}}$. Whilst the definition $t^*$ has changed, this only introduces an additive constant into the linear transformation defined by $g$, and thus the Jacobian of the transformation is unchanged. Following the argument above we thus get to the same expression for $Q_{\ZZ}(v'|v)$ after making the change of variables:
\begin{eqnarray*}
\lefteqn{Q_{\ZZ}(v'|v)=\left|\frac{\langle v,n\rangle}{\langle v',n\rangle}\right|
\int_{\tilde{\mathcal{T}}}
\left(\prod_{i=1}^m (-n_iv_i)\right) }\\
&\times&
\exp\left\{\sum_{i\in\mathcal{S}_1} n_iv_i \tilde{t}^*(\tilde{\tau}_{1:m})+ \sum_{i=1}^m n_iv_i(\tilde{t}^*(\tilde{\tau}_{1:m})-\tilde{\tau}_i)\right\}
\dif \tilde{\tau}_{1:m}.
\end{eqnarray*}
Now substituting in our new definition of $\tilde{t}^*(\tilde{\tau}_{1:m})=t^*(\tau_{1:m})$ we get
\[
Q_{\ZZ}(v'|v)=\left|\frac{\langle v,n\rangle}{\langle v',n\rangle}\right| Q_{\ZZ(-v|v'} \exp\{C\},
\]
where the additional factor of $\exp\{C\}$ is due to the different definition of $t^*$. As $C=\log(\pi_{k_2(x)}(x)/\pi_{k_1(x)}(x) )$ this rearranges to
\[
\pi_{k_1(x)}(x)|\langle v',n\rangle|Q_{\ZZ}(v'|v) =
\pi_{k_2(x)}(x)|\langle v,n\rangle|Q_{\ZZ}(-v|v'),
\]
as required.

\section{Additional Simulation Results} \label{sec:add_results}

Figures~\ref{fig:BPS-trajectories}--\ref{fig:trajectories-ZZ} show trajectories for the Bouncy Particle Sampler, the Coordindate Sampler and the Zig-Zag Process for dimensions $d=2$, 10, 100 for the sampling from a Gaussian restricted to a cube. Figure~\ref{fig:trajectories-ZZ2} shows trajectories for the Zig-Zag Process if we use the canonical basis -- in this case the distribution of all coordinates are independent, and the Zig-Zag Process benefits from this by being able to run independent dynamics for each coordinates.

\begin{figure}
    \centering
    \begin{tabular}{|c | c | c | c|}
    \hline
    & dim = 2 & dim = 10 & dim = 100 \\
    \hline
    \rotatebox{90}{BPS Limit} & \includegraphics[width=3.5cm]{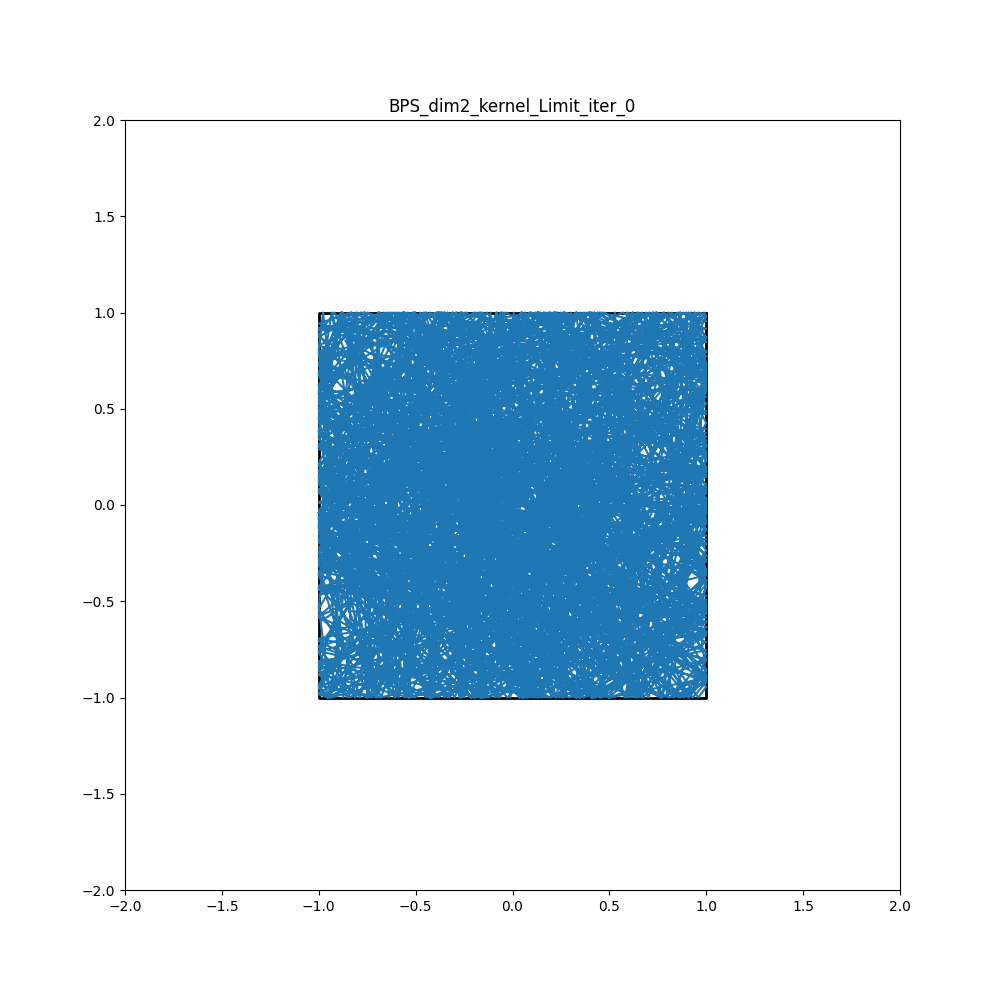}&
    \includegraphics[width=3.5cm]{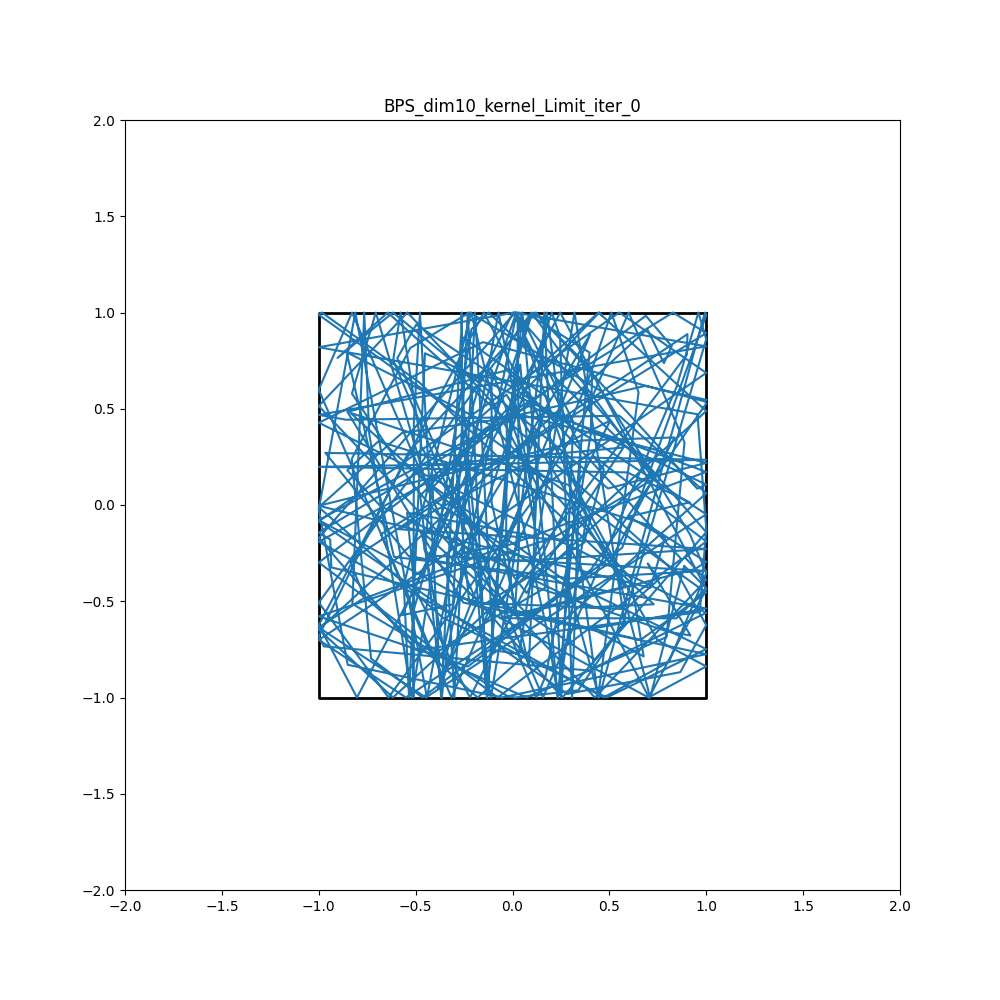} & \includegraphics[width=3.5cm]{BPS_dim100_kernel_Limit_iter_0.png} \\
    \hline
    \rotatebox{90}{BPS Metropolis 1} & \includegraphics[width=3.5cm]{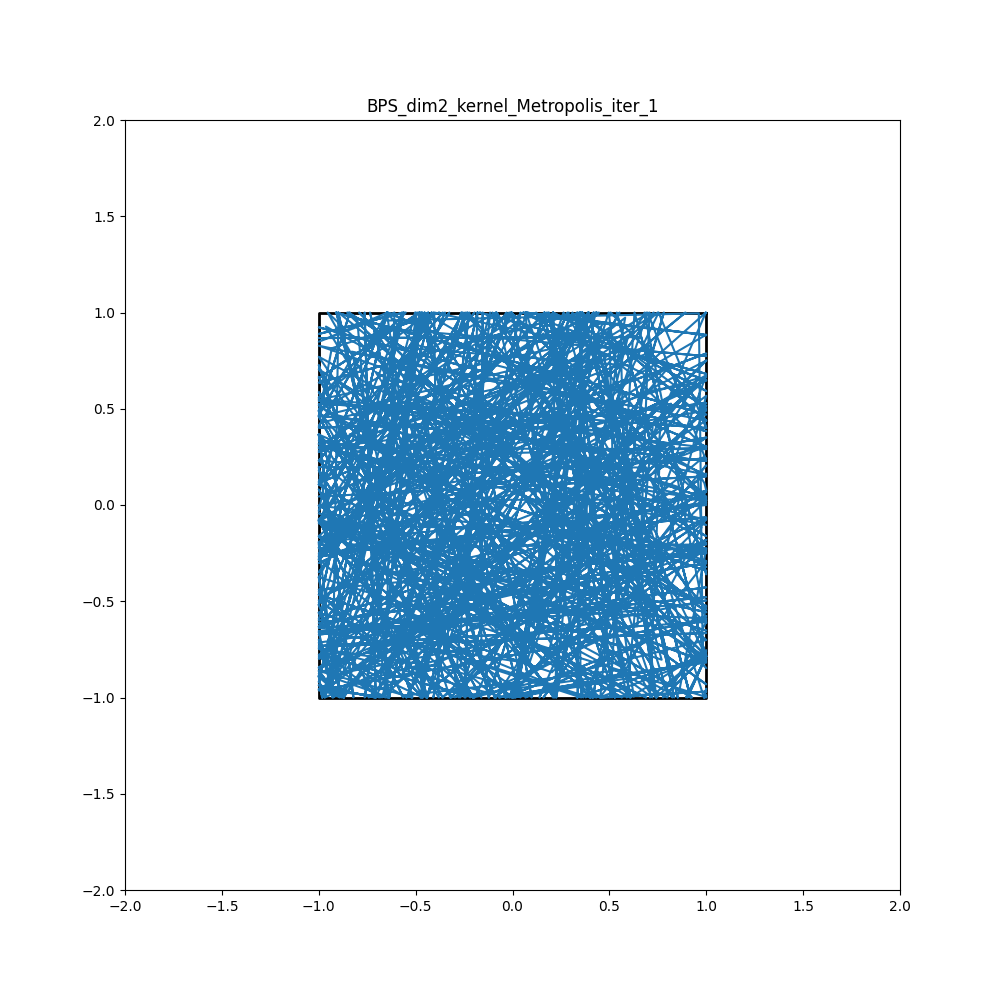}&
    \includegraphics[width=3.5cm]{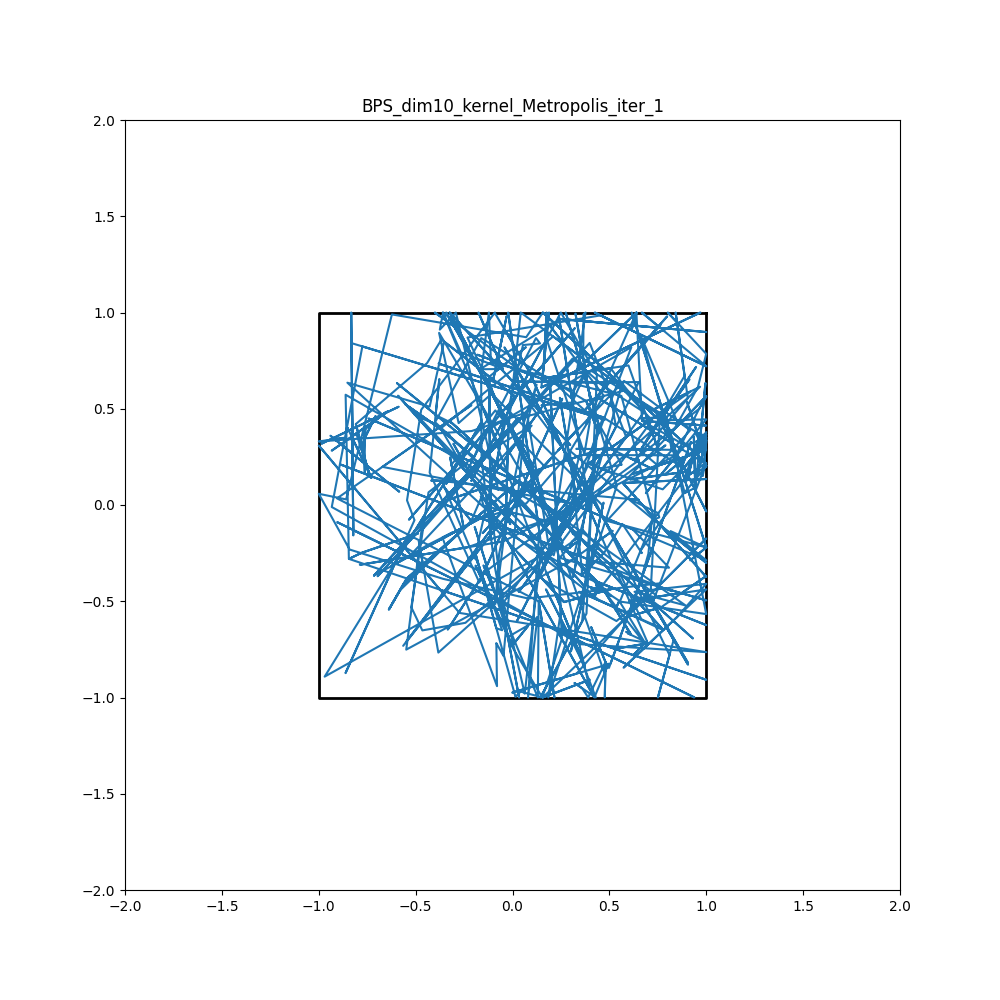} & \includegraphics[width=3.5cm]{BPS_dim100_kernel_Metropolis_iter_1.png} \\
    \hline
    \rotatebox{90}{BPS Metropolis 100} & \includegraphics[width=3.5cm]{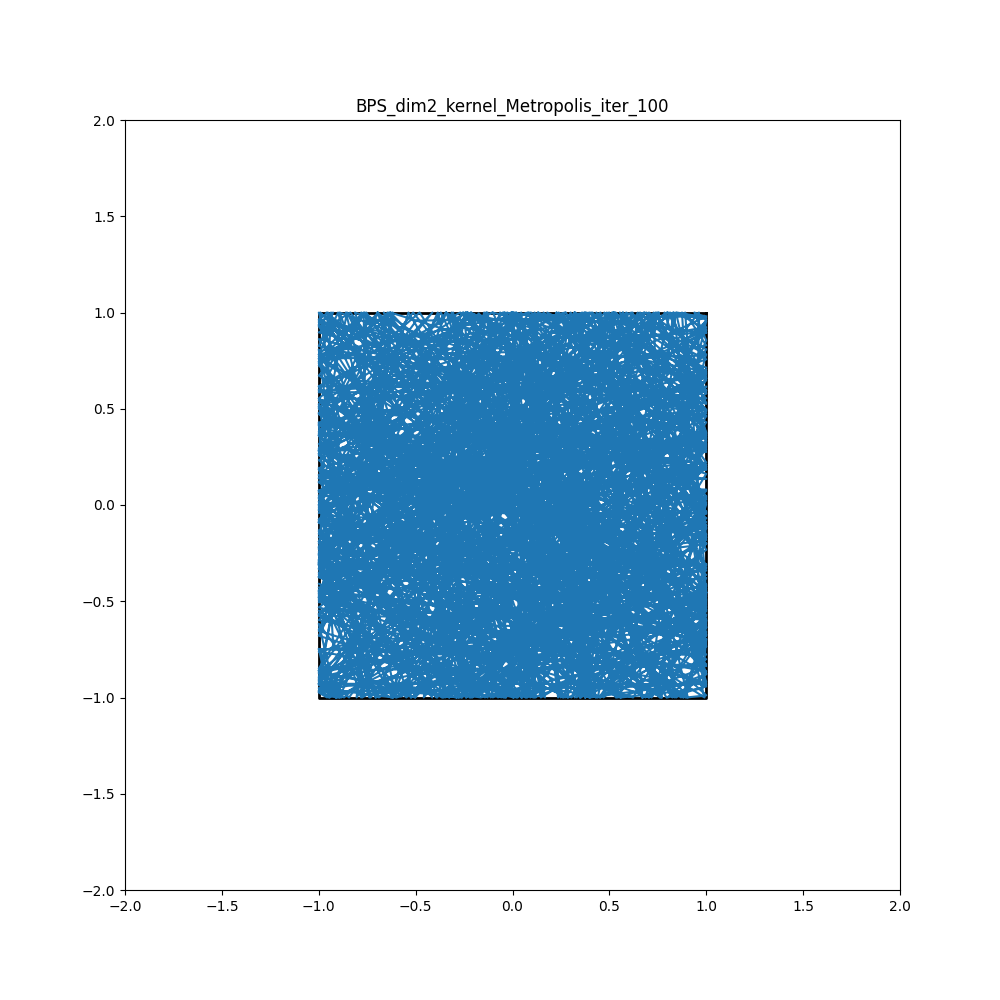}&
    \includegraphics[width=3.5cm]{BPS_dim10_kernel_Metropolis_iter_1.png} & \includegraphics[width=3.5cm]{BPS_dim100_kernel_Metropolis_iter_100.png} \\  
    \hline
    \end{tabular}
    \caption{Example trajectories for the Bouncy Particle Sampler for simulating from a $d$-dimensional Gaussian distribution restricted to a cube, for $d=$2, 10, 100; and for different transitions on the boundary. For $d=10$, 100 we show the dynamics for the first two coordinates only. The different transitions correspond to the limiting behaviour, $Q_{\BPS}$ of Section~\ref{sec:BPSlim} (top); using a single Metropolis-Hastings step to sample from $l_x$ (middle); and using 100 Metropolis-Hastings steps to sample from $l_x$ (bottom).}
    \label{fig:BPS-trajectories}
\end{figure}

\begin{figure}
    \centering
    \begin{tabular}{|c | c | c | c|}
    \hline
    & dim = 2 & dim = 10 & dim = 100 \\
    \hline
    \rotatebox{90}{CS Limit} & \includegraphics[width=3.5cm]{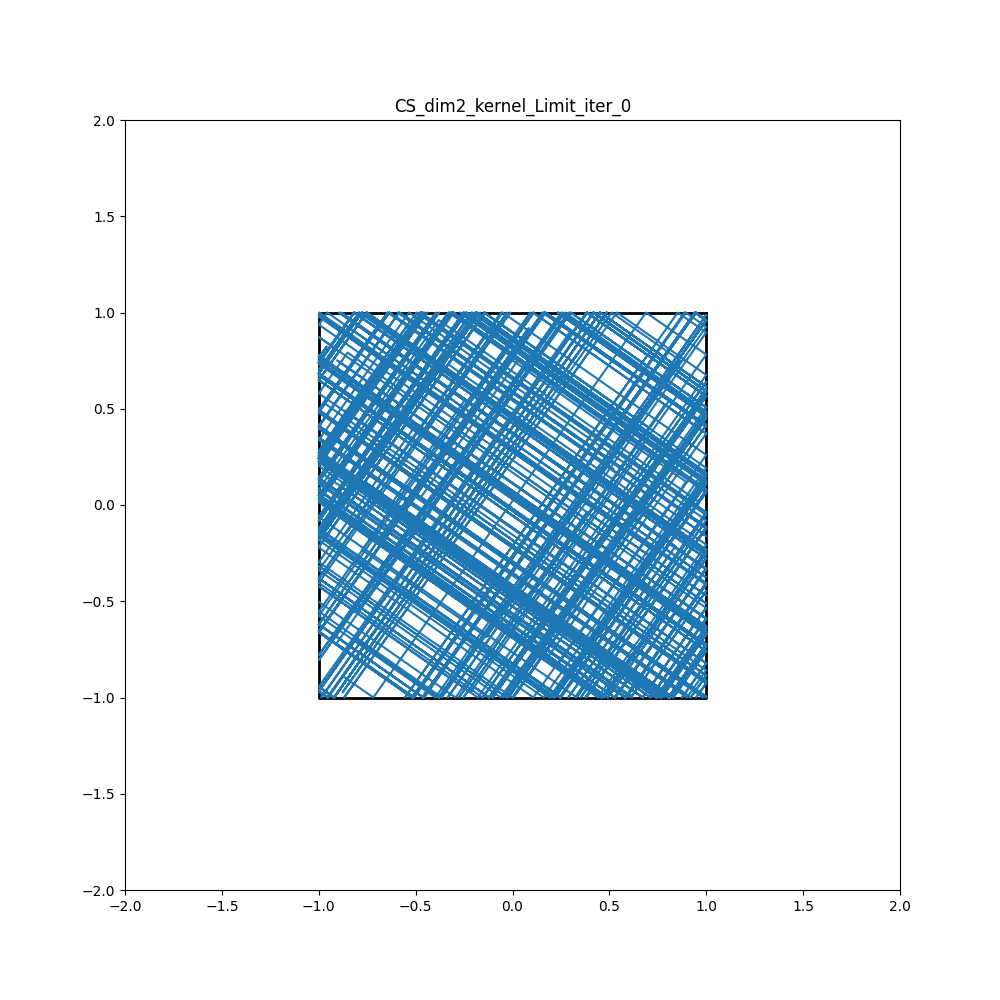}&
    \includegraphics[width=3.5cm]{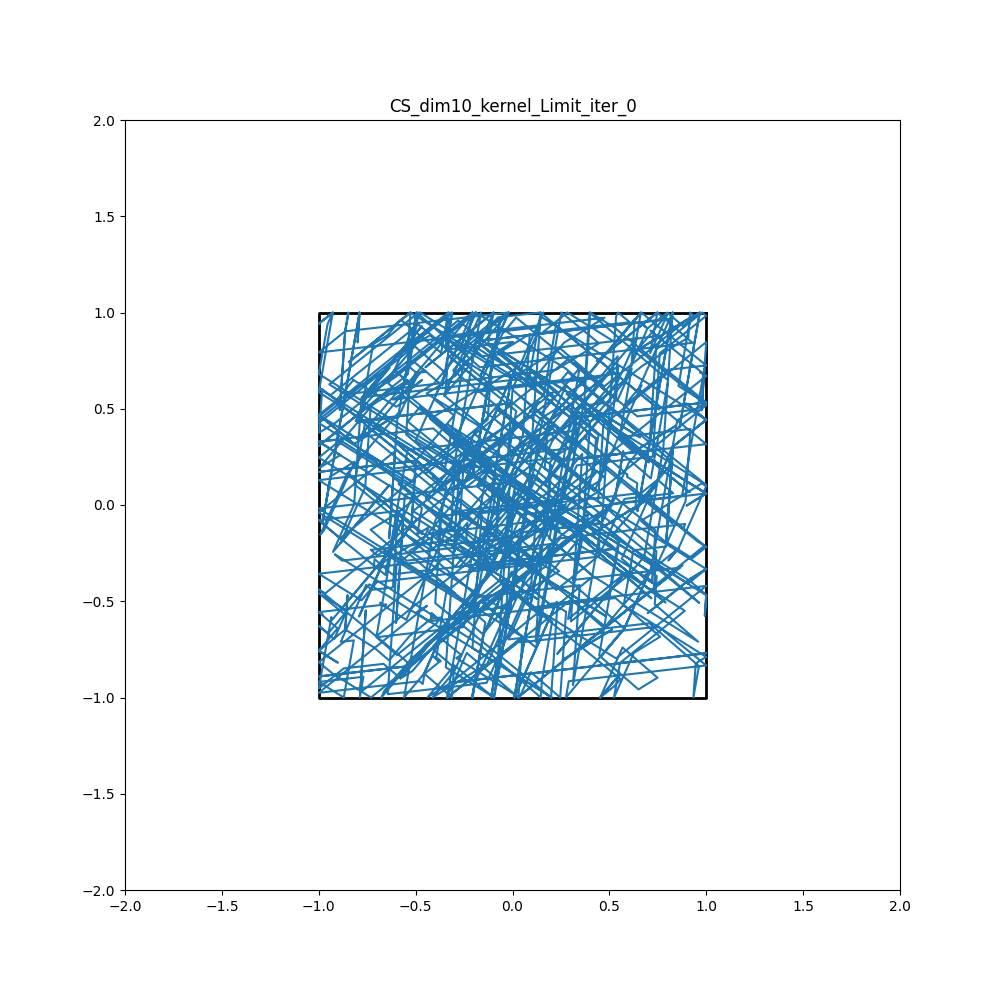} & \includegraphics[width=3.5cm]{CS_dim100_kernel_Limit_iter_0.png} \\
    \hline
    \rotatebox{90}{CS Metropolis 1} & \includegraphics[width=3.5cm]{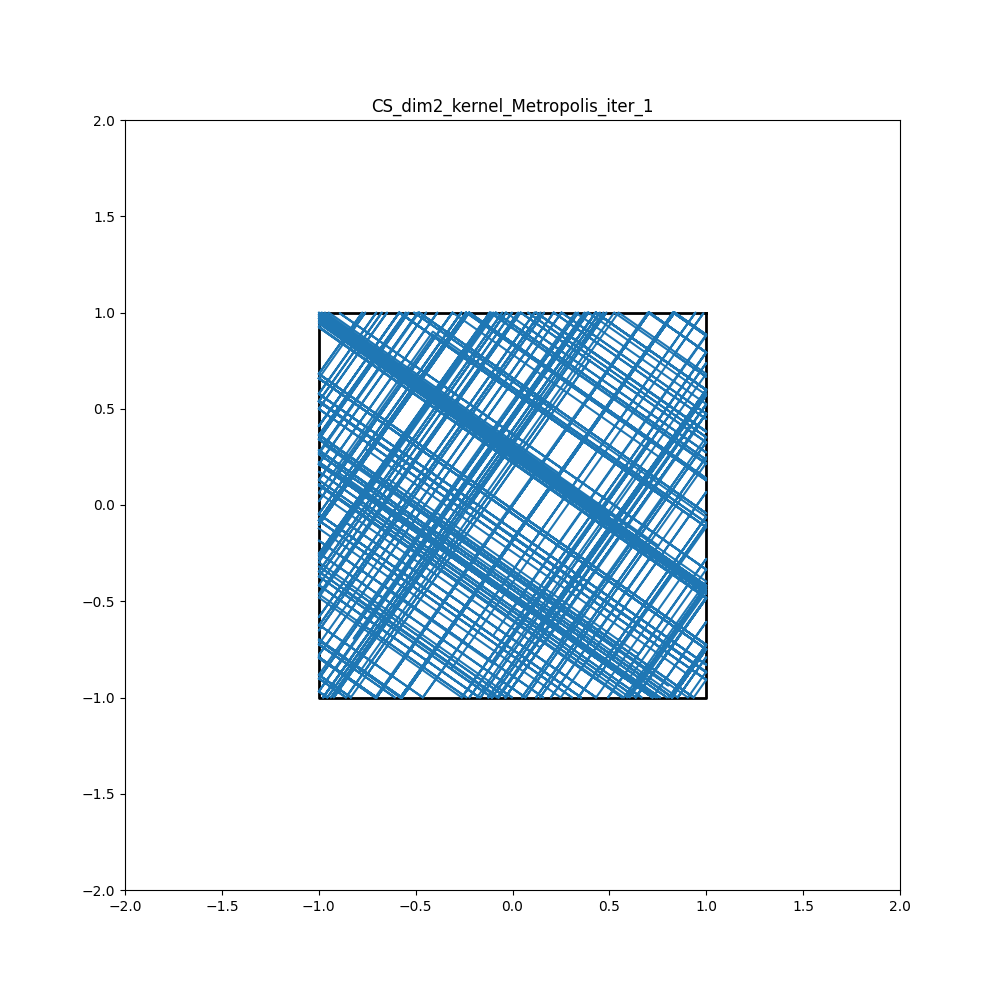}&
    \includegraphics[width=3.5cm]{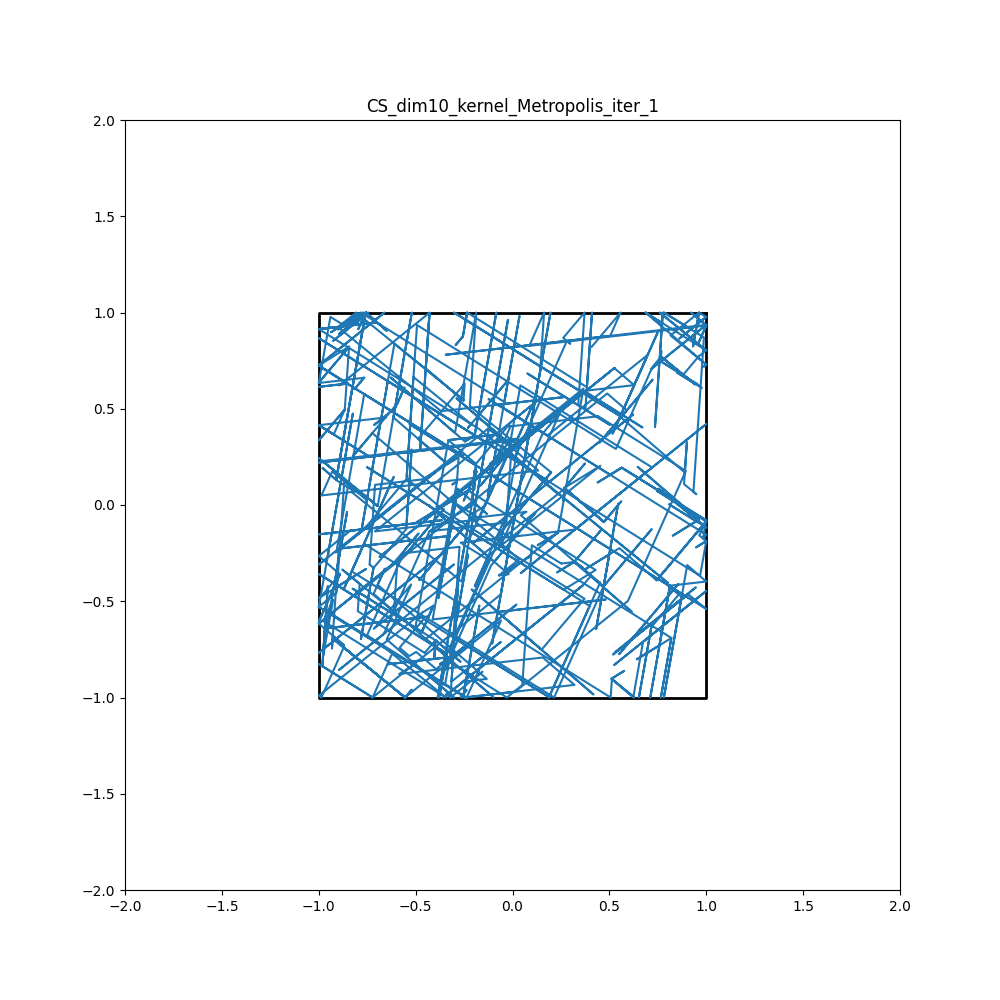} & \includegraphics[width=3.5cm]{CS_dim100_kernel_Metropolis_iter_1.png} \\
    \hline
    \rotatebox{90}{CS Metropolis 100} & \includegraphics[width=3.5cm]{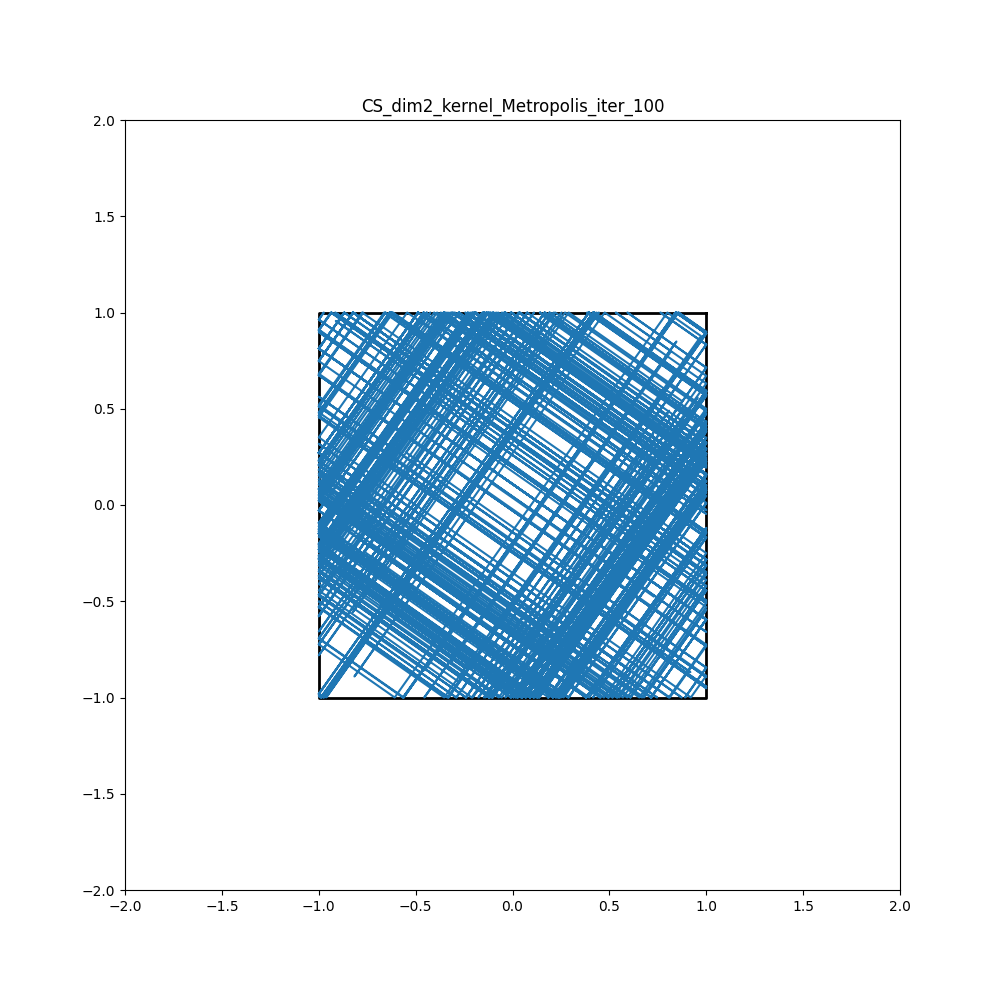}&
    \includegraphics[width=3.5cm]{CS_dim10_kernel_Metropolis_iter_1.png} & \includegraphics[width=3.5cm]{CS_dim100_kernel_Metropolis_iter_100.png} \\  
    \hline
    \end{tabular}
    \caption{Example trajectories for the Coordinate Sampler for simulating from a $d$-dimensional Gaussian distribution restricted to a cube, for $d=$2, 10, 100; and for different transitions on the boundary. For $d=10$, 100 we show the dynamics for the first two coordinates only. The different transitions correspond to the limiting behaviour, $Q_{\CS}$ of Section~\ref{sec:CSlim} (top); using a single Metropolis-Hastings step to sample from $l_x$ (middle); and using 100 Metropolis-Hastings steps to sample from $l_x$ (bottom).}
    \label{fig:trajectories-CS}
\end{figure}

\begin{figure}
    \centering
    \begin{tabular}{|c | c | c | c|}
    \hline
    & dim = 2 & dim = 10 & dim = 100 \\
    \hline
    \rotatebox{90}{ZZ Limit} & \includegraphics[width=3.5cm]{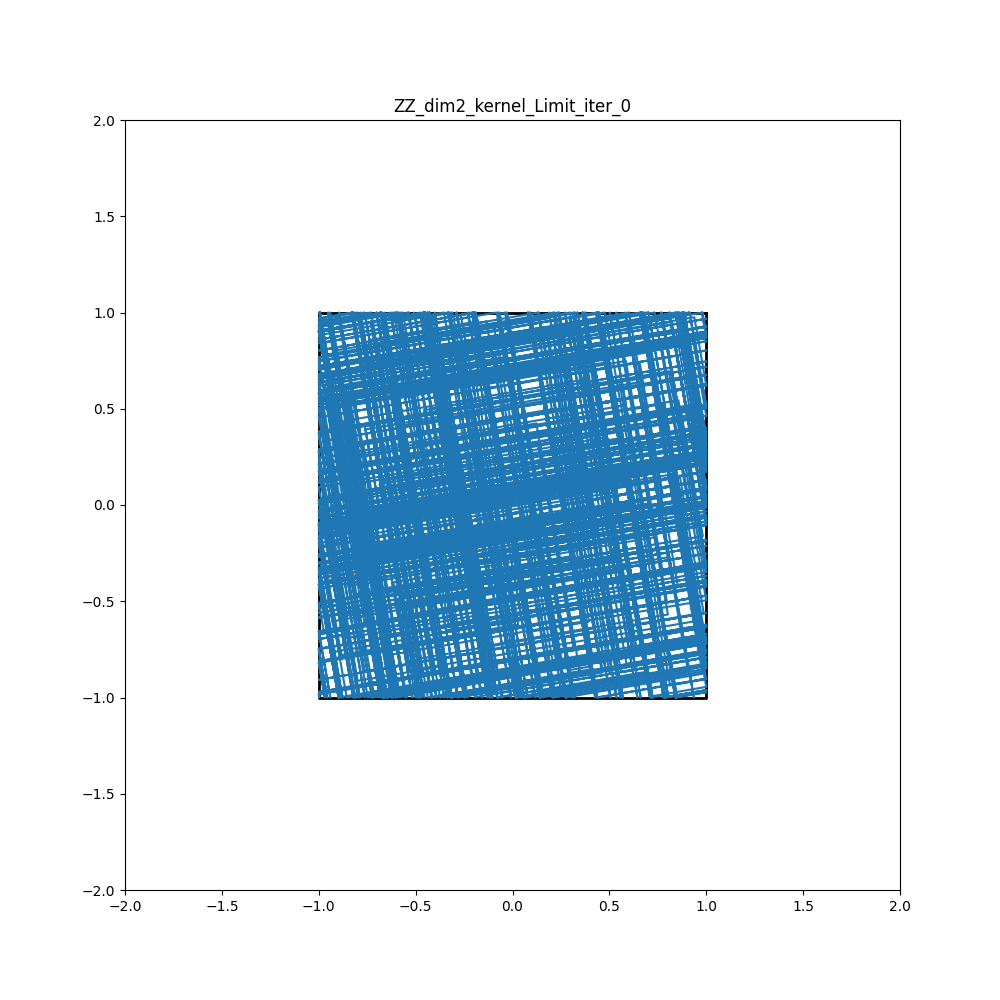}&
    \includegraphics[width=3.5cm]{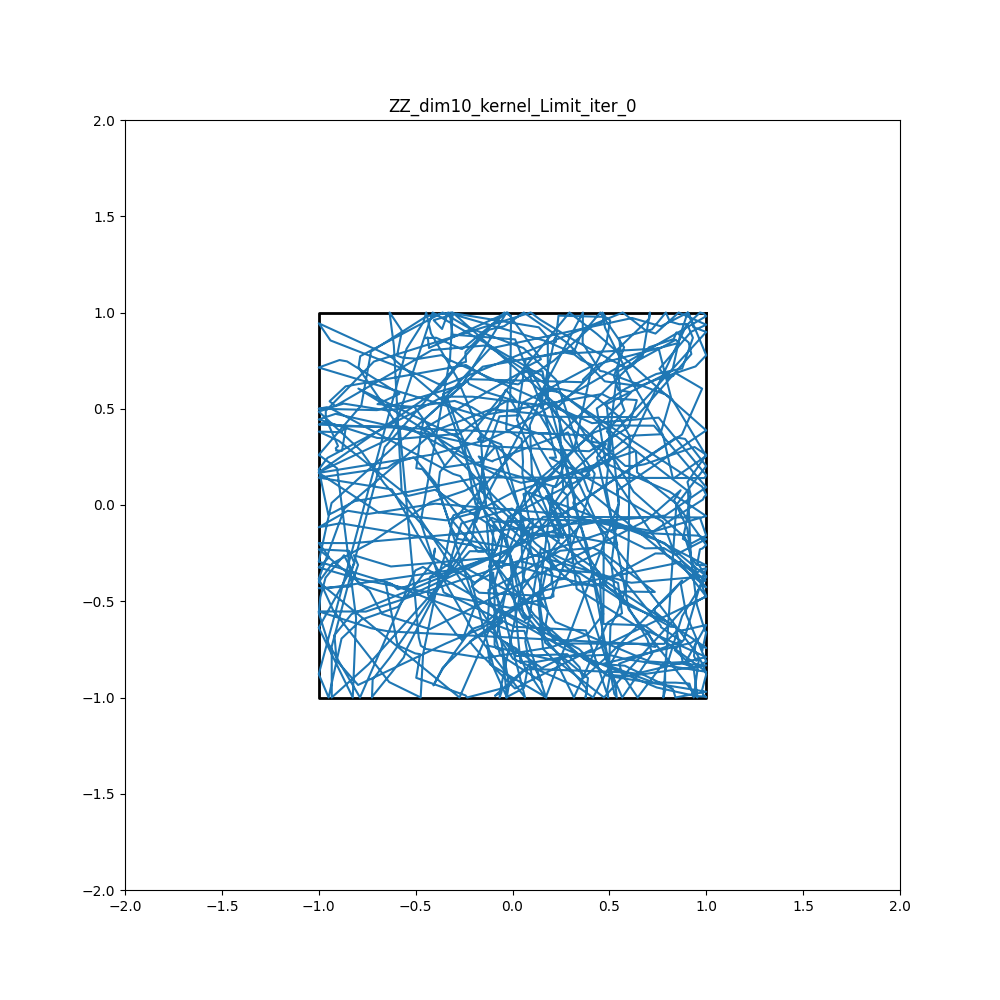} & \includegraphics[width=3.5cm]{ZZ_dim100_kernel_Limit_iter_0.png} \\
    \hline
    \rotatebox{90}{ZZ Metropolis 1} & \includegraphics[width=3.5cm]{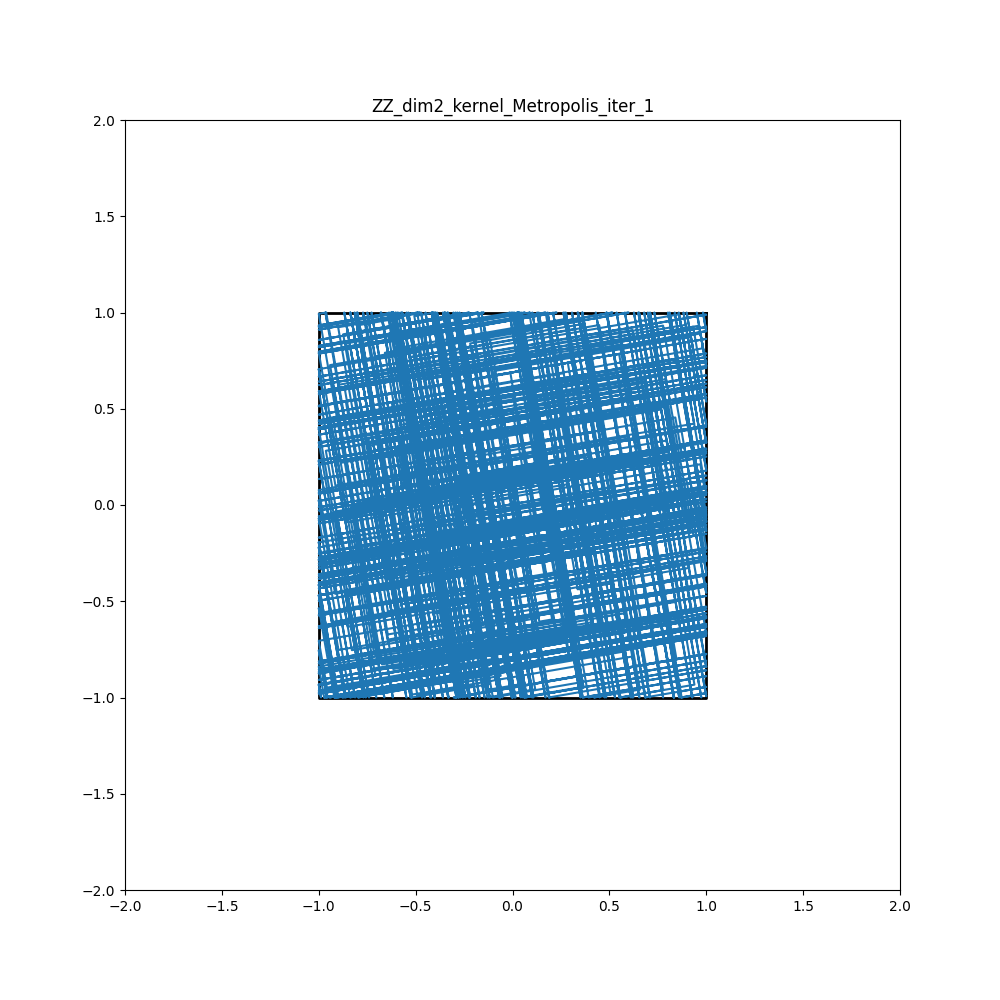}&
    \includegraphics[width=3.5cm]{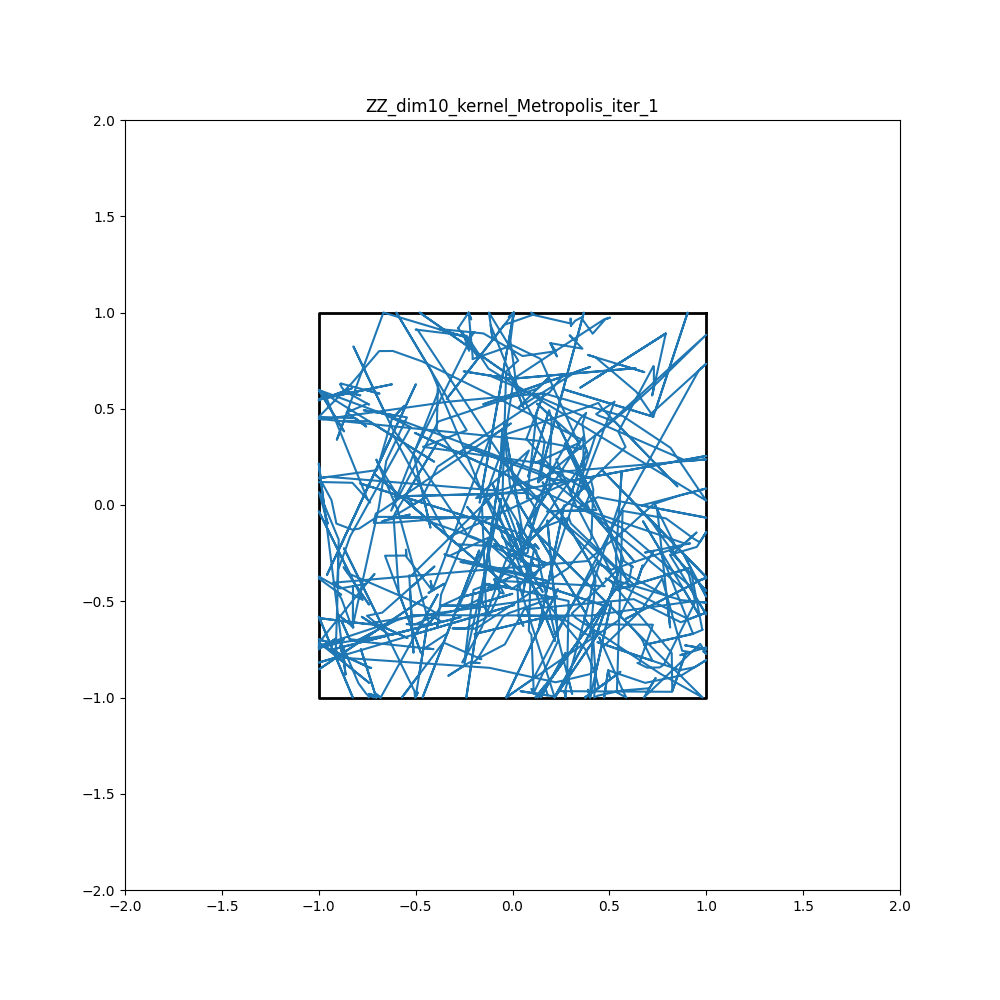} & \includegraphics[width=3.5cm]{ZZ_dim100_kernel_Metropolis_iter_1.png} \\
    \hline
    \rotatebox{90}{ZZ Metropolis 100} & \includegraphics[width=3.5cm]{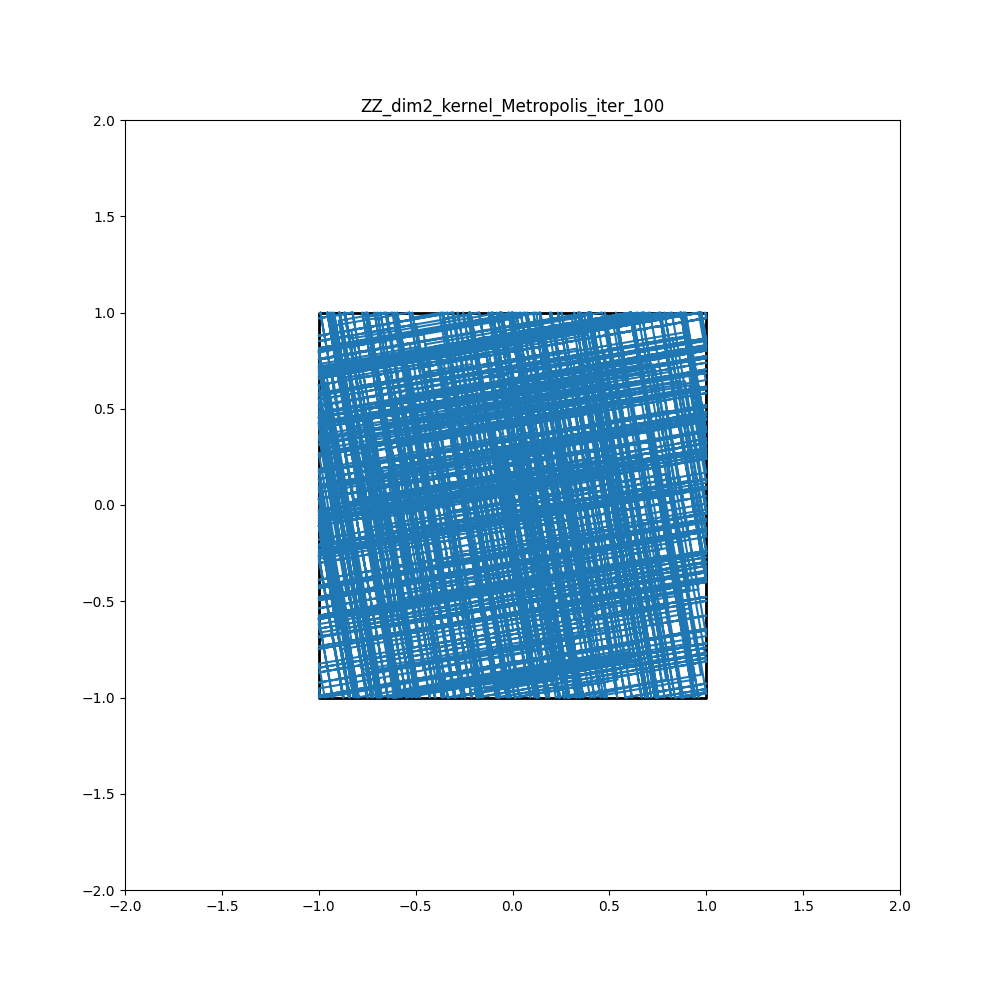}&
    \includegraphics[width=3.5cm]{ZZ_dim10_kernel_Metropolis_iter_1.png} & \includegraphics[width=3.5cm]{ZZ_dim100_kernel_Metropolis_iter_100.png} \\  
    \hline
    \end{tabular}
    \caption{Example trajectories for the Zig-Zag Sampler for simulating from a $d$-dimensional Gaussian distribution restricted to a cube, for $d=$2, 10, 100; and for different transitions on the boundary. For $d=10$, 100 we show the dynamics for the first two coordinates only. The different transitions correspond to the limiting behaviour, $Q_{\ZZ}$ of Section~\ref{sec:ZZlim} (top); using a single Metropolis-Hastings step to sample from $l_x$ (middle); and using 100 Metropolis-Hastings steps to sample from $l_x$ (bottom).}
    \label{fig:trajectories-ZZ}
\end{figure}

\begin{figure}
    \centering
    \begin{tabular}{|c | c | c | c|}
    \hline
    & dim = 2 & dim = 10 & dim = 100 \\
    \hline
    \rotatebox{90}{ZZ Limit} & \includegraphics[width=3.5cm]{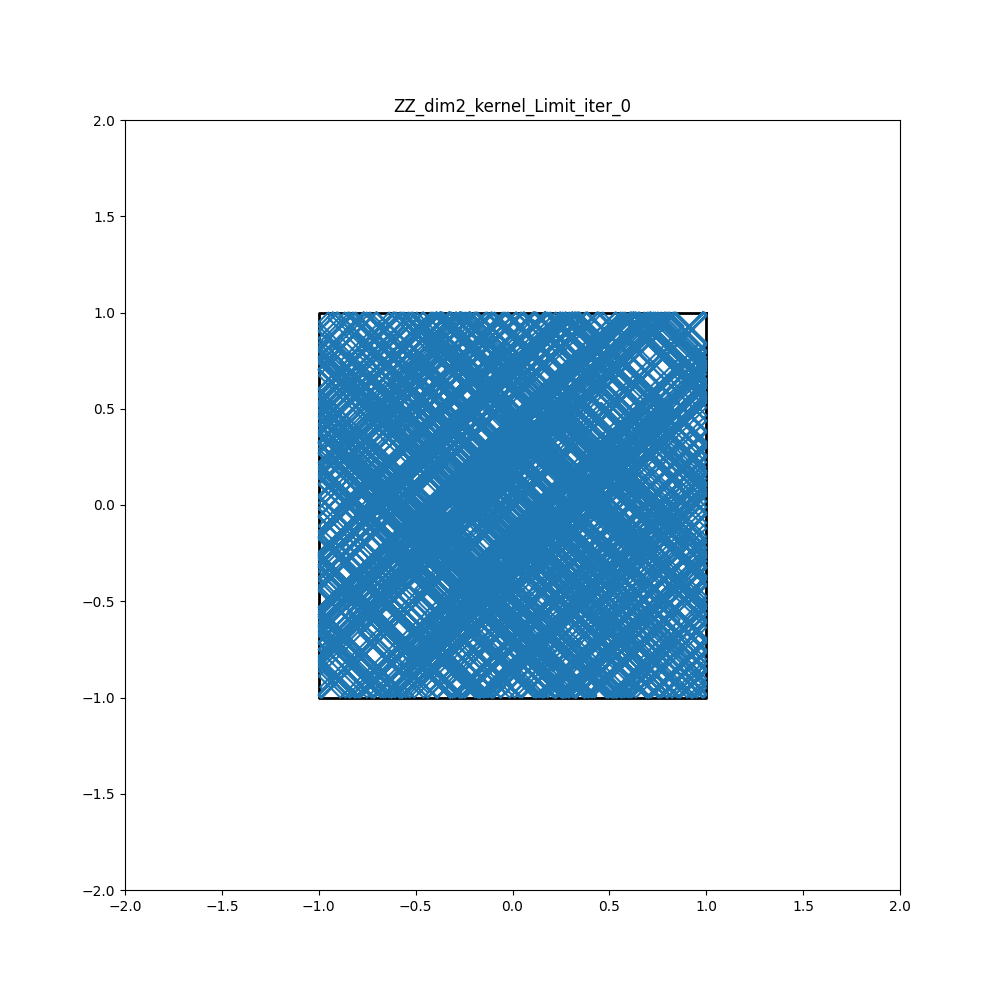}&
    \includegraphics[width=3.5cm]{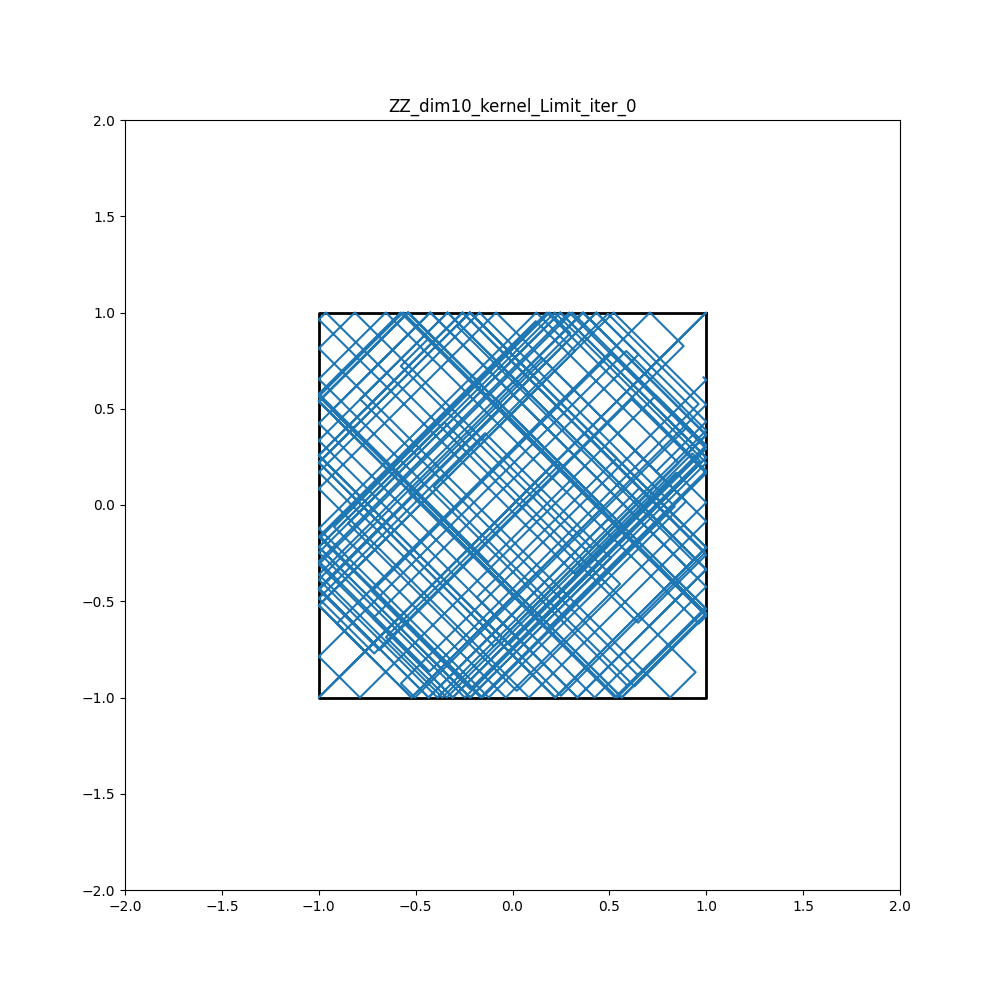} & \includegraphics[width=3.5cm]{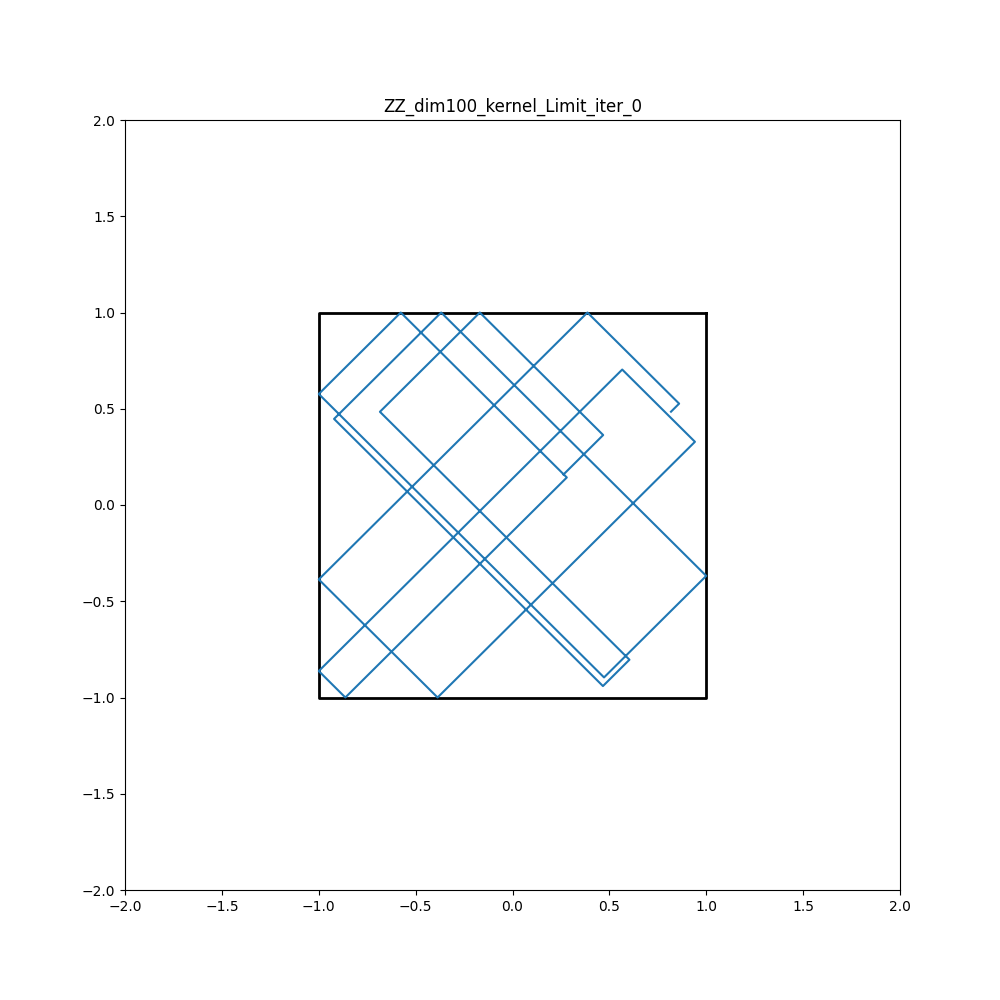} \\
    \hline
    \rotatebox{90}{ZZ Metropolis 1} & \includegraphics[width=3.5cm]{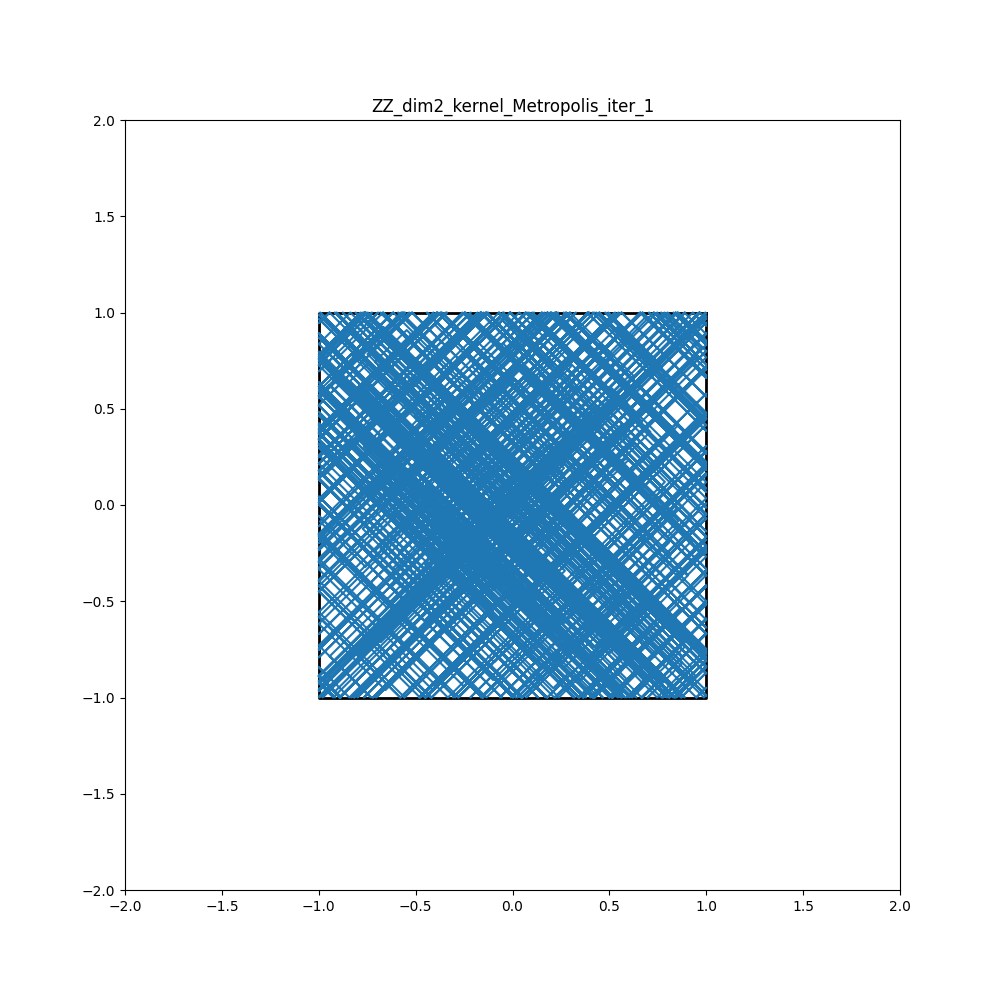}&
    \includegraphics[width=3.5cm]{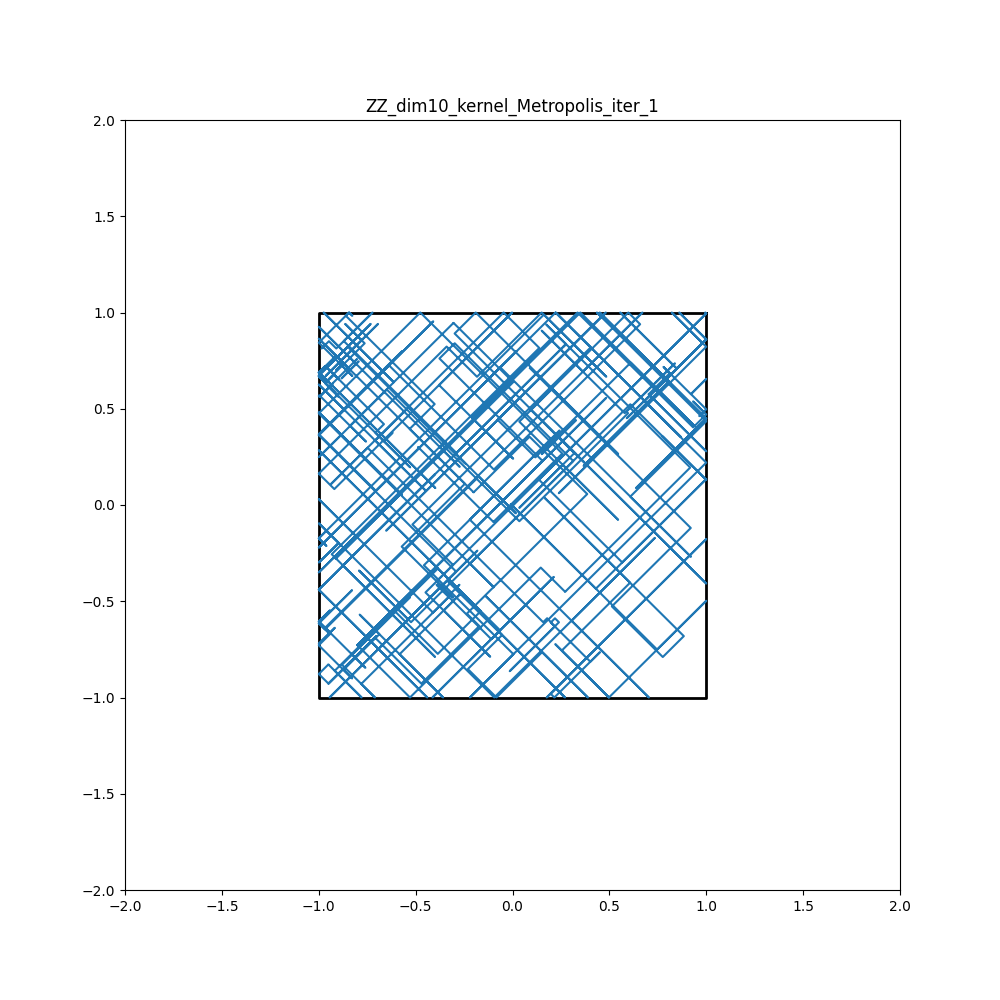} & \includegraphics[width=3.5cm]{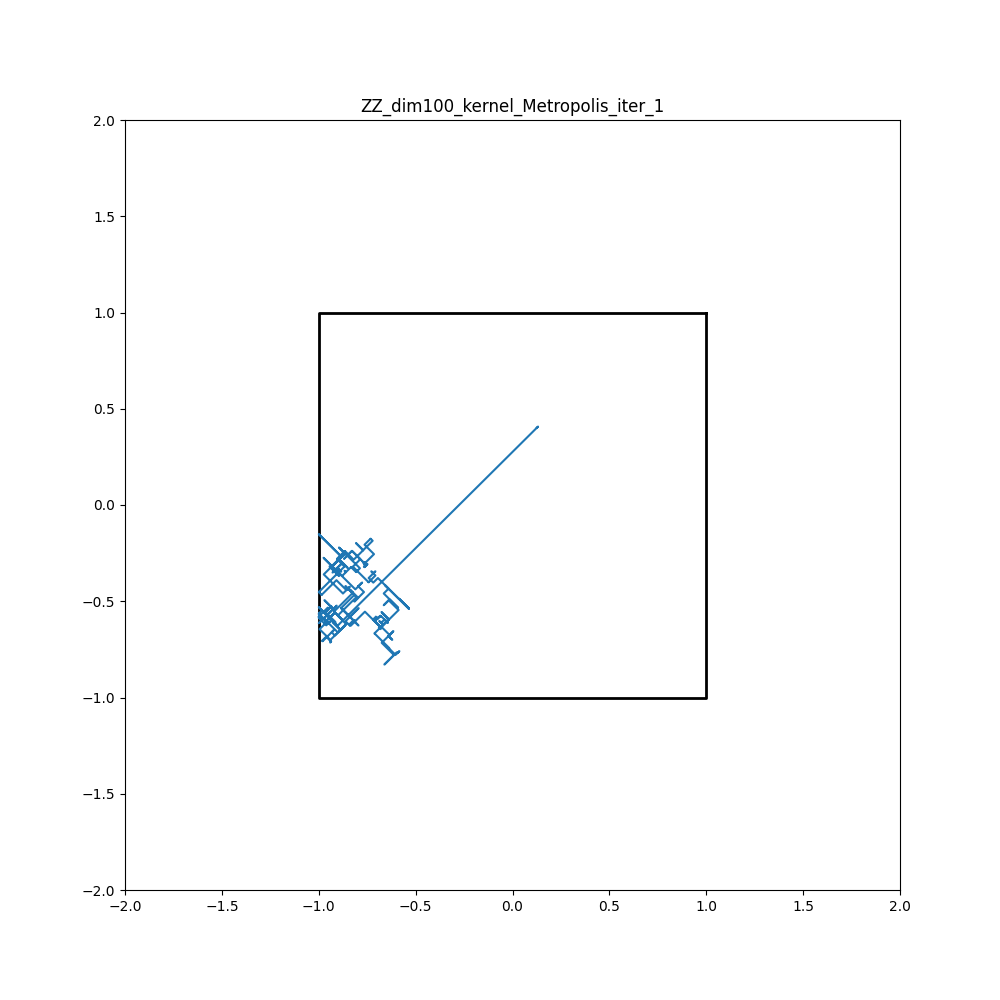} \\
    \hline
    \rotatebox{90}{ZZ Metropolis 100} & \includegraphics[width=3.5cm]{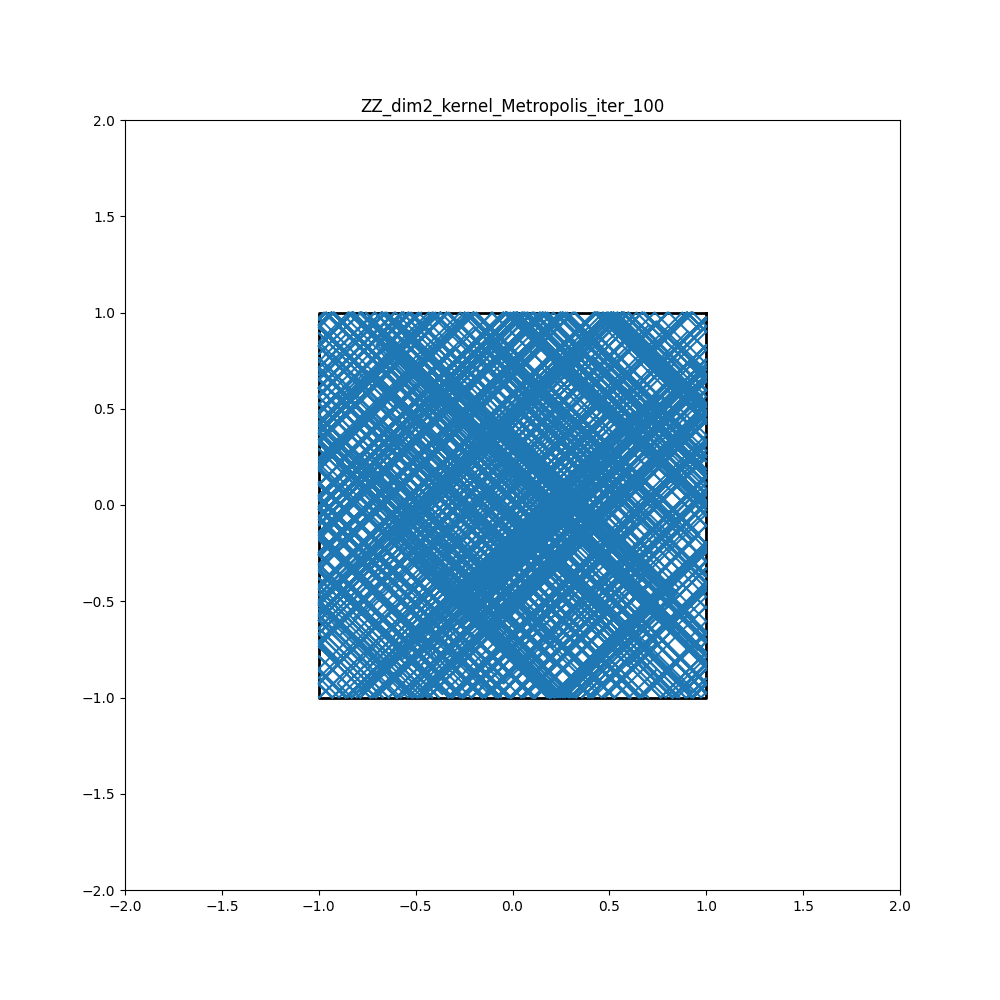}&
    \includegraphics[width=3.5cm]{pZZ_dim10_kernel_Metropolis_iter_1.png} & \includegraphics[width=3.5cm]{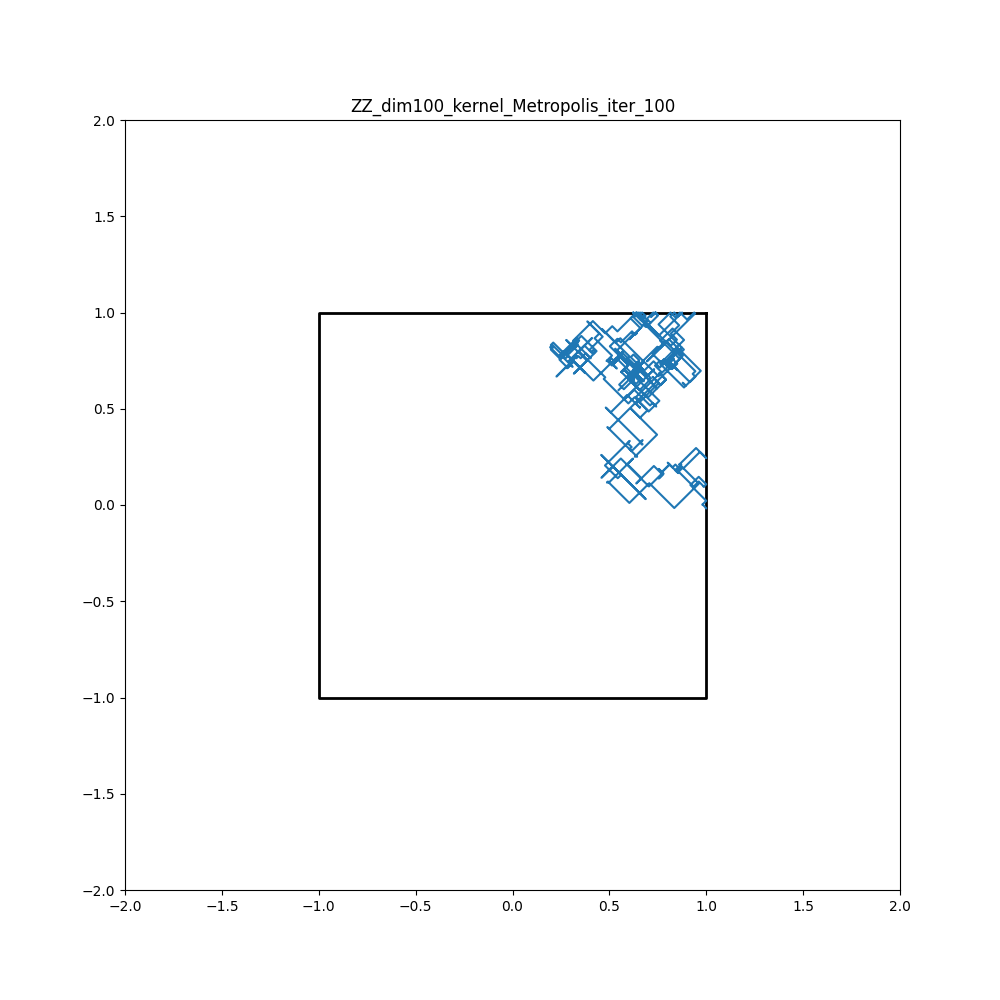} \\  
    \hline
    \end{tabular}
    \caption{Example trajectories for the Zig-Zag Sampler for simulating from a $d$-dimensional Gaussian distribution restricted to a cube, for $d=$2, 10, 100; and for different transitions on the boundary -- using the canonical basis. For $d=10$, 100 we show the dynamics for the first two coordinates only. The different transitions correspond to the limiting behaviour, $Q_{\ZZ}$ of Section~\ref{sec:ZZlim} (top); using a single Metropolis-Hastings step to sample from $l_x$ (middle); and using 100 Metropolis-Hastings steps to sample from $l_x$ (bottom).}
    \label{fig:trajectories-ZZ2}
\end{figure}
\end{document}